\documentclass[letterpaper,11pt,reqno]{article}

\usepackage[usenames,dvipsnames]{xcolor}
\usepackage{dynkin-diagrams}
\usepackage{amssymb,amsmath,amsthm,mathrsfs,xspace,multicol,stmaryrd,mathtools}
\usepackage[mathscr]{eucal}
\usepackage{amscd}
\usepackage[all]{xy}
\usepackage[shortlabels]{enumitem}
\usepackage{tikz}
\usepackage{tikz-3dplot}
\usepackage{tikz-cd}

\usepackage{xspace}

\usepackage[normalem]{ulem} 
\usepackage{framed}
\usepackage[textsize=scriptsize,textwidth=2cm]{todonotes}
\usepackage{tabularx}
\usepackage{changepage}
\usepackage{wasysym}
\usepackage[super]{nth}
\usepackage{faktor}

\newtheorem{innercustomgeneric}{\customgenericname}
\providecommand{\customgenericname}{}
\newcommand{\newcustomtheorem}[2]{%
  \newenvironment{#1}[1]
  {%
   \renewcommand\customgenericname{#2}%
   \renewcommand\theinnercustomgeneric{##1}%
   \innercustomgeneric
  }
  {\endinnercustomgeneric}
}

\newcustomtheorem{customthm}{Theorem}
\newcustomtheorem{customlemma}{Lemma}

\usepackage{currfile}
\usepackage{datetime}
\newdateformat{mydate}{{\THEDAY}  \monthname[\THEMONTH] \THEYEAR}
\mydate

\usepackage[breaklinks=true,colorlinks=true,citecolor=blue,urlcolor=MidnightBlue,linkcolor=black]{hyperref}

\usepackage[margin=1.5in]{geometry}

\usepgflibrary{arrows}
\usetikzlibrary{matrix}
\usetikzlibrary{calc}
\usetikzlibrary{positioning}
\usetikzlibrary{backgrounds}

\tikzset{%
  operad style/.style={fill=white,inner sep=2.5pt},
  baseline=(current  bounding  box.center),
  ampersand replacement=\&, row sep=2em,column sep=2em,
  std/.style={->,font=\scriptsize}
}

\tikzset{%
  clr style/.style={fill=white,inner sep=2.5pt},
  baseline=(current  bounding  box.center),
  ampersand replacement=\&, row sep=2em,column sep=2em,
  std/.style={->,font=\scriptsize}
}

\usetikzlibrary{decorations.pathreplacing,decorations.markings}
\usepgflibrary{shapes.geometric}
\usetikzlibrary{shapes.geometric}

\tikzset{middlearrow/.style={
        decoration={markings,
            mark= at position 0.5 with {\arrow{#1}} ,
        },
        postaction={decorate}
    }
}

\newenvironment{tikzdiag}[2]
{
\begin{tikzpicture}[clr style]
\matrix (m) [matrix of math nodes, row sep=#1em, column sep=#2em]
}
{
\end{tikzpicture}
}

\let\Sec\S

\def\release#1{\let#1\undefined}

\usepackage[utf8]{inputenc}
\DeclareFontFamily{U}{min}{}
\DeclareFontShape{U}{min}{m}{n}{<-> udmj30}{}
\newcommand\yo{\!\text{\usefont{U}{min}{m}{n}\symbol{'207}}\!}

\newcommand{\myindent}{\hspace{.5cm}}
\newcommand{\mind}{\myindent}

\newcommand{\myspace}{\vspace{.25cm}}

\newcommand{\df}[1]{\textbf{\textit{#1}}}

\newlist{pfitems}{itemize}{1}
\setlist[pfitems]{label=--}

\newcounter{grtotal}
\newcounter{grmax}
\makeatletter
\providecommand{\leftsquigarrow}{%
  \mathrel{\mathpalette\reflect@squig\relax}%
}
\newcommand{\reflect@squig}[2]{%
  \reflectbox{$\m@th#1\rightsquigarrow$}%
}
\makeatother

\newcommand{\N}{\mathbb{N}}
\newcommand{\R}{\mathbb{R}}

\newcommand{\Z}{\mathbb{Z}}

\newcommand{\kk}{\Bbbk}

\newcommand{\g}{\mathfrak{g}}

\newcommand{\gl}{\mathfrak{gl}}

\newcommand{\brac}{[ \cdot, \cdot ]}
\newcommand{\bbrac}[2]{\left \llbracket #1,#2 \right \rrbracket}

\newcommand{\CE}{\mathrm{CE}}

\DeclareMathOperator{\GL}{\mathrm{GL}}

\DeclareMathOperator{\Ho}{\mathrm{Ho}}

\DeclareMathOperator{\Hom}{\mathrm{Hom}}

\DeclareMathOperator{\id}{\mathrm{id}}

\DeclareMathOperator{\ft}{\mathrm{fin}} 

\newcommand{\op}{\mathrm{op}}

\newcommand{\cat}[1]{\mathsf{#1}}
\newcommand{\catC}{\cat{C}}
\newcommand{\Gpd}{\cat{Gpd}}

\newcommand{\Grp}{\cat{Grp}}

\newcommand{\LnA}[1]{\mathsf{Lie}_{#1}\mathsf{Alg}}

\newcommand{\dgcocom}{\mathsf{dgCoCom}_{\geq 0}}

\newcommand{\Ch}{\mathsf{Ch}}
\newcommand{\Chain}{\mathsf{Ch}_{\geq 0}}

\newcommand{\Mfd}{\mathsf{Mfd}}

\newcommand{\sSet}{s\mathsf{Set}}

\DeclareMathOperator{\cdga}{\mathsf{cdga}}
\DeclareMathOperator{\Vect}{\mathsf{Vect}}

\newcommand{\maps}{\colon}

\newcommand{\xto}[1]{\xrightarrow{#1}}

\newcommand{\Eq}{\rightrightarrows}

\newcommand{\fib}{\twoheadrightarrow}
\newcommand{\weq}{\xrightarrow{\sim}}

\newcommand{\emb}{\hookrightarrow}

\newcommand{\adj}{\dashv}
\newcommand{\adjunct}{\rightleftarrows}

\newcommand{\al}{\alpha}
\newcommand{\be}{\beta}

\newcommand{\ga}{\gamma}
\newcommand{\Del}{\Delta}
\newcommand{\del}{\delta}

\newcommand{\io}{\iota}
\newcommand{\lam}{\lambda}
\newcommand{\Lam}{\Lambda}
\newcommand{\ro}{\rho}

\newcommand{\tha}{\theta}
\newcommand{\vtha}{\vartheta}
\newcommand{\ph}{\phi}

\newcommand{\vph}{\varphi}
\newcommand{\vphi}{\varphi}
\newcommand{\eps}{\epsilon}

\newcommand{\el}{\ell}
\newcommand{\si}{\sigma}

\newcommand{\jm}{\jmath}

\newcommand{\bl}{\bullet}

\newcommand{\ba}[1]{\bar{#1}}
\newcommand{\ti}[1]{\tilde{#1}}

\newcommand{\wh}[1]{\widehat{#1}}
\newcommand{\wti}[1]{\widetilde{#1}}
\newcommand{\und}[1]{{\underline{#1}}}
\newcommand{\ov}[1]{{\overline{#1}}}
\newcommand{\bul}{\bullet}

\newcommand{\cU}{\mathcal{U}}

\newcommand{\cG}{\mathcal{G}}
\newcommand{\cC}{\mathcal{C}}

\newcommand{\cc}{\circ}
\newcommand{\iso}{\cong}

\newcommand{\sse}{\subseteq}

\newcommand{\abs}[1]{\left \lvert #1 \right \rvert}

\newcommand{\st}{~ \left. \right \vert ~} 

\DeclareMathOperator{\diag}{\mathrm{diag}}

\newcommand{\gen}[1]{\bigl \langle #1 \bigr\rangle}

\newcommand{\spann}{\mathrm{span}}
\newcommand{\tensor}{\otimes}

\newcommand{\dsum}{\oplus}

\newcommand{\ev}{\mathrm{ev}}
\newcommand{\pr}{\mathrm{pr}}

\DeclareMathOperator{\im}{\mathrm{im}}

\DeclareMathOperator{\Cobar}{\mathrm{Cobar}}

\renewcommand{\deg}[1]{\left \lvert #1 \right \rvert}
\newcommand{\sgn}[1]{ (-1)^{\frac{#1(#1-1)}{2}}}
\newcommand{\vv}{\vee}

\newcommand{\Sh}{\mathrm{Sh}}

\renewcommand{\S}{\bar{S}}

\newcommand{\pa}{\partial}
                 
\newcommand{\bs}{\mathbf{s}}

\newcommand{\ideal}{\trianglelefteq}

\newcommand{\semiop}{\ltimes}

\DeclareMathOperator{\Path}{\mathrm{Path}}

\newcommand{\one}[1]{\mathbf{1}_{#1}}

\DeclareMathOperator{\sk}{sk}

\newcommand{\plim}{\varprojlim}

\renewcommand{\O}{\mathcal{O}}

\newcommand{\mm}{\mathfrak{m}}

\renewcommand{\vv}[1]{\vec{#1}}

\newcommand{\cO}{\mathcal{O}}

\DeclareFontFamily{U}{mathx}{}
\DeclareFontShape{U}{mathx}{m}{n}{<-> mathx10}{}
\DeclareSymbolFont{mathx}{U}{mathx}{m}{n}
\DeclareMathAccent{\widecheck}{0}{mathx}{"71}

\theoremstyle{plain}
\newtheorem{theorem}{Theorem}[section]
\newtheorem*{theorem*}{Theorem}
\newtheorem*{theorem1}{Theorem 1}

\newtheorem{proposition}[theorem]{Proposition}
\newtheorem*{proposition*}{Proposition}
\newtheorem{lemma}[theorem]{Lemma}
\newtheorem*{lemma*}{Lemma}
\newtheorem{corollary}[theorem]{Corollary}

\theoremstyle{definition}
\newtheorem{definition}[theorem]{Definition}

\newtheorem*{ass*}{Assumption}
\newtheorem*{question}{Question}

\newtheorem{construction}[theorem]{Construction}
\newtheorem{notation}[theorem]{Notation}
\newtheorem*{notation*}{Notation}

\newtheorem{claim}{Claim}
\newtheorem{lemclaim}[theorem]{Claim}

\newtheorem{remark}[theorem]{Remark}
\newtheorem{example}[theorem]{Example}

\numberwithin{equation}{section}

\newcommand{\cVect}{c\cat{Vect}}
\newcommand{\cpVect}{c^+\cat{Vect}}
\newcommand{\spFDVect}{s^+\mathsf{Vect}^{\fd}}

\newcommand{\Com}{\mathrm{alg}}
\newcommand{\Compp}{\mathrm{alg}^{+}}
\newcommand{\cCom}{c\cat{Com}}
\newcommand{\cpCom}{c^{+}\cat{Com}}

\newcommand{\crCom}{c\cat{Com}^{\mathrm{fg}}_{0}}

\newcommand{\crComfg}{c\cat{Com}^{(\mathrm{fg})}_{0}}
\newcommand{\crpComfg}{c^+\cat{Com}^{(\mathrm{fg})}_{0}}
\newcommand{\cprCom}{c^+\cat{Com}^{\mathrm{fg}}_{0}}
\newcommand{\cCh}{\Ch^{\ast \geq 0}}
\newcommand{\rcdgafg}{\cdga^{(\mathrm{fg})}_{0}}
\newcommand{\rcgafg}{\cat{cga}^{(\mathrm{fg})}_{0}}
\newcommand{\rcdga}{\cdga^{\mathrm{fg}}_{0}}
\newcommand{\rcga}{\cat{cga}^{\mathrm{fg}}_{0}}

\newcommand{\EM}{\mathrm{EM}}

\newcommand{\Nas}{N^{\ast}}
\newcommand{\Kb}{K^{\bl}}

\newcommand{\sVect}{s\mathsf{Vect}}

\newcommand{\cD}{\mathcal{D}}

\newcommand{\cI}{\mathcal{I}}

\DeclareMathOperator{\Sym}{Sym}
\DeclareMathOperator{\Diff}{Diff}
\DeclareMathOperator{\FDiff}{FmlDiff}

\newcommand{\gVect}{\cat{grVect}}
\newcommand{\cga}{\cat{cga}}
\newcommand{\LinfGpd}{\mathsf{Lie_{\infty}Gpd}^{\fd}}
\newcommand{\LinfGrp}{\mathsf{Lie_{\infty}Grp}^{\fd}}

\renewcommand{\Sym}{\mathrm{Sym}}
\newcommand{\dgSym}{{\Sym}^{\ast}}
\newcommand{\gSym}{{\Sym}^{\ast}_{+}}
\newcommand{\cSym}{{\Sym}^{\bl}}
\newcommand{\cpSym}{{\Sym}^{\bl}_{+}}
\newcommand{\cphSym}{{\wh{\Sym}}^{\bl}_{+}}

\newcommand{\coSym}{\mathrm{Sym}^\coalg}
\newcommand{\cosSym}{\mathrm{Sym}^{\coalg}_{\bl}}
\newcommand{\cospSym}{\mathrm{Sym}^{\coalg\,+}_{\bl}}
\newcommand{\codgSym}{\mathrm{Sym}^{\coalg}_{\ast}}
\newcommand{\cogrSym}{\mathrm{Sym}^{\coalg \, +}_{\ast}}

\newcommand{\inhom}{\mathrm{inh}}

\newcommand{\Vb}{V^\bl}
\newcommand{\VVb}{W^\bl}

\newcommand{\cX}{\mathscr{X}}
\newcommand{\cY}{\mathscr{Y}}

\newcommand{\cont}{\mathrm{cont}}

\newcommand{\ccot}[1]{\cot^{\ast}_{#1}}

\newcommand{\linf}{\LnA{\infty}^{\fft}}
\release{\sgn}
\DeclareMathOperator{\sgn}{sgn}

\newcommand{\sLie}{s\cat{Lie}}

\newcommand{\PP}{(+)}
\newcommand{\cqga}{\cat{c(d)ga}}
\newcommand{\Comp}{\cat{Com^{+}}}

\newcommand{\coeq}{\mathrm{coeq}}
\newcommand{\Da}{D^\ast}
\newcommand{\Dap}{D^\ast_{+}}
\newcommand{\DPP}{D^\ast_{\PP}}

\usepackage{xparse}
\usepackage{tikz}
\usetikzlibrary{matrix,arrows.meta}
\usepackage{xparse}

\NewDocumentCommand{\AdjRLh}{ O{above} O{below} m m m m }{%
  \draw[->,>=Stealth,transform canvas={yshift=1.95pt}]
    (#3) -- node[#1] {$#5$} (#4);
  \draw[<-,>=Stealth,transform canvas={yshift=-1.95pt}]
    (#3) -- node[#2] {$#6$} (#4);
}

\NewDocumentCommand{\AdjRLv}{ O{left} O{right} m m m m }{%
  \draw[->,>=Stealth,transform canvas={xshift=-1.95pt}]
    (#3) -- node[#1] {$#5$} (#4);
  \draw[<-,>=Stealth,transform canvas={xshift=1.95pt}]
    (#3) -- node[#2] {$#6$} (#4);
}

\definecolor{mygray}{gray}{0.4}
\newcommand{\tc}[1]{\textcolor{mygray}{#1}}

\DeclareMathOperator*{\bightensor}{\wh{\bigotimes}}
\DeclareMathOperator{\sig}{sig}

\newcommand{\fft}{\mathrm{ft}}
\newcommand{\cH}{\mathscr{H}}
\newcommand{\cK}{\mathscr{K}}
\newcommand{\DK}{\mathrm{DK}}

\newcommand{\htensor}{\wh{\tensor}}
\newcommand{\hA}{\wh{A}}

\newcommand{\clnAlg}{\wh{\cat{Alg}}^{\rm fg,loc}_{/ \kk}}
\newcommand{\rclnAlg}{\wh{\cat{Alg}}^{\rm fg,loc}_{/ \kk , \, 0}}
\newcommand{\cclnAlg}{c\wh{\cat{Alg}}^{\rm fg,loc}_{/ \kk}}
\newcommand{\cpclnAlg}{c^+\wh{\cat{Alg}}^{\rm fg,loc}_{/ \kk}}

\newcommand{\cast}{\circledast}

\newcommand{\cW}{\mathscr{W}}
\newcommand{\fG}{\mathscr{G}}
\newcommand{\Wb}{\ov{\cW}}
\newcommand{\Wbin}{\ov{\cW}^{\inhom}}
\newcommand{\Wbar}[1]{\ov{\mathscr{W}}_{\! \bl}\, #1}
\newcommand{\Wbarin}[1]{\ov{\mathscr{W}}^{\inhom}_{\bl} #1}
\renewcommand{\cU}{\mathcal{U}}

\newcommand{\scocom}{s\cat{CoCom}}

\newcommand{\pbw}{\mathrm{pbw}}

\newcommand{\gb}{\g_{\bl}}
\newcommand{\Ng}{N_\ast \g_\bl}
\newcommand{\Nag}{N_\ast \g}
\newcommand{\Kcoalg}{K^{\coalg}}
\newcommand{\Kcoalgp}{K^{\coalg, +}}
\newcommand{\dE}{d^{E\cU}}
\newcommand{\tvp}{\ti{\vph}}

\newcommand{\spf}{\mathrm{Spf}}
\newcommand{\dgspf}{\mathrm{Spf_\ast}}
\newcommand{\gspf}{\mathrm{Spf^{+}_{\ast}}}

\renewcommand{\cC}{\catC}

\newcommand{\catGrp}{\cat{Grp}(s\cC)}

\newcommand{\gcocom}{\cat{grCoCom}_{\geq 0}}
\newcommand{\rgcocom}{\cat{grCoCom}^{{\fft}}_{0}}
\newcommand{\rdgcocom}{\cat{dgCoCom^{\fft}_{0}}}
\renewcommand{\dgcocom}{\cat{dgCoCom_{\geq 0}}}
\newcommand{\scoccom}{s\cat{CoCom}_{\mathrm{conil}}}
\newcommand{\conil}{\cat{CoCom}^{\fft}_{\mathrm{conil}}}
\newcommand{\rconil}{\cat{CoCom}^{\fft}_{\mathrm{conil,\, 0}}}
\newcommand{\conilsm}{\cat{FmlMfd_\ast}}
\newcommand{\sconil}{s\conil}
\newcommand{\spconil}{s^+\conil}
\newcommand{\prim}{\mathrm{Prim}}
\newcommand{\coalg}{\mathrm{co}}
\newcommand{\fd}{\mathrm{fd}}
\renewcommand{\cD}{\cat{D}}
\newcommand{\cS}{\cat{S}}
\newcommand{\fs}{\mathit{fs}}
\renewcommand{\Mfd}{\cat{Mfd}^{\fd}}
\newcommand{\spVect}{s^{+}\Vect}

\newcommand{\gFDVect}{\gVect^{\fd}}
\newcommand{\FDVect}{\Vect^{\fd}}
\newcommand{\sFDVect}{s\mathsf{Vect}^{\fd}}
\newcommand{\FG}{\cat{FmlGrp_\infty}}
\newcommand{\FGpd}{\cat{FmlGrpd^{\mathrm{pt}}_\infty}}
\newcommand{\Dist}{\mathrm{Dist}}

\newcommand{\kan}{\mathrm{Kan}}
\newcommand{\acyc}{\mathrm{Acyc}}
\newcommand{\mix}{\mathit{mix}}
\newcommand{\SC}{\cS_{\cC}}
\newcommand{\SD}{\cS_{\cD}}
\newcommand{\grC}{\cat{grC}}

\newcommand{\sur}{\mathscr{D}}
\renewcommand{\df}[1]{{\bf #1}}
\newcommand{\Brac}{\bbrac{\cdot}{\cdot}}

\newcommand{\infcatW}[2]{\mathbf{#1}[{#2}^{-1}]}
\newcommand{\infcat}[1]{\mathbf{#1}}

\newcommand{\const}{\mathrm{const}}

\usepackage{float}

\usepackage{scalerel}

\makeatletter
\newcommand*\mdot{\mathpalette\mdot@{.5}}
\newcommand*\mdot@[2]{\mathbin{\vcenter{\hbox{\scalebox{#2}{$\m@th#1\bullet$}}}}}
\makeatother

\usepackage[nottoc]{tocbibind}

\usepackage{titlesec}

\titleclass{\subsubsubsection}{straight}[\subsection]

\newcounter{subsubsubsection}[subsubsection]
\renewcommand\thesubsubsubsection{\thesubsubsection.\arabic{subsubsubsection}}

\titleformat{\subsubsubsection}
  {\normalfont\normalsize\bfseries}{\thesubsubsubsection}{1em}{}
\titlespacing*{\subsubsubsection}
{0pt}{3.25ex plus 1ex minus .2ex}{1.5ex plus .2ex}

\makeatletter
\renewcommand\paragraph{\@startsection{paragraph}{5}{\z@}%
  {3.25ex \@plus1ex \@minus.2ex}%
  {-1em}%
  {\normalfont\normalsize\bfseries}}
\renewcommand\subparagraph{\@startsection{subparagraph}{6}{\parindent}%
  {3.25ex \@plus1ex \@minus .2ex}%
  {-1em}%
  {\normalfont\normalsize\bfseries}}
\def\toclevel@subsubsubsection{4}
\def\toclevel@paragraph{5}
\def\toclevel@subparagraph{6}
\def\l@subsubsubsection{\@dottedtocline{4}{7em}{4em}}
\def\l@paragraph{\@dottedtocline{5}{10em}{5em}}
\def\l@subparagraph{\@dottedtocline{6}{14em}{6em}}
\makeatother

\title{Higher differentiation via higher formal groupoids}
\author{Christopher L.\ Rogers \thanks{Department of Mathematics \& Statistics, University of Nevada,  Reno. \url{chrisrogers@unr.edu}}
}
\date{\today}

\begin{document}

\maketitle

\begin{abstract}
We solve the differentiation problem for Lie $\infty$-groups. 
Our approach builds on a classical version of Cartier duality which canonically 
identifies the Hopf algebra of point distributions supported at the identity of a Lie group with the universal enveloping algebra of its Lie algebra. Hence, for Lie $\infty$-groups, we consider simplicial coalgebras of point distributions. To do this properly, we first develop the homotopy theory of pointed formal $\infty$-groupoids within K.\ Behrend  and E.\ Getzler's framework for higher geometric stacks. These objects are ``higher'' but  not ``derived'', which is an important distinction for the geometric applications in mind. The second part of our construction relies on a careful analysis of J.\ Pridham's variation of the Dold-Kan adjunction for cosimplicial algebras. Our main result is a differentiation functor at the level of 1-categories from finite-dimensional Lie $\infty$-groups to finite-type Lie $\infty$-algebras that is homotopically well-behaved, coordinate-free, and explicit yet tractable. In particular, if $G_\bl$ is a simplicial Lie group with Lie algebra $\g_\bl$, we prove that the differentiation of its classifying space $\Wbar{G}$ is canonically isomorphic 
to the dg Lie algebra of normalized chains $N_\ast(\g_\bl)$.  

\end{abstract}

\newpage

\tableofcontents

\section{Introduction}
Lie $\infty$-groups are reduced simplicial manifolds that satisfy a smooth geometric analog of the horn filling conditions for Kan simplicial sets. Classical examples include the nerve of any Lie group $N_{\bl}G$, and, more generally, the classifying space $\Wbar{G}$ of a simplicial Lie group $G_\bl$. In general, Lie $\infty$-groups can be understood as the ``higher'' or ``stacky'' analogue of Lie groups. More interesting examples were introduced by A.\ Henriques in  \cite{Hen} within the context of simplicial Banach manifolds in order to geometrically model the full Whitehead tower of a Lie group. However, in this paper, we only consider Lie $\infty$-groups that are level-wise finite-dimensional. This is, of course, the ideal situation geometrically, and we lose nothing by doing so. Indeed, by our recent joint work with J.\ Wolfson \cite{Lie3}, all known examples of interest, including those considered in \cite{Hen}, are weakly equivalent to such finite-dimensional Lie $\infty$-groups. Let us also emphasize right away that all of the homotopical machinery considered in the present paper, including the aforementioned notion of weak equivalence, lives at the 1-categorical level.       
No $\infty$-category theory is required, although our constructions immediately pass to the correct underlying $\infty$-categories, if one wishes to do so.

The infinitesimal analogues of finite-dimensional Lie $\infty$-groups are finite-type Lie $\infty$-algebras, i.e.\ degree-wise finite-dimensional, homologically and non-negatively graded $L_\infty$-algebras. They are the source of the above examples, and the main result of \cite{Lie3} can be characterized as a finite-dimensional integration procedure, i.e., Lie's 3rd Theorem for Lie $\infty$-algebras. The focus of this paper is the reverse process: differentiation.

A long-standing unfulfilled desideratum in ``higher geometry'' is a lift of the classical differentiation procedure for Lie groups to the world of Lie $\infty$-groups which satisfies the following criteria:  
\begin{itemize}
\item {\bf functoriality} at the 1-categorical level, from finite-dimensional Lie $\infty$-groups to finite-type Lie $\infty$-algebras;

\item {\bf exactness}, i.e.\  maximal homotopical compatibility between the respective homotopy theories for Lie $\infty$-groups and Lie $\infty$-algebras;

\item {\bf canonical identifications} with classical examples, i.e.\ a natural isomorphism between the differentiation of $\Wbar{G}$ of a simplicial Lie group $G_\bl$, and the dg Lie algebra of normalized chains on the simplicial Lie algebra of $G_\bl$; 

\item {\bf tractability}, i.e.\ in the general case, the underlying complex of the differentiation should coincide with the shifted tangent complex of the Lie $\infty$-group, and the higher $L_\infty$ brackets of the differentiation should be computationally accessible.
\end{itemize}
Partial results along these lines were given in 2006 by P.\ \v{S}evera \cite{Sevu} (see also \cite[\Sec 8.3]{Li}) using the language of supermanifolds.
But further progress in this direction appears to have stalled, and nothing resembling a general theory has been proposed using this framework.

In this paper, we solve the above differentiation problem for Lie $\infty$-groups. In contrast with \cite{Sevu}, our approach stems from the following classical result attributed to P.\ Cartier\footnote{See also \cite[\Sec 3.7]{Cartier} and references within.} and proved in detail by J.-P.\ Serre in \cite[\Sec 6]{Serre:Lie}: If $G$ is a Lie group with Lie algebra $\g$, then there is a canonical isomorphism of Hopf algebras
\[
\cU(\g) \xto{\cong} \Dist(G),
\] 
where $\cU(\g)$ is the universal enveloping algebra, and $\Dist(G)$ is the Hopf algebra of distributions on $G$ supported at the identity.
Hence, $\g$ is recovered canonically as the Lie algebra of primitives of $\Dist(G)$. More generally, one can functorially assign to any pointed manifold the cocommutative coalgebra of distributions supported at the basepoint. Applying this level-wise to a Lie $\infty$-group $\cG_\bl$ produces a reduced Kan simplicial object $\Dist_\bl(\cG)$ in the category conilpotent cocommutative coalgebras. In other words, $\Dist_\bl(\cG)$ is the formal $\infty$-group associated to $\cG_\bl$.     

Composing $\Dist_{\bl}(-)$ with the linear dual of J.\ Pridham's variation of the Dold-Kan adjunction for cosimplicial algebras
\cite{JP:Def, JP:Mixed,JP:Poisson},
we obtain a functor
\[
\Diff(-) \maps \LinfGrp \to \linf
\] 
between the 1-categories of Lie $\infty$-groups and finite-type Lie $\infty$-algebras. Our functor satisfies all of the design criteria listed above. More precisely:   

\begin{theorem1}[Thm.\ \ref{thm:Diff}]\label{thm:main}
The differentiation functor $\Diff(-) \maps \LinfGrp \to \linf$ has the following properties:
\begin{enumerate}
\item The underlying complex of the Lie $\infty$-algebra $\Diff(\cG_\bl) =(L,\el_1,\el_2,\el_3, \cdots)$ is isomorphic to the shifted tangent complex of $\cG_{\bl}$, i.e.,
\[
(L,\el_1) \cong N_{\ast}(T_{\bl}\cG)[-1].
\]

\item Let $G_\bl$ be a simplicial Lie group with Lie algebra $\gb$. There is a \und{canonical} natural isomorphism of $L_\infty$-algebras
\[
\Phi_{G_\bl} \maps N_\ast \gb \xto{\cong} \Diff(\Wbar{G}).
\]

\item The functor $\Diff(-)$ is exact in the following sense: It sends weak equivalences and Kan fibrations between Lie $\infty$-groups to $L_\infty$-quasi-isomorphisms and fibrations between Lie $\infty$-algebras, respectively, and it preserves pullbacks along Kan fibrations. 
\end{enumerate}
\end{theorem1}
Note that the morphisms in the category $\linf$ are taken to be the so-called ``$\infty$-morphisms'' or ``weak'' $L_\infty$-morphisms. That is, we implicitly consider $\linf \sse \cat{dgCoCom}$ as a full subcategory of conilpotent dg cocommutative coalgebras.

\subsection{Tractability}
Let us describe how the theorem addresses the fourth design criterion. As explained, for example, in \cite[\Sec 4.4]{Lie3} every finite-type Lie $\infty$-algebra $L=(L,\el_1,\el_2,\el_3,\cdots)$ admits a Postnikov decomposition $ \cdots \fib \tau_{\leq 2}L \fib \tau_{\leq 1}L \fib \tau_{\leq 0}L=H_0(L)$ in which the zeroth stage is a finite-dimensional Lie algebra, and the $n$th stage is presented as the pullback along a $k$-invariant. Such a $k$-invariant is a morphism $\tau_{\leq n-1}L \to \gl(V) \semiop V[n]$ in $\linf$ where $V$ is a finite-dimensional vector space, $V[n]$ is a shifted copy, and  $\gl(V) \semiop V[n]$ is the usual semi-direct product. A crucial observation made in \cite{Lie3} is that this $k$-invariant, roughly speaking, integrates to a $k$-invariant for the finite-dimensional Lie $\infty$-group $\cG_\bl(L)$ integrating $L$. This global $k$-invariant corresponds to a map into the classifying space of a simplicial Lie group $\Wbar{\bigl(K(V,n)//\wti{\GL}(V) \bigr)}$. Statements (2) and (3) of Theorem \ref{thm:main} above guarantee that the differentiation of the global $k$-invariant recovers the infinitesimal $k$-invariant that we started with. This observation plays an important role in our joint work in progress \cite{Lie2} with Wolfson, which establishes Lie's 2nd Theorem for Lie $\infty$-algebras.

\subsection{The $\infty$-categorical/formal moduli point of view} \label{sec:fmp}
Given the appearance of formal $\infty$-groups and $L_\infty$-algebras in this work, it is not surprising that a comparison can be made with the theory of formal moduli stacks and derived deformation theory. From this point of view, the difficulty of the differentiation problem for Lie $\infty$-groups is analogous to the difficulty in extracting an explicit dg Lie algebra directly from a given derived Schlessinger functor for a deformation problem.

Recall that the Lurie-Pridham Theorem provides an equivalence 
$\infcatW{dgLie}{W} \simeq \infcat{FMP}$ between the $\infty$-category of unbounded dg Lie algebras over a field of characteristic zero localized with respect to the quasi-isomorphisms, and the $\infty$-category of formal moduli problems\footnote{See, for example, Def.\ 1.4 in the exposition \cite{CG:DDT}.}. While the technical aspects of our work rely heavily on the ideas developed by Pridham in his proof \cite{JP:Def} of the theorem, the conceptual aspects are perhaps closer in spirit to the formalism developed by D.\ Gaitsgory and N.\ Rozenblyum in \cite{GR2}. 

In addition to these conceptual relationships, one can say something precise. As we recall in Sec.\ \ref{sec:canonical-iCFO}, it follows from 
a theorem of K.\ Behrend and E.\ Getzler \cite{BG} that $\infty$-groupoids in $\Sh(\Mfd)$ form a category of fibrant objects for a homotopy theory. Our Prop.\ \ref{prop:icfo-Mfd} implies that the category of pointed Lie $\infty$-groupoids canonically inherits, via the Yoneda embedding, the structure of an ``incomplete category of fibrant objects''. This, in turn, makes the full subcategory of Lie $\infty$-groups a relative category $(\LinfGrp,W)$ with some extra structure. In particular, \cite[Prop.\ 6.7]{RZ} implies that a morphism in $\LinfGrp$ is a weak equivalence if and only if it is a weak equivalence of simplicial sheaves in Joyal's model structure \cite{Joyal:1984}. As a result, localizing with respect to these weak equivalences gives the correct  $\infty$-category $\infcatW{Lie_{\infty}Grp^{\fd}}{W}$, as expected. On the infinitesimal side, 
we showed in \cite{R} that $\linf$ admits the structure of a category of fibrant objects, in which the weak equivalences $W \sse \linf$ are precisely the $L_\infty$ quasi-isomorphisms. V.\ Hinich proved in \cite{Hinich} that the category of unbounded conilpotent dg cocommutative coalgebras admits a model structure for which the cobar functor $\cat{dgCoCom} \to \cat{dgLie}$ is a left Quillen equivalence. Furthermore, it follows from a result of Pridham \cite[Cor.\ 4.56]{JP:Def} (see also \cite[Prop.\ 4.5]{Vallette}) that a morphism in $\linf$ is an $L_\infty$ quasi-isomorphism if and only if it is a weak equivalence in $\cat{dgCoCom}$ if and only if its image under $\Cobar \maps \linf \to \cat{dgLie}$ is a quasi-isomorphism of dg Lie algebras. Hence, by Theorem \ref{thm:main}, differentiation induces an $\infty$-functor
\[
\infcatW{Lie_{\infty}Grp^{\fd}}{W} \xto{\mathbf{\Cobar \cc \Diff(-)}} \infcatW{dgLie}{W}.
\]       
Note that the above $\infty$-functor, in a sense,  ``derives away'' the classical geometry. The resulting dg Lie algebra is quasi-free and, in general, degree-wise infinite dimensional. This makes it incompatible (on the point-set level) with the integration theory developed in \cite{Hen} and \cite{Lie3}. 

\subsection{Extending to algebraic $\kk$-schemes}
Although we focus on reduced Kan simplicial objects in finite-dimensional manifolds, all of the results of Theorem \ref{thm:main} extend straightforwardly if we define a Lie $\infty$-group to be a reduced Kan simplicial object in regular algebraic $\kk$-schemes \cite[\href{https://stacks.math.columbia.edu/tag/06LF}{Tag 06LF}]{stacks-project} over a field of characteristic zero, and take covers in Sec.\ \ref{sec:covers} to be surjective smooth maps. Indeed, the starting point for our construction is the following well-known fact, e.g.\ \cite[\Sec 5]{Serre:Lie}: Given a manifold $X$ with $\cO_X$ its sheaf of smooth functions, and $p \in X$, the completion $\plim_n {\cO_{X,\, p}}/\mm^n_p$ of the corresponding ring of germs is a regular complete local Noetherian $\R$-algebra.

\newcommand{\myparagraph}[1]{\noindent {\bf #1}~}

\subsection{Summary of the paper}
This section provides an overview of the key ideas and technical machinery needed to construct the functor $\Diff(-)$ and prove Theorem \ref{thm:main}. Throughout, $\kk$ denotes a field of characteristic zero.

\subsubsection{Summary of Section \ref{sec:FG}}
\paragraph{Formal submersions and formal manifolds} In Sec.\ \ref{sec:conil}, we begin by considering the category $\conil$ of {\it finite-type} conilpotent counital cocommutative coalgebras. By finite type, we mean that the space of primitives of such a coalgebra is finite dimensional over $\kk$. We recall in Prop.\ \ref{prop:dual} the equivalence 
\begin{equation} \label{eq:intro-equiv}
\kk[-] \maps \conil \overset{\simeq}{\longleftrightarrow} \bigl(\clnAlg\bigr)^\op \maps \spf(-)
\end{equation}
where $\clnAlg$ denotes the category of complete local Noetherian $\kk$-algebras with residue field $\kk$. We summarize the basic properties of these algebras in Appendix \ref{sec:comalg}. The functor $\kk[-]$ sends a coalgebra to its $\kk$-linear dual, while $\spf(-)$ sends an algebra to its continuous dual with respect to the adic topology. We consider the coalgebra $\spf(A)$ as the pointed affine formal $\kk$-scheme whose algebra of functions is $\kk[\spf(A)] = A$. 

We recall in Example \ref{ex:mfd-coalg} that given a pointed manifold $(X,x_0)$, the $\mm$-adic completion $\hA_{X}:=\plim_n \cO_{X}\vert_{x_0}/\mm^n_{x_0}$ described earlier is an object in $\wh{\cat{Alg}}^{\rm fg,loc}_{/ \R}$. Its dual $\spf(\hA_X)$ is the coalgebra of {\it point distributions} $\Dist(X)$ supported at $x_0$. In Sec.\ \ref{sec:formsub}, we define a {\it formal submersion} $C \to C'$ in $\conil$ to be a coalgebra morphism whose dual morphism in $\clnAlg$ is formally smooth in the sense of classical commutative algebra. We borrow the terminology of \cite[\Sec 3]{Kont:Formal} and say $C \in \conil$ is a {\it pointed formal manifold} if its counit is a formal submersion.

\paragraph{Descent categories} In Def.\ \ref{def:descent}, we recall from \cite{BG} the notion of a {\it descent category} $\cD$ equipped with a distinguished subcategory $\cS \sse \cD$ of covers. We prove in Prop.\ \ref{prop:descent} that $\conil$ equipped with the subcategory $\cS_{fs}$ of formal submersions is a descent category. Furthermore, we show that, as expected, every pointed formal manifold $\cX$ is isomorphic to the cofree coalgebra $\coSym(\prim(\cX))$ cogenerated by its primitives. 
In Sec.\ \ref{sec:simpdesc}, we recall the definitions of $\infty$-groupoid and $\infty$-group objects in a descent category, along with the notions of Kan fibrations, covering fibrations, and hypercovers. 

\paragraph{Categories with covers}
Section \ref{sec:covers} begins with the definition of a {\it category of covers} $(\cC,\cS)$, introduced in \cite{W} and further developed in \cite{Lie3}. 
The axioms are designed to capture the correct notion of a ``smooth object'' in a descent category. The category of manifolds $\Mfd$ with surjective submersions $\cS_{ss}$ as covers is the standard example. We show that the full subcategory $\conilsm \sse \conil$ of pointed formal manifolds with formal submersions as covers provides another example. The notion of $\infty$-groupoid, $\infty$-group, Kan fibration, hypercover etc.\ extend in the obvious way from descent categories to categories with covers. 

\paragraph{Pointed formal $\infty$-groupoids}
We define in Def.\ \ref{def:form-inf-gpd} the category of {\it pointed formal $\infty$-groupoids} $\FGpd:=\Gpd_\infty(\conilsm)$ as the full subcategory of simplicial coalgebras whose objects are $\infty$-groupoids in $\conilsm$. The category of {\it formal $\infty$-groups} is then full subcategory $\FG$ of reduced objects in $\FGpd$. The classifying space $\Wbar{H}$ of a simplicial Hopf algebra $H_\bl \in \Grp(\sconil)$ is a formal $\infty$-group, as expected. (See Example \ref{ex:WG}.)

\paragraph{Geometric functors}
We introduce in Sec.\ \ref{sec:geo-func} a useful technical tool. We say a functor $F \maps (\cC,\cS_\cC) \to (\cD,\cS_\cD)$ between a category with covers and a descent category is a {\it geometric functor} if it preserves covers and pullbacks of covers. Proposition \ref{prop:geo-func} gives the following key results:

First, we show that 
every geometric functor induces a functor $F_\bl \maps \Gpd_\infty(\cC) \to \Grp_\infty(\cD)$ between $\infty$-groupoids which preserves Kan fibrations, hypercovers, and pullbacks along covering fibrations. Furthermore, if $F$ {\it reflects covers}, i.e.\ $F(f) \in \cS_{\cD} \Leftrightarrow  f \in \cS_{\cC}$, then $F_\bl$ reflects Kan fibrations, covering fibrations, and hypercovers. 
Finally, if $F$ is geometric, then it sends simplicial group objects in $\cC$ to simplicial group objects in $\cD$, and this is compatible with taking classifying spaces, i.e.\  there is a natural isomorphism 
\begin{equation} \label{eq:intro-geo-Wb}
F_\bl \cc \Wbar{(-)} \cong \Wbar{(-)} \cc F_\bl.   
\end{equation}
Corollary \ref{cor:geo-func-ex} implies that the functor of point distributions 
\begin{equation} \label{eq:intro-geo-point-dist}
\Dist \maps (\Mfd_\ast,\cS_{ss}) \to (\conil,\cS_{fs}) 
\end{equation}
is a geometric functor. In particular, every surjective submersion between pointed manifolds induces a formal submersion between their coalgebras of point distributions. Other important examples include, respectively, the inclusion of pointed formal manifolds and the primitives functor: 
\begin{eqnarray} 
&(\conilsm,\cS_{fs} ) \emb (\conil,\cS_{fs}) \label{eq:intro-geo-func-emb}\\
&(\conilsm,\cS_{fs}) \xto{\prim} (\FDVect,\cS_{epi}). \label{eq:intro-geo-func-prim}
\end{eqnarray}
Both of these are geometric functors that reflect covers. Also, in Sec.\ \ref{sec:canonical-iCFO}, we show that the category of sheaves $\Sh(\cC)$  
on a category with covers admits the structure of a descent category
$(\Sh(\cC),\cS_{\mix})$ where $\cS_{\mix}$ is the subcategory of so-called ``mixed covers'' (Def.\ \ref{def:mixed}). In this way, the Yoneda embedding
\begin{equation} \label{eq:intro-yon}
\yo \maps (\cC,\SC) \to (\Sh(\cC),\cS_{\mix}) 
\end{equation}
is promoted to a geometric functor.

\paragraph{Homotopy theory via CFOs and iCFOs}
We begin Sec.\ \ref{sec:hmtpy} by recalling the definition of a {\it category of fibrant objects} (CFO) \cite{Brown}, and the notion of an {\it exact functor} between such categories. We recall in Thm.\ \ref{thm:descent-CFO} one of the main results of \cite{BG}: The category $\Gpd_\infty(\cD)$ of $\infty$-groupoid objects in any descent category $(\cD,\cS)$ admits the structure of a CFO in which the fibrations are the Kan fibrations, and the acyclic fibrations are the hypercovers. To obtain the analogous statement for $\infty$-groupoid objects in a category with covers, we need a slightly weaker structure, due to the lack of finite limits. We recall in Def.\ \ref{def:icfo} the notion of an {\it incomplete category of fibrant objects} (iCFO) introduced in \cite{RZ}. Every CFO is an iCFO, but in general for the latter, we only require the existence of pullbacks of fibrations if the fibration is acyclic. However, if the pullback of a fibration does exist, then we require that it be a fibration. The analogue of an exact functor between iCFOs $\cC$ and $\cD$ is an {\it exact functor with respect to $S$} (Def.\ \ref{def:iCFO-exact}), where $S$ is a subclass of fibrations in $\cC$ containing all acyclic fibrations and all morphisms to the terminal object. 

\paragraph{The canonical iCFO structure on $\Gpd_\infty(\cC)$}
Proposition \ref{prop:icfo-exist} allows us to transfer the Behrend-Getzler CFO structure on $\infty$-groupoids in a descent category to an iCFO structure on $\infty$-groupoids in a category with covers. 
More precisely, if $F \maps (\cC,\SC) \to (\cD,\SD)$ is a geometric functor that reflects covers, then the category $\Gpd_\infty(\cC)$ admits the structure of an iCFO with functorial path objects, in which the fibrations are the Kan fibrations, and in which a morphism in $\Gpd_{\infty}(\cC)$ is  a weak equivalence if and only if its image under $F_\bl$ is a weak equivalence in $\Gpd_{\infty}(\cD)$. In particular, the acyclic fibrations in $\Gpd_\infty(\cC)$ are exactly the hypercovers. We call any iCFO structure on $\Gpd_\infty(\cC)$  with these fibrations and acyclic fibrations the {\it canonical iCFO structure}. 

Proposition \ref{prop:icfo-exact-geo} implies that if $\Gpd_\infty(\cC)$ admits the canonical iCFO structure, then any geometric functor $F \maps \cC \to \cD$ induces an exact functor $F_\bl \maps \Gpd_\infty(\cC) \to \Gpd_\infty(\cD)$ with respect to covering fibrations. Moreover, $F_\bl$ preserves all fibrations and path objects. 

Note that Prop.\ \ref{prop:icfo-exist}, when combined with Prop.\ \ref{prop:icfo-Mfd} discussed below, subsumes one of the main results of our previous work \cite{RZ} with C.\ Zhu. Furthermore, in contrast to \cite{RZ}, our results avoid the use of points for Grothendieck topologies, and ``stalk-wise'' constructions for sheaves.

\paragraph{Homotopy theory in $\LinfGpd$ and $\FGpd$} 
Using the geometric functors \eqref{eq:intro-geo-point-dist}, 
\eqref{eq:intro-geo-func-emb}, and \eqref{eq:intro-yon}, we show in 
Prop.\ \ref{prop:icfo-Mfd} and Thm.\ \ref{thm:formal-hmtpy} that  
the categories of Lie $\infty$-groupoids $\LinfGpd$, pointed Lie $\infty$-groupoids $\LinfGpd_\ast$, and pointed formal $\infty$-groupoids $\FGpd$
all admit the canonical iCFO structure. 

We then use the geometric functors \eqref{eq:intro-geo-func-prim} and
\eqref{eq:intro-geo-point-dist} in Thm.\ \ref{thm:formal-hmtpy} to show that the simplicial primitives functor
\begin{equation} \label{eq:intro-prim-exact}
\prim_\bl \maps \FGpd \to \sFDVect 
\end{equation}
is exact with respect to covering fibrations and reflects weak equivalences and fibrations. Furthermore, the simplicial point distributions functor $\Dist_\bl \maps \LinfGpd_\ast \to \FGpd$
is exact with respect to covering fibrations. 

\paragraph{The formal $\infty$-group of a Lie $\infty$-group}
In Sec.\ \ref{sec:LG-FG} we consider the restriction of functor $\Dist_{\bl}(-)$ to the category of Lie $\infty$-groups:
\begin{equation} \label{eq:intro-Dist-exact}
\Dist_\bl \maps \LinfGrp \to \FG.
\end{equation}
This functor preserves weak equivalences, all fibrations, and all pullbacks of fibrations, since every fibration between reduced objects is a covering fibration. 

Given a simplicial Lie group $G_\bl$ with Lie algebra $\gb$, we show in Cor.\ \ref{cor:WG-Ug} that classical Cartier duality combined with the isomorphism \eqref{eq:intro-geo-Wb} produces a natural isomorphism of formal $\infty$-groups 
\begin{equation} \label{eq:intro-WG-Ug}
\Dist_\bl(\Wb G) \cong \Wbar{\cU(\g)}.
\end{equation}

\paragraph{Every formal $\infty$-group admits a PBW basis}
We recall in Sec.\ \ref{sec:almost} the definition of the category $s^{+}\cC$ of {\it almost simplicial objects} in a category $\cC$. Such an object may be thought of as a simplicial object without a zeroth face map. Denote by $U^{+}_{\bl} \maps s\cC \to s^{+}\cC$ the obvious forgetful functor. 

Next, we consider in Sec.\ \ref{sec:Wbar} the {\it inhomogeneous bar construction} $\Wbarin{G}$ of a simplicial group object $G_\bl$, and we recall from \cite{JP:Dmod} the natural isomorphism between $\Wbarin{G}$ and the classical $\Wbar{G}$ as in \cite{EM}. From the definition \eqref{eq:Winh}, it is clear that the group structure of $G_\bl$ is completely encoded in the zeroth face map of $\Wbarin{G}$. As a result, given a simplicial Lie algebra $\gb$, Prop.\ \ref{prop:Winhom} shows that the classical PBW theorem yields a canonical isomorphism of almost simplicial coalgebras 
\begin{equation} \label{eq:intro-PBW-classic}
U^{+}_{\bl} \coSym \bigl(\Wbar {\g} \bigr) \cong U^{+}_{\bl}\Wbar{\cU(\g)}
\end{equation}
Here, the simplicial vector space $\Wbar{{\g}}$ is the classifying space of the underlying additive simplicial group of $\gb$, and $\coSym(\Wbar{{\g}})$ is the cofree conilpotent cocommutative simplicial coalgebra cogenerated by $\Wbar{{\g}}$. 

We generalize \eqref{eq:intro-PBW-classic} to all pointed formal $\infty$-groupoids 
by promoting it to a definition (Def.\ \ref{def:PBW}). A {\it PBW basis  for a pointed formal $\infty$-groupoid $\cX_\bl$} is an isomorphism of almost simplicial coalgebras 
\begin{equation} \label{eq:intro-PBW}
U^+_{\bl}\coSym(\prim_\bl(\cX)) \iso U^+_{\bl}(\cX).
\end{equation}
Theorem \ref{thm:pbw} asserts that every pointed formal $\infty$-groupoid admits a PBW basis. This is a key result, since it will imply that the differentiation of a Lie $\infty$-group is not just a dg coalgebra, but the Chevalley-Eilenberg coalgebra of a Lie $\infty$-algebra.

\subsubsection{Summary of Section \ref{sec:FD}} \label{sec:intro-sec3}
The main result of this section is the construction of a formal differentiation functor from formal $\infty$-groups to Lie $\infty$-algebras.

\paragraph{Dold-Kan for simplicial coalgebras} After recalling the simplicial and almost simplicial Dold-Kan correspondence in Sec.\ \ref{sec:simp-DK}, we observe in Lemma \ref{lem:coalg-adjoint} that the inverse of the normalized chains functor restricts to a functor $K^{\coalg}_\bl \maps \dgcocom \to \scocom$ between conilpotent counital dg cocommutative coalgebras and simplicial cocommutative coalgebras.

\paragraph{Chevalley-Eilenberg functor vs.\ $\Wbar{}$-construction} 
Let $\gb$ be a simplicial Lie algebra and $\cU(\gb)$ its simplicial universal enveloping algebra. In Sec.\ \ref{sec:TheMap}, we compare the formal $\infty$-group $\Wbar{\cU(\g)}$ with the simplicial coalgebra $K^{\coalg}_\bl(\CE(N_\ast \g))$. Here, $\CE(N_\ast \g)$ is the Chevalley-Eilenberg coalgebra of the normalized chains $\Ng$, which is a dg Lie algebra. The main result of this subsection is Thm.\ \ref{thm:themap} which exhibits a natural map of simplicial coalgebras
\begin{equation} \label{eq:intro-CE1}
K^{\coalg}_\bl(\CE(N_\ast \g)) \to \Wbar{\cU(\g)}. 
\end{equation}
Corollary \ref{cor:prim-map} implies that this map induces an isomorphism of normalized complexes after taking primitives:
\begin{equation} \label{eq:intro-CE2}
N_{\ast}(\gb)[1] \cong N_\ast\bigl(\prim_\bl\Wbar{\cU(\g)}\bigr),
\end{equation}
where $N_{\ast}(\gb)[1]_{k} = N_{\ast}(\gb)_{k-1}$.
This result is crucial for establishing statement (2) of Theorem \ref{thm:main} above. The proof of Thm.\ \ref{thm:themap} incorporates D.\ Quillen's theory of principal dg coalgebra bundles \cite{Quillen:RHT} along with ideas borrowed from the work of H.\ Cartan \cite{Cartan} and J.\ Moore \cite{Moore1,Moore2} on the classical $\Wbar{}$-construction.

\paragraph{Cosimplicial Dold-Kan adjunction for commutative algebras}
In Sec.\ \ref{sec:cosimDK}, we consider the Dold-Kan correspondences  $N^\ast \adj K^\bl$ and $N^\ast_{+} \adj K^\bl_{+}$ for cosimplicial and almost cosimplicial objects, respectively. The right adjoints $\Kb$ and $\Kb_{+}$ restrict to functors $\Kb_{\Com} \maps \cdga \to \cCom$ and $\Kb_{\Com,+} \maps \cga \to \cpCom$ on the categories of commutative dg algebras, and commutative graded algebras, respectively. We explain in Rmk.\ \ref{rmk:JP} that the left adjoint $\Da \maps \cCom \to \cdga$ to $\Kb_{\Com}$ is featured throughout Pridham's work in derived deformation theory. 

In Thm.\ \ref{thm:D-exists} and Prop.\ \ref{prop:reduced} we provide a careful analysis of $\Da$ and its almost cosimplicial analogue $\Dap \maps \cpCom \to \cga$ via the adjoint lifting theorem for monadic functors. This yields several key results. First, the adjoint pairs $\Da \adj \Kb_{\Com}$ and $\Dap \adj \Kb_{\Com, +}$ induce adjunctions
\[
\Da \maps c\rclnAlg \adjunct \rcdga \maps \Kb_{\Com}, \qquad 
\Dap \maps c^{+}\rclnAlg \to \rcga \maps \Kb_{\Com, +}.
\]
Above, $c\rclnAlg$ is the category of {\it reduced} cosimplicial commutative algebras that are level-wise complete local Noetherian $\kk$-algebras. As usual, a cosimplicial $\kk$-algebra $A^\bl$ is reduced if $A^0 =\kk$. The 
almost cosimplicial analogue $c^{+}\rclnAlg$ is defined similarly. On the other side, $\rcdga$ denotes the category of {\it reduced} and finitely-generated cdgas.
A graded algebra $A^\ast$ is reduced if $A^0 = \kk$. The category $\rcga$ of reduced finitely-generated graded commutative algebras is defined similarly.    

Second, if $W^\bl$ is an almost cosimplicial vector space, then \eqref{eq:D-exists-Sym} and \eqref{eq:prop-reduced2} imply that 
there is a natural isomorphism of graded algebras 
\begin{equation} \label{eq:intro-sym}
\Dap(\cphSym(W)) \cong \gSym(\Nas_+(W))
\end{equation}
Here $\cphSym(W)$ is the completion of the free almost cosimplicial commutative algebra generated by $W^{\bl}$, and $\gSym(\Nas_+(W))$ is the free graded commutative algebra generated by the normalized graded vector space $\Nas_+(W)$. 

Third, if $U^\bl_+ \maps \cCom \to \cpCom$ and $U^\ast_+ \maps \cdga \to \cga$ denote the functors that forget, respectively, the zeroth coface map and differential, then we have an isomorphism of graded algebras 
\begin{equation} \label{eq:intro-underlying}
U^*_{+} \Da(A^\bl)  \cong \Dap U^\bl_{+}(A^\bl), 
\end{equation}
natural in the cosimplicial algebra $A^\bl$.

\paragraph{Lie $\infty$-algebras} We establish our conventions for Lie $\infty$-algebras in Sec \ref{sec:Linf}, following \cite[\Sec 3]{R}, with one minor yet important difference. Denote by $\rdgcocom \sse \dgcocom$ the full subcategory whose objects are {\it reduced finite-type} dg coalgebras. By definition, $C_\ast$ is reduced if $C_0 =\kk$, and $C_\ast$ is finite-type if its complex of primitives is degree-wise finite-dimensional. In contrast with \Sec 3.1 of \cite{R}, we take as our category of Lie $\infty$-algebras to be the full subcategory $\linf \sse \rdgcocom$ whose objects are dg coalgebras $(C_\ast,\del)$ such that there exists an isomorphism of graded coalgebras 
\[
C_\ast \cong \cogrSym(L) \quad \text{in $\rgcocom$}
\]
for some positively graded vector space $L$. Here $\cogrSym(L)$ is the usual graded symmetric coalgebra cogenerated by $L$. In the early derived algebraic geometry literature, e.g.\ \cite[Prop.\ 1.2]{Kap}, objects in $\linf$ are called ``weakly isomorphic weak Lie algebras''. 

It follows that every object $(C_\ast,\del)$ in $\linf$ is non-canonically isomorphic to the Chevalley-Eilenberg coalgebra of a Lie $\infty$-algebra of the form $(L[-1],\el_1,\el_2,\el_3, \cdots)$. Moreover, the chain complex $(L[-1],\el_1)$ is isomorphic to the {\it tangent complex} of $(C_\ast,\del)$:
\[
\tan_\ast(C):=\prim_\ast(C)[-1],
\]
i.e., the shifted complex of its primitives. Since $\linf$ is the essential image of the embedding $\LnA{\infty}^{\ft} \emb \dgcocom$ of the category of finite-type Lie $\infty$-algebras considered in \cite{R}, we have $\linf \simeq \LnA{\infty}^{\ft}$. Hence, $\linf$ admits the usual CFO structure (Thm.\ \ref{thm:Linf-CFO}) in which a morphism in $\linf$ is a weak equivalence or a fibration if and only if the induced morphism on tangent complexes is, respectively, a quasi-isomorphism or a surjection in all positive degrees.

\paragraph{The formal differentiation functor}
The above equivalence \eqref{eq:intro-equiv} between coalgebras and algebras extends to the reduced (almost) simplicial and dg/graded contexts, e.g.,
\[
\begin{split}
\kk[-]^{\bl} \maps s\rconil &\overset{\simeq}{\longleftrightarrow} \bigl(c\rclnAlg \bigr)^\op \maps \spf_{\bl}(-)\\
\kk[-]^{\ast} \maps \rdgcocom &\overset{\simeq}{\longleftrightarrow} \bigl(\rcdga)^\op \maps \dgspf(-).
\end{split}
\] 
In Thm.\ \ref{thm:FDiff}, we introduce the {\it formal differentiation functor}
\begin{equation} \label{eq:intro-FDiff}
\FDiff(-) \maps \FG \to \linf
\end{equation}
as the composition
\[
\FG \emb s\rconil \xto{\kk[-]^{\bl}} \bigl(c\rclnAlg \bigr)^{\op} \xto{{D^\ast}^{\op}} \bigl(\rcdga \bigr)^{\op} \xto{\dgspf} \rdgcocom,
\]
and we prove the following key statements concerning its properties.

First, we verify that $\FDiff(\fG_\bl)$ is indeed a Lie $\infty$-algebra by using the isomorphisms \eqref{eq:intro-sym} and \eqref{eq:intro-underlying}, along with the existence of a PBW basis \eqref{eq:intro-PBW} for the formal $\infty$-group $\fG_\bl$.

Second, we prove that the underlying tangent complex of the Lie $\infty$-algebra $\FDiff(\fG_{\bl})$ is isomorphic to the shifted normalized complex of primitives of $\fG_\bl$:
\[
\tan_\ast \bigl( \FDiff(\fG_\bl) \bigr) \cong N_\ast\bigl( \prim_\bl(\fG)\bigr)[-1].
\]
This statement follows from combining the (not necessarily canonical) isomorphism of graded coalgebras $U^+_\ast \FDiff(\fG_\bl) \cong U^+_\ast \codgSym(N_\ast \prim_\bl(\fG))$, along with Lemma \ref{lem:G1} which exhibits a \und{natural} isomorphism of chain complexes
\begin{equation} \label{eq:intro-prim}
N_\ast \prim_\bl(\fG) \xto{\cong} \prim_\ast\bigl(\FDiff(\fG_\bl) \bigr).
\end{equation}   

Third, if $\fG_\bl=\Wbar{G}$, where $G_\bl$ is a simplicial Lie group, then we can improve the previous results by exhibiting a \und{canonical} natural isomorphism of Lie $\infty$-algebras
\begin{equation} \label{eq:intro-WG-iso}
\Ng \xto{\cong} \FDiff( \Wbar \cU(\gb)). 
\end{equation}
Its existence is due to the simplicial coalgebra morphism \eqref{eq:intro-CE1} discussed above. That it is an isomorphism follows from \eqref{eq:intro-CE2}, along with the well known fact that the tangent complex functor for $L_\infty$-algebras reflects isomorphisms.

The last assertion in Thm.\ \ref{thm:FDiff} is that the functor $\FDiff(-)$ preserves all limits, preserves and reflects weak equivalences, and preserves and reflects fibrations. These statements follow from combining the natural isomorphism \eqref{eq:intro-prim} with the fact that the primitives functor \eqref{eq:intro-prim-exact} preserves and reflects both weak equivalences and fibrations between formal $\infty$-groups.  

\subsubsection{Summary of Section \ref{sec:LieDiff}}
Finally, we define the {\it differentiation functor} for Lie $\infty$-groups
\[
\Diff(-) \maps \LinfGrp \to \linf
\]
as the composition of the formal differentiation functor \eqref{eq:intro-FDiff} with the simplicial point distribution functor \eqref{eq:intro-Dist-exact}:
\[
\LinfGrp \xto{\Dist_\bl} \FG \xto{\FDiff} \linf.
\]
We then restate our main theorem, Thm.\ \ref{thm:main}, as Thm.\ \ref{thm:Diff} and give the proof. Let us sketch the key arguments. 

For the proof of statement (1) of Thm.\ \ref{thm:main}, first denote by $T_{\bl}(-) \maps \LinfGrp \to s\FDVect$  the tangent functor which takes the level-wise tangent space at the canonical basepoint. There is a natural isomorphism of simplicial vector spaces $T_{\bl}(\cG) \cong \prim_\bl \bigl(\Dist(\cG)\bigr)$. Then (1) follows from the fact that the formal differentiation functor takes values in Lie $\infty$-algebras along with the isomorphism \eqref{eq:intro-prim}. Statement (2) follows from \eqref{eq:intro-WG-iso} along
with the isomorphism \eqref{eq:intro-WG-Ug}  between $\Dist_\bl(\Wb G)$ and  $\Wbar{\cU(\g)}$. Statement (3) follows from the exactness properties of $\FDiff(-)$ summarized above, along with the exactness properties of the simplicial point distributions functor \eqref{eq:intro-Dist-exact}.

\subsection{Notation and conventions} 
From Sec.\ \ref{sec:FG} onward, we adhere to the notations and conventions described below. 
\paragraph{Graded and dg objects}
We follow the conventions of \cite[Sec.\ 2.1]{R}, in general, for graded linear algebra. There is, however, one exception. In contrast with \cite[Sec.\ 2.1]{R}, given a homologically graded vector space $V_\ast$ and $n \in \Z$, we denote by $V[n]$ the graded vector space
\[
V[n]:=\bs^nV \qquad V[n]_{k}:= V_{k-n}.
\]
This agrees with our conventions in \cite{Lie3}. In particular, 
any reference to a ``shifted tangent complex'' $T[-1]$ in this paper will have the underlying graded vector space
\[
T[-1]_{i} = T_{i+1}.
\]
This convention aligns with some of the derived geometry literature, e.g. \cite{Nuiten}.

When the grading of an object is indicated using subscripts, e.g. $C_\ast$, then this always implies that it is homologically graded. The analogous statement 
holds for superscripts, e.g.\ $A^\ast$, and cohomological grading.

In order to keep the notation under control, we will often introduce a graded object as, for example, $C_\ast$ or $A^\ast$, and then reference it later on without subscripts or superscripts, e.g. $C$ or $A$.   

When we wish to emphasize a graded object in contrast with a dg object, we will frequently use the notation
\[
C^{+}_{\ast}
\] 
if $C$ is homologically graded, and 
\[
C^{\ast}_{+}
\] 
if $C$ is cohomologically graded.

All chain complexes are homologically graded and concentrated in non-negative degrees. All cochain complexes are cohomologically graded and concentrated in non-negative degrees, with one exception: the underlying cochain complex of the cdga 
\[
\Lam^{n,\ast}
\] 
defined in \eqref{eq:lam-def} is cohomologically graded and concentrated in degrees $-(n+1),\ldots,0$.

With the exception of Sec.\ \ref{sec:fmp} above, all (differential) graded Lie algebras are homologically graded and concentrated in non-negative degrees. By definition, a Lie $\infty$-algebra in this paper is a homologically graded $L_\infty$-algebra concentrated in non-negative degrees.

Any notion of ``tangent complex'' in this paper will have an underlying homologically graded chain complex.  

All (differential) graded coalgebras have underlying chain complexes, and hence are homologically graded. All (differential) graded algebras have underlying cochain complexes and are therefore cohomologically graded.  

Given a homologically graded vector space $V_\ast$, we denote by $\bs V$ the suspension of $V$, and we denote by $\bs^{-1} V$ the desuspension of $V$. That is: 
\[
(\bs V)_{i}:=V_{i-1} \qquad (\bs^{-1} V_{i}):=V_{i+1}.
\]

\paragraph{Linear duals}
Given an object $V$ with an underlying vector space over $\kk$, the notation  
$(V)^{\vee}$ denotes the usual (i.e.\ discrete) $\kk$-linear dual $\hom_{\Vect}(V,\kk)$. In particular: 
If $V_\bl$ is a (almost) simplicial vector space, then $(V^{\vee})^{\bl}$ denotes the dual cosimplicial vector space $(V^{\vee})^{n} = \hom_{\kk}(V_n,\kk)$. The dual of a (almost) cosimplicial vector space $W_\bl$ is the simplicial vector space $(W^{\vee})_\bl$ defined in an analogous way.   

If $(V_\ast,d)$ is chain complex (or, more generally, a homologically graded vector space) then its dual $\big((V^{\vee})^\ast,\del \bigr)$ is the cochain complex $(V^{\vee})^n:= \hom_{\kk}(V_n,\kk)$ with differential 
\[
\delta(f)(x):=(-1)^{\deg{f}}f(dx).
\]
If $(W^\ast,\delta)$ is a cochain complex, then its dual chain complex  
$\big((W^{\vee})_\ast,d \bigr)$ is defined in the analogous way.

We emphasize that the formal spectrum functor $\spf(-)$ in \eqref{eq:dual} and its simplicial and graded analogues $\spf_{\bl}$ and $\spf_{\ast}$ are defined using the continuous linear dual $\hom^{\cont}_{\kk}(-,\kk)$ rather than $(-)^{\vee}$. Of course, the two operations agree on finite-dimensional objects.     

\paragraph{Cofree coalgebras and free algebras}
In both graded and (almost) co/simplicial contexts: 
\begin{itemize}
\item $\coSym(V)$, and in some places $S(V)$, denote
the cofree counital conilpotent cocommutative coalgebra cogenerated by $V$, i.e., the ``symmetric coalgebra''. 
\item $\Sym(V)$ denotes the free unital commutative algebra generated by $V$.
\end{itemize}

\paragraph{Notation for (co)simplicial and almost (co)simplicial objects}
In order to minimize notation, we will often introduce a simplicial or cosimplicial object as, for example, $C_\bl$ or $A^\bl$, and then reference it later on without subscripts or superscripts, e.g. $C$ or $A$.   

In Sections \ref{sec:almost} and Sec.\ \ref{sec:alm-cosim} we introduce almost simplicial  and almost cosimplicial objects, respectively.
When we wish to emphasize an almost simplicial object in contrast with a simplicial object, we will frequently use the notation
\[
C^{+}_{\bl}
\] 
Similarly, when considering almost cosimplicial objects in contrast with cosimplicial objects, we will use
\[
C^{\bl}_{+}.
\] 

\paragraph{A remark on exposition} We attempt to write for a somewhat broad audience that includes homotopy theorists of different stripes, as well as geometers, both differential and algebraic. We aim to give a self-contained presentation and/or precise references to the literature, introductory or otherwise, whenever possible. In particular, we assume no familiarity with derived geometric machinery, and instead appeal to basic arguments in classical commutative algebra whenever it is reasonable to do so.

\subsection{Acknowledgments}
This paper arose as an offshoot of an ongoing collaboration with Jesse Wolfson, and the author thanks him for many substantial conversations and helpful advice. The results in Section \ref{sec:cosimpDKalgebras} owe a significant intellectual debt to the work of Jon Pridham, in particular, \cite{JP:Def} and  \cite{JP:Poisson}.   
The author was supported by NSF Grant DMS-2305407.

\section{Formal $\infty$-groups}\label{sec:FG} 

\subsection{Conilpotent coalgebras and complete local algebras}
\label{sec:conil}
We first consider non-graded conilpotent counital cocommutative coalgebras over $\kk$. These are equivalently connected graded coalgebras in the sense of
\cite[\Sec B.3]{Quillen:RHT} concentrated in degree 0. By definition, such a coalgebra $C$ admits a canonical coaugmentation $\kk \to C$. Denote by $\prim(C) \in \Vect$ the linear subspace of primitive elements of $C$. Let $\conil$ denote the category of \df{finite-type} conilpotent counital cocommutative coalgebras. 
By finite-type, we mean that the space of primitives is finite-dimensional. We have an adjoint pair
\begin{equation} \label{eq:prim-adj}
\kk \semiop^{\coalg}(-) \maps \FDVect \adjunct \conil \maps \prim 
\end{equation}
where the comultiplication on $\kk \semiop^{\coalg} (V)$ is determined by assignment $\Del(v):= v \tensor 1 + 1 \tensor v$ for all $v \in V$. The category $\conil$ has all finite limits. Indeed, products are given by $\kk$-linear tensor products $C \tensor C'$, as usual, and pullbacks are constructed as in \cite[Sec.\ 2.4.4]{Grunen-Pare}. 

Let $\clnAlg$ denote the category of complete local Noetherian $\kk$-algebras with residue field $\kk$ as described Appendix \ref{sec:comalg}.
The following proposition and its corollary are classical, e.g. \cite[\Sec 37.2]{Haze}. 
\begin{proposition} \label{prop:dual}
The assignment 
\[
C \mapsto \kk[C]:=C^{\vee}
\]
of a coalgebra $C$ to its $\kk$-linear dual induces an equivalence of categories
\begin{equation} \label{eq:dual}
\kk[-] \maps \conil \overset{\simeq}{\longleftrightarrow} \bigl(\clnAlg\bigr)^\op \maps \spf(-)
\end{equation}
whose inverse is the functor 
\[
\spf(A):= \hom^{\cont}_\kk(A,\kk)
\]
which assigns an algebra $A$ to its continuous\footnote{Here $\kk$ is equipped with the discrete topology.} linear dual. 
\end{proposition}
It is conceptually useful to borrow the language of pointed formal schemes over $\kk$, and consider $\kk[C]$ as the ring of functions on its formal spectrum, i.e.\ the coalgebra $C= \spf(\kk[C])$.

\begin{corollary}\label{cor:prim-cot}
Let $C \in \conil$ and $\mm \ideal \kk[C]$ the unique maximal ideal. There is an isomorphism of finite-dimensional vector spaces
\[
\prim(C)^\vee \xto{\cong} \mm/\mm^2
\] 
natural in $C$.
\end{corollary}

\begin{example}\label{ex:mfd-coalg}
Set $\kk=\R$. Given a pointed manifold $X \in \Mfd_\ast$ of dimension $n$,  denote its sheaf of smooth functions by $\cO_X$, its localization at the base point by $A_X:=(\cO_{X,0},\mm)$, and its ${\mm}$-adic completion by $\hA_X:=(\wh{\cO}_{X,0},\wh{\mm})$. Then, by Hadamard's Lemma, $\hA_X$ is a regular complete local Noetherian $\kk$-algebra of dimension $n$. In particular, there are non-canonical generators $x_1,\ldots,x_n \in \wh{\mm} $ such that $\hA_X= \R[[x_1,\ldots,x_n]]$. This induces a functor 
\begin{equation} \label{eq:dist}
\Dist \maps \Mfd_\ast \to \conil, \qquad \Dist(X):=\spf(\hA_X),
\end{equation}
which sends a pointed manifold to its coalgebra of \df{point distributions} \cite[\Sec 5]{Serre:Lie}.
\end{example}

\subsection{Formal submersions as covers} \label{sec:formsub}

\begin{definition} 
A morphism $f \maps C \to D$ in $\conil$ is a \df{formal submersion} if the corresponding morphism 
\[
f^\vee \maps \kk[D] \to \kk[C]
\]
in $\clnAlg$ is formally smooth, e.g.\ \cite[\href{https://stacks.math.columbia.edu/tag/0DYF}{Tag 0DYF}]{stacks-project}. We say $C$ is a \df{pointed formal manifold} if the counit $C \to \kk$ is a formal submersion. We denote by $\conilsm \sse \conil$ the full subcategory of pointed formal manifolds.
\end{definition}

\begin{definition}[\cite{BG}; \Sec 5.1 \cite{Lie3}] \label{def:descent}
A \df{descent category}\footnote{Also called a ``subcanonical'' descent category.} $(\cD,\cS)$ is a category $\cD$ equipped with a subcategory $\cS$ of \df{covers} which satisfies the following axioms:
\begin{enumerate}[label=(\roman*)]
\item \label{ax1} $\cD$ has all finite limits,
\item \label{ax2} The unique map from the terminal object to itself is in $\cS$,
\item \label{ax3} $\cS$ is closed under base change
\item \label{ax4} If $f$ and $g\circ f$ are in $\cS$, then $g$ is in $\cS$,
\item \label{ax5} Every morphism in $\cS$ is an effective epimorphism.
\end{enumerate}
\end{definition}

Denote by $\cS_{\fs} \sse \conil$ the subcategory whose morphisms are formal submersions.

\begin{proposition} \label{prop:descent}
\mbox{}
\begin{enumerate}

\item $(\conil, \cS_\fs)$ is a descent category.

\item If $f \maps C \to D$ is a formal submersion, then $\prim(f)$ is a surjection.

\item $C$ is a formal manifold if and only if there exists $V \in \FDVect$
such that $C$ is isomorphic to the coalgebra $\coSym(V)$.
\item If $f \maps C \to D$ is a morphism between formal manifolds, and $\prim(f)$ is a surjection, then $f$ is a formal submersion. \label{item:surj-formal}

\end{enumerate}
\end{proposition}
\begin{proof}
To keep the notation under control throughout the proof, given $C \in \conil$ 
we set $A_C:=\kk[C]$.
\begin{enumerate}
\item Axioms \ref{ax1} and \ref{ax2} are clear. For axiom \ref{ax3}, given any morphism $B \to A$ in $\clnAlg$ and a formally smooth morphism $B \to A'$, the induced morphism $A \to A \tensor_{B} A'$ is formally smooth \cite[\href{https://stacks.math.columbia.edu/tag/00TJ}{Tag 00TJ}]{stacks-project}. Hence, its completion $A \to A \htensor_{B} A'$, as defined in \eqref{eq:ideal-of-def}, is formally smooth by \cite[\href{https://stacks.math.columbia.edu/tag/07ED}{Tag 07ED}]{stacks-project}.

\mind For axiom \ref{ax4}, let $\ph \maps A \to B$ in $\clnAlg$. Since $\ph$ is a local morphism between complete local Noetherian rings, $\ph$ is formally smooth if and only if $\ph$ is regular, and $\ph$ regular implies that $\ph$ is flat (hence faithfully flat) by \cite[\href{https://stacks.math.columbia.edu/tag/07PM}{Tag 07PM}]{stacks-project}. Now consider composable morphisms $C \xto{f} D \xto{g} E$ in $\conil$ with $f$ and $g \circ f$ formal submersions. Hence, we have      
$A_E \xto{g^\vee} A_D \xto{f^\vee} A_C$ with $f^\vee \cc g^\vee$ regular, and $f^\vee$ faithfully flat. It then follows from \cite[\href{https://stacks.math.columbia.edu/tag/07NT}{Tag 07NT}]{stacks-project} that $g^\vee$ is regular, hence formally smooth.  

\mind For axiom \ref{ax5}, since all morphisms in $\conil$ have kernel pairs, it suffices to show that every formal submersion $f \maps C \to D$ is a split epimorphism. Since $f^\vee$ is formally smooth, for any $n > 1$ there exists, by induction, a solid commutative diagram along with an extension $\si_{n}$:
\[
\begin{tikzdiag}{2}{3}
{
A_C   \& A_D / \mm^{n-1} \\
A_D  \& A_D / \mm^{n} \\
};

\path[->,font=\scriptsize]
(m-1-1) edge node[auto] {$\si_{n-1}$} (m-1-2)
(m-2-1) edge node[auto] {$f^\vee$} (m-1-1)
(m-2-1) edge node[auto] {$$} (m-2-2)
(m-2-2) edge node[auto] {$$} (m-1-2)
;
\path[->,font=\scriptsize,dashed]
(m-1-1) edge node[auto] {$\si_{n}$} (m-2-2)
;
\end{tikzdiag}
\]
Then $\plim \si_{n}$ is the desired section of $f^\vee$.

\item By Cor.\ \ref{cor:prim-cot} it suffices to show that the map on cotangent spaces $\mm_D/\mm_D^2 \to  \mm_C/\mm_C^2$ induced by $f^\vee \maps A_D \to A_C$ is injective. As in the verification of axiom \ref{ax5} above, there exists an algebra morphism $\si \maps A_C \to A_D/\mm^2_D$ such that $\si \cc f^\vee \maps A_D \to A_D/\mm^2_D$ is the canonical surjection. Hence, we have a left inverse to  $\mm_D/\mm_D^2 \to  \mm_C/\mm_C^2$.

\item If $\kk \to A$ is formally smooth, then $A$ is a regular complete local Noetherian $\kk$-algebra \cite[\href{https://stacks.math.columbia.edu/tag/07EI}{Tag 07EI}]{stacks-project}. Hence, $A$ is isomorphic to a ring of formal power series $\kk[[x_1,\ldots,x_n]]$, for some $n\geq 0$, and therefore $\hom^\cont_\kk(A,\kk) \cong  
\coSym(V)$ with $\dim_\kk V =n$.

\item Dualizing the hypotheses on $f \maps C \to D$ consider a morphism $f^\vee \maps A_D \to A_C $ such that the induced map on cotangent spaces $\mm_{D}/\mm^2_{D} \to \mm_{C}/\mm^2_{C}$ is injective. Write $A_D = \kk[[x_1,\ldots,x_m]]$. By extending the elements $f^\vee(x_i) + \mm^2_C$ to a basis for $\mm_{C}/\mm^2_{C}$, and applying Lemma \ref{lem:gen} to $A_C$, we deduce that $f^\vee$ is isomorphic to an inclusion of the form $\kk[[x_1,\ldots,x_m]] \to 
\kk[[x_1,\ldots,x_m,y_1,\ldots,y_d]]$. Since the codomain is isomorphic to  
\[
\kk[[x_1,\ldots,x_m]] \htensor_\kk \kk[[y_1,\ldots,y_d]],
\]
 and $\kk \to \kk[[y_1,\ldots,y_d]]$ is formally smooth, it follows from axiom \ref{ax3} 
that $f^\vee$ is formally smooth as well.
\end{enumerate}
\end{proof}

\begin{example}\label{ex:mfd-fsub}
Keeping in mind Example \ref{ex:mfd-coalg}, if $X \in \Mfd_\ast$, then 
Prop.\ \ref{prop:descent}(3) implies that the coalgebra $\Dist(X)$ is a pointed formal manifold. Let $\phi \maps X \to Y$ be a surjective submersion between pointed manifolds. Then $\phi$ admits a section in a neighborhood of the base point. Hence, the local morphism of $\R$-algebras $\wh{\ph}^\ast \maps \hA_Y \to \hA_X$ induces an injective map $\mm_Y/\mm^2_Y \to \mm_X/\mm_X^2$ between cotangent spaces. Therefore, Prop.\ \ref{prop:descent} \eqref{item:surj-formal} implies that
\[
\wh{\phi}_\ast \maps \Dist(X) \to \Dist(Y) 
\]
is a formal submersion.    

\end{example}

\subsection{Simplicial objects in descent categories} \label{sec:simpdesc}

\begin{notation}[Matching objects] \label{note:hom}
Let $\cC$ be a category with a terminal object, and $X_\bl \in s\cC$ a pointed simplicial object. Given $S \in \sSet$, we denote by 
\[
X^S \in \cC 
\]
the pointed object representing the equalizer\footnote{Described explicitly, for example, in \cite[Prop.\ 4.5]{RZ}.} $\hom(S,X_\bl)\in \cC$ whenever it exists.  
We also use the notation 
\[
X_{n,k}: = \hom(\Lam^{n}_k,X_\bl)
\]
when it is convenient to do so.   
\end{notation} 

Let us recall some key definitions.

\begin{definition} \label{def:kan}
Let $\cD$ be a descent category. Let $f \maps X_\bl \to Y_\bl$ be a morphism in $s\cD$.
\begin{enumerate}
\item Let $n \geq 1$ and $0 \leq k \leq n$. We say $f$  satisfies the \df{Kan condition} $\kan(n,k)$ if the canonical morphism
\[
X_n \to X^{\Lambda^{n}_{k}} \times_{Y^{\Lambda^{n}_{k}}} Y_n
\]   
is a cover. We say $f$ is a \df{Kan fibration} if $\kan(n,k)$ is satisfied for all $n \geq 1$ and all $0\leq k \leq n$.

\item  We say $f$ is a \df{covering fibration} \cite[Def.\ 2.6(4)]{Lie3} if $f$ is a Kan fibration and $f_0$ is a cover.

\item Let $m \geq 0$. We say $f$  satisfies the \df{acyclic Kan condition} $\acyc(m)$ if the canonical morphism    
\[
X_m \xto{\mu_m(f)} X^{\pa \Del^m} \times_{Y^{\pa \Del^m}} Y_m
\] 
is a cover. We say $f$ is a \df{hypercover} if $\acyc(m)$ is satisfied for all $m \geq 0$. 
\end{enumerate}
\end{definition}

A simplicial object $X_\bl$ in a descent category is an \df{$\infty$-groupoid} object if $X_\bl \to \ast$ is a Kan fibration. We denote by $\Gpd_\infty(\cD) \sse s\cD$ the full subcategory of $\infty$-groupoids in $\cD$.

\subsection{Categories with covers and geometric functors} \label{sec:covers}
We recall that a \df{category with covers}\footnote{Also called a ``subcanonical'' category with covers.}  $(\cC, \cS)$ is a category $\cC$ equipped with a subcategory $\cS$ of covers which satisfies the following axioms:
\begin{enumerate}[label=(\alph*)]
\item \label{axC1} $\cC$ has a terminal object $\ast$ and for all $X \in \cC$, the map $X \to \ast$ is a cover,
\item \label{axC2} pullbacks of covers exist in $\cC$, and $\cS$ is closed under base change,
\item if $f$ and $ g\circ f $ are in $\cS$, then so is $g$,
\item every cover is an effective epimorphism.
\end{enumerate}
 
In addition to the examples given in \cite[Sec.\ 3.3]{RZ}, \cite[Sec.\ 2]{W},
\cite[Sec.\ 2]{W2}, the following simple variation will be useful: 
\begin{example}\label{ex:CWC}
Let $(\cC,\cS)$ be a category with covers, $A \in \cC$, and $\cC_A:={}^{A/}\cC$ 
the slice category under $A$. Then the forgetful functor $\cC_A \to \cC$ reflects effective epimorphisms, and therefore induces the structure of a category of covers on $\cC_A$ in the obvious way. Clearly, the analogous result holds for descent categories as well.   
\end{example}

Recall, e.g.\ \cite{W}, that if $\cC$ is a category with covers, then, by using the same definitions as in the descent category case \eqref{def:kan}, one may speak of \df{$\infty$-groupoids}, \df{(covering) Kan fibrations}, and \df{hypercovers} in $s\cC$. There is, however, an additional implicit requirement that the specific limits appearing in the definitions must exist in $\cC$.

\begin{remark}
The class of covering fibrations in $\Gpd_{\infty}(\cC)$ includes all hypercovers (hence all isomorphisms) and terminal morphisms $X_\bl \to \ast$. 
\end{remark}

\begin{lemma}\label{lem:life-saver}
Let $(\cC,\cS)$ be a category with covers. Let $f \maps X_\bl \to Z_{\bl}$ be any morphism in $\Gpd_{\infty}(\cC)$ and $g \maps Y_{\bl} \fib Z_{\bl}$  a Kan fibration. 
\begin{enumerate}[leftmargin=15pt]
\item 
The pullback of the diagram in $\Gpd_\infty(\cC)$
\begin{equation} \label{diag:life-saver}
\begin{tikzdiag}{2}{2}
{
X_{\bl}\& Z_{\bl} \& Y_{\bl}\\
};
\path[->,font=\scriptsize]
(m-1-1) edge node[auto] {$f$} (m-1-2)
;
\path[->>,font=\scriptsize]
(m-1-3) edge node[auto,swap] {$g$} (m-1-2)
;
\end{tikzdiag}
\end{equation}
exists if and only if the pullback of $X_0 \xto{f_0} Z_{0} \xleftarrow{g_0} Y_0$ exists in $\cC$.  

\item \label{it:life-saver2} If the pullback of \eqref{diag:life-saver} exists, then the induced morphism
$\ti{g} \maps X_{\bl} \times_{Z_{\bl}} Y_{\bl} \to X_{\bl}$ is a Kan fibration.

\item 
If $g \maps Y_\bl \fib Z_{\bl}$ is a covering fibration, then the pullback of the diagram  \eqref{diag:life-saver} exists and $\ti{g}$ is a covering fibration.

\item \label{it:life-saver4} If $g \maps Y_\bl \fib Z_{\bl}$ is a hypercover, then the pullback of the diagram  \eqref{diag:life-saver} exists and $\ti{g}$ is a hypercover.

\end{enumerate}
\end{lemma}
\begin{proof}
\begin{enumerate}[leftmargin=15pt]
\item If the pullback of $X_0 \to Z_{0} \leftarrow Y_0$ exists in $\cC$, then, since $g$ is a Kan fibration, $X_{\bl} \times_{Z_{\bl}} Y_{\bl}$ exists in $\Gpd_\infty(\cC)$ by \cite[Thm.\ 2.17(4)]{W}. For the converse, suppose the pullback of
\eqref{diag:life-saver} exists. Recall that there is an adjunction 
\[
\const_{\bl}(-) \maps \cC \adjunct \Gpd_{\infty}(\cC) \maps \ev_0.
\]
The left adjoint sends $Z \in \cC$ to the constant $\infty$-groupoid $\const_{n}(Z):=Z$, while the right adjoint sends an $\infty$-groupoid $X_\bl$ to its object of zero simplices $\ev_0(X):=X_0$. Hence, the pullback of $X_0 \to Z_{0} \leftarrow Y_0$ exists.

\item Follows from (1) and \cite[Thm.\ 2.17(4)]{W}.

\item Since $g_0 \maps Y_0 \to Z_0$ is a cover, the statement follows from (1) and (2) above, along with Axiom \eqref{axC2} for a category with covers. 

\item This is \cite[Thm.\ 2.12]{W}.
\end{enumerate}
\end{proof}

We record for later use in Sec.\ \ref{sec:hmtpy} the following proposition concerning factorization of the diagonal. For the descent category case, the result is implied by \cite[Thm.\ 3.21]{BG}. For a category with covers, it follows from \cite[Prop.\ 7.2]{RZ}.
 
\begin{proposition}\label{prop:path-obj}
Let $X_\bl$ be an $\infty$-groupoid object in either a descent category or a category with covers. 
\begin{enumerate}[leftmargin=15pt]
\item \label{it:path-1} For each $n \geq 0$, the limit $\hom(\Delta^n \times \Delta^1,X_\bl)$ exists and is canonically isomorphic to the iterated pullback of covers
\[
X_{n+1} ~ \prescript{}{d_1}\times_{d_1} ~ X_{n+1}
~ \prescript{}{d_2}\times_{d_2} \cdots
\prescript{}{d_i}\times_{d_{i}} ~ X_{n+1} ~ \prescript{}{d_{i+1}}
\times_{d_{i+1}} \cdots \prescript{}{d_n} \times_{d_n} ~ X_{n+1}.
\]
The resulting simplicial object $X^I_\bl:=\hom(\Delta^\bl \times \Delta^1,X_\bl)$ is an $\infty$-groupoid.
\item \label{it:path-2}  The cofaces $\Del^0 \rightrightarrows \Del^1$ and codegeneracy $\Delta^1 \to \Del^0$ induce canonical morphisms
\[
X_{\bl} \xto{s^0_X} X^I_\bl \xto{(d^0_X,d^1_X)} X_\bl \times X_\bl,
\]  
natural in $X_\bl$, whose composition is the diagonal of $X_\bl$.
\end{enumerate}
\end{proposition}

As mentioned in the introduction, categories with covers are closely related to the notion of descent category. See \cite[Rmk.\ 5.5]{Lie3} for further discussion on this point. In particular:
\begin{lemma}[Lemma 5.6 \cite{Lie3}]\label{lem:CWC-descent}
Let $(\cD,\cS)$ be a descent category, and let $\cC \sse \cD$ denote the full subcategory of objects $X$ for which $X \to \ast$ is a cover. Then $(\cC, \cS \cap \cC)$ is a category with covers. 
\end{lemma}

Let $\cS^{\mathrm{mfd}}_{fs}:=\conilsm \cap \cS_{fs}$ denote the subcategory whose morphisms are formal submersions between pointed formal manifolds. Then $(\conilsm, \cS^{\mathrm{mfd}}_{fs})$ is a category with covers by the above lemma. \begin{definition} \label{def:form-inf-gpd}
The category $\FGpd:=\Gpd_\infty(\conilsm)$ of \df{pointed formal $\infty$-groupoids} over $\kk$ is the full subcategory of $\sconil$ consisting of $\infty$-groupoid objects in pointed formal manifolds. The category $\FG \sse \FGpd$ of \df{formal $\infty$-groups} over $\kk$ is the full subcategory of reduced formal $\infty$-groupoids, i.e.\ $C_0 = \kk$.
\end{definition}

\begin{example}\label{ex:WG}
Let $H_\bl$ be a group object in $\sconil$. Hence, $H_\bl$ is a simplicial (connected) cocommutative Hopf algebra. Let $\Wbar{H} \in \sconil$
be the classifying space of $H_\bl$ built via the $\Wb$-construction\footnote{We recall the precise definition in Sec.\ \ref{sec:Wbar}. See also \cite[Sec.\ 2.3]{Priddy:1970}.} of Eilenberg and Mac Lane \cite{EM}. By the Milnor-Moore theorem, $H_\bl$ is isomorphic to the enveloping algebra $\cU(\g_\bl)$ of its simplicial Lie algebra of primitives $\g_\bl = \prim_\bl(H)$. Furthermore, for each $n \geq 0$, the PBW theorem provides an isomorphism of coalgebras  $H_n \cong \coSym(\g_n)$. Hence, $H_\bl$ is group object $s\conilsm$. It then follows from \cite[Sec.\ 2.3; Rmk.\ 5.5]{Lie3} that $\Wbar{H}$ is a formal $\infty$-group.   
\end{example}

\subsubsection{Geometric functors} \label{sec:geo-func}
Suppose $i \maps (\cC, \cS \cap \cC) \to (\cD, \cS)$ is the inclusion of a category of covers into a descent category as in Lem.\ \ref{lem:CWC-descent}. By comparing the respective axioms, it is clear that $i$ preserves covers, pullbacks along covers, and, in particular, all products. It will be useful for us to axiomatize functors with such properties, in rough analogy with morphisms between sites in topos theory.

\begin{definition} \label{def:geo-func}
Let $(\cC,\cS_{\cC})$ be a category with covers, and $(\cD,\cS_{\cD})$ a descent category. A functor $F \maps \cC \to \cD$ is \df{geometric} if $F$ sends covers to covers, and preserves pullbacks of covers. If, in addition, $F(f) \in \SD$ implies $f \in \SC$, then we say $F$ \df{reflects covers}.
\end{definition}
Note that axioms \ref{axC1} and \ref{axC2} above imply that a geometric functor preserves finite products.

\begin{proposition} \label{prop:geo-func} 
Let $(\cC,\cS_{\cC})$ be a category with covers, $(\cD,\cS_{\cD})$ a descent category, and  $F \maps \cC \to \cD$ a geometric functor. Let $F_\bl \maps s\cC \to s\cD$ denote the induced functor between simplicial objects.
\begin{enumerate}
\item \label{it:gf-1} Restriction induces a functor $F_\bl \maps \Gpd_\infty(\cC) \to \Grp_\infty(\cD)$ between categories of $\infty$-groupoid objects which sends Kan fibrations to Kan fibrations and covering fibrations to covering fibrations. Furthermore, if $F$ reflects covers, then $F_\bl$ reflects Kan fibrations and covering fibrations. 

\item \label{it:gf-2} $F_\bl \maps \Gpd_\infty(\cC) \to \Grp_\infty(\cD)$ preserves hypercovers.
Furthermore, if $F$ reflects covers, then $F_\bl$ reflects hypercovers. 

\item 
$F_\bl \maps \Gpd_\infty(\cC) \to \Grp_\infty(\cD)$ preserves pullbacks along covering fibrations.

\item \label{it:grp-obj} Restriction induces functors $F \maps \Grp(\cC) \to \Grp(\cD)$ 
and $F_\bl \maps \Grp(s\cC) \to \Grp(s\cD)$ between categories of group objects. 
\item 
There is a natural isomorphism $F_\bl \cc \Wbar(-) \cong \Wbar(-) \cc F_\bl$ between functors from $\Grp(s\cC)$ to $\Gpd_\infty(\cD)$, where $\Wbar(-)$ denotes the classical $\ov{\cW}$-construction \cite{EM}. 
\end{enumerate}
\end{proposition}

To prove the proposition, we will use the following lemmas.
\begin{lemma}\label{lem:mono-collapse}
Let $F \maps (\cC, \SC) \to (\cD,\SD)$ be a geometric functor, and $X_\bl \in \Gpd_{\infty}(\cC)$. Let $S$ be a finitely generated simplicial set of dimension $d \geq  0$ that is  ``monotonically collapsible'', i.e.\ $S$ admits a filtration of the form
\[
\Del^0= S_0 \sse S_1 \sse \cdots \sse S_m = S,
\]  
for which there exists an increasing sequence $0=n_0 \leq n_1 \leq n_2 \leq \cdots \leq n_m \leq d$ such that for all $i \geq 1$ we have $S_i = S_{i-1} \bigcup_{\Lam^{n_i}_{k_i}} \Del^{n_i}$ for some $0 \leq k_i \leq n_i$. Then there is an isomorphism $F\bigl(X_\bl^S \bigr) \cong F_\bl(X)^S$ natural in $X_\bl$. 
\end{lemma}

\begin{proof}
First, we observe that any horn $\Lam^{\el}_k$ is a monotonically collapsible simplicial set of dimension $ \el-1$. 
Given $S$ as in the statement, we proceed by induction on the dimension $d=\dim S$. For $d=0$, we have $S= \coprod\Del^0$, and the statement follows, since a geometric functor preserves products. Now suppose $d \geq 1$, and that the statement holds for all monotonically collapsible $S'$ of dimension $< d$. In particular, since $n_i-1 <  d$ for all $i$, the induction hypothesis implies that $F(X_{n_i,k_i}) \cong F(X)_{n_i,k_i}$. 

Let $r \geq 1$ be the largest $r$ such that $n_{r-1} < d$. Then $F(X^{S_{r-1}}) \cong F(X)^{S_{r-1}}$.
We then verify recursively that $F(X^{S_{\el}}) \cong F(X)^{S_{\el}}$ for $r \leq \el \leq m$ in the following way. Assuming $F(X^{S_{\el-1}}) \cong F(X)^{S_{\el-1}}$ observe that the pullback diagram 
\[
\begin{tikzdiag}{3}{3}
{
{X^{S_{\el}}} \& X_{n_{\el}}  \& \\
{X^{S_{\el-1}}}\& X_{n_{\el}, k_{\el}} \& \\
};

\path[->,font=\scriptsize]
(m-1-1.-10) edge node[auto] {$$} (m-1-2)
(m-1-1) edge node[auto,swap] {$$} (m-2-1)
(m-1-2) edge node[auto,swap] {$$} (m-2-2)
(m-2-1.-8) edge node[auto,swap] {$$} (m-2-2)
;
\begin{scope}[shift=($(m-1-1)!.4!(m-2-2)$)]
\draw +(-0.25,0) -- +(0,0) -- + (0,0.25);
\end{scope}
\end{tikzdiag}
\]
exists in $\cC$, since the right vertical morphism is a cover. Furthermore, since $F$ preserves pullbacks of covers, we have
\[
F(X^{S_{\el}}) \cong F(X^{S_{\el-1}}) \times_{F(X_{n_{\el}, k_{\el}})}F(X_{n_\el}) 
\cong F(X)^{S_{\el-1}} \times_{F(X)_{n_{\el}, k_{\el}}}F(X)_{n_\el}. 
\]
On the other hand, $\cD$ has all finite limits, so the functor $\hom(-,F_\bl(X))$ is left exact. Hence,
\[
F(X)^{S_\el} \cong \hom\Bigl( \bigl( S_{\el -1} \cup_{\Lam^{n_\el}_{k_\el}} \Del^{n_\el}\bigr), F_\bl(X) \Bigr) \cong F(X)^{S_{\el-1}} \times_{F(X)_{n_{\el}, k_{\el}}}F(X)_{n_\el},
\]
and this completes the proof.
\end{proof}

\begin{lemma}\label{lem:represent}
Let $F \maps (\cC, \SC) \to (\cD,\SD)$ be a geometric functor and $f \maps X_\bl \to Y_\bl$ a morphism in $\Gpd_\infty(\cC)$. 
Let $m \geq 0$. Suppose $ S \emb T$ is a monomorphism of finite simplicial sets with $\dim S =m$, and   
for all $n \leq m$, $f$  satisfies the condition $\acyc(n)$ in Def.\ \ref{def:kan}.
\begin{enumerate}
\item If $Y^{T}$ exists in $\cC$, then the limit $X^{S} \times_{Y^{S}} Y^{T}$ exists in $\cC$.
\item If $Y^{T}$ exists in $\cC$ and $F(Y^{T}) \cong F_{\bl}(Y)^{T}$ is a natural isomorphism, then there is a natural isomorphism  $F \bigl(X^{S} \times_{Y^{S}} Y^{T}) \cong F_\bl(X)^{S} \times_{F_\bl(Y)^{S}} F_\bl(Y)^{T}$.
\end{enumerate} 
\end{lemma}
\begin{proof}
The first statement is \cite[Lemma 2.10]{W}. For the second statement, proceed via induction exactly as in the proof of 
\cite[Lemma 2.10]{W} using the fact that $F$ preserves covers and pullbacks along covers. 
\end{proof}

\begin{proof}[{\bf Proof of Proposition \ref{prop:geo-func}}]
\mbox{}
\begin{enumerate}
\item The first assertion follows directly from Lemma \ref{lem:mono-collapse}, since
$F$ sends covers to covers and every horn is monotonically collapsible. More generally, if $f \maps X_\bl \to Y_\bl$ is any morphism in $\Gpd_\infty(\cC)$ then, for all $n \geq 1$ and $ 0 \leq k \leq n$,
$X_{n,k} \times_{Y_{n,k}} Y_n$ is a pullback along the cover $Y_n \to Y_{n,k}$ and therefore is preserved by $F$. Combining this with the first assertion, we conclude $F\big(X_{n,k} \times_{Y_{n,k}} Y_n\bigr) \cong F(X)_{n,k} \times_{F(Y)_{n,k}} F(Y)_n$. Hence, if $f$ is a Kan fibration, then $F_\bl(f)$ is as well, since $F$ preserves covers. Moreover, if $F$ reflects covers and $F_\bl(f)$ is a Kan fibration, then the same argument implies that $f$ is Kan as well. The analogous statements for covering fibrations clearly hold. 

\item Follows from Lemma \ref{lem:represent}. Indeed, suppose $f \maps X_\bl \to Y_\bl$ is a hypercover in $\Gpd_\infty(\cC)$. Proceed by induction. For the base case, $f_0$ being a cover implies $F(f_0) \maps F(X)_0 \to F(Y)_0$ is a cover. Let $m >0$ and suppose for all $n \leq m$ $F(f)$ satisfies $\acyc(n)$ in Def.\ \ref{def:kan}. Set $S= \pa \Del^{m+1}$ and $T=\Del^{m+1}$ in Lemma \ref{lem:represent}. Clearly $F(Y_{m+1})=F(Y)_{m+1}$. Therefore, by the lemma there is a commutative diagram  
\begin{equation} \label{diag:hyper-induct}
\begin{tikzdiag}{2}{3}
{
F(X_{m+1})\&  \&F\bigl(X^{\pa \Del^{m+1}} \times_{Y^{\pa \Del^{m+1}}} Y_{m+1})  \\
F(X)_{m+1}\&  \& F(X)^{\pa \Del^{m+1}} \times_{F(Y)^{\pa \Del^{m+1}}} F(Y)_{m+1} \\
};
\path[->,font=\scriptsize]
(m-1-1) edge node[auto] {$F(\mu_{m+1}(f))$} (m-1-3)
(m-1-1) edge node[auto,sloped,swap] {$=$} (m-2-1)
(m-1-3) edge node[auto,sloped] {$\cong$} (m-2-3)
(m-2-1) edge node[auto] {$\mu_{m+1}(F(f))$} (m-2-3)
;
\end{tikzdiag}
\end{equation}
Since $F(\mu_{m+1}(f))$ is a cover, $\mu_{m+1}(F(f))$ is a cover. This completes the inductive step. Hence, $F(f)$ is a hypercover.

Next, assume $F$ reflects covers and $f \maps X_\bl \to Y_{\bl}$ is a morphism such that $F(f)$ is a hypercover. Hence $f \maps X_0 \to Y_0$ is a cover. Again proceed by induction. Let $m >0$ and suppose for all $n \leq m$ $f$ satisfies $\acyc(n)$.  
Set $S= \pa \Del^{m+1}$ and $T=\Del^{m+1}$ in Lemma \ref{lem:represent}. Then $X^{\pa \Del^{m+1}} \times_{Y^{\pa \Del^{m+1}}} Y_{m+1}$ exists by the first statement of the Lemma, and the second statement implies that the above diagram \eqref{diag:hyper-induct} exists.  
Since $\mu_{m+1}(F(f))$ is a cover, $F(\mu_{m+1}(f))$ is a cover, and hence $\mu_{m+1}(f)$ is as well.

\item If $f \maps X_\bl \to Y_\bl$ is a covering fibration, then \cite[Lem. 2.10]{Lie3} implies that $f_n \maps X_n \to Y_n$ is a cover for all $n \geq 0$. Therefore, since $F$ is a geometric functor, $F_\bl$ preserves the pullback of $f$ along any morphism.

\item Since $X \to \ast$ is a cover for all $X \in \cC$, $F$ preserves all finite products, and therefore group objects. Hence, $F_\bl$ does as well.

\item This follows from statement (3), along with the fact that $F$ preserves products, and the explicit description (e.g.\ \cite[Sec.\ 2.3]{Lie3}) for $\Wbar(G)$  when $G_\bl$ is a simplicial group object in a category with covers.
\end{enumerate}
\end{proof}

Fix $n \in \N$ and an $\infty$-groupoid object $X_\bl$ in either a descent category or a category with covers. Recall that $X$ is an \df{$n$-groupoid} if the covers
\[
X_m \to X^{\Lambda^{m}_k}
\]
are isomorphisms for all $m > n$ and for all $0 \leq k \leq  m$. Prop.\ \ref{prop:geo-func} immediately implies the following corollary.
\begin{corollary}\label{cor:n-grpd}
Let $F \maps \cC \to \cD$ be a geometric functor as in Prop.\ \ref{prop:geo-func}, and $F_\bl \maps \Gpd_\infty(\cC) \to \Gpd_\infty(\cD)$ the induced functor between categories of $\infty$-groupoids. If $X_\bl$ is an $n$-groupoid in $\cC$, then $F_\bl(X)$ is an $n$-groupoid in $\cD$.
\end{corollary}

\subsubsection{Examples of geometric functors} \label{sec:geo-func-ex} 
We conclude this section with examples of geometric functors, and some corollaries. Let $ \cS_{ss} \sse \Mfd$ denote the subcategory of finite dimensional manifolds whose morphisms consist of surjective submersions. Recall that $(\Mfd,\cS_{ss})$ is a category of covers. From Example \ref{ex:CWC}, it then follows that the category of pointed manifolds $(\Mfd_\ast,{\cS_\ast}_{ss})$ is also a category of covers.  Finally, let $\cS_{epi} \sse \FDVect$ denote the subcategory of finite-dimensional vector spaces whose morphisms are surjections. Then $(\FDVect,\cS_{epi})$ is a descent category. 

\begin{corollary} \label{cor:geo-func-ex}
\mbox{}
\begin{enumerate}
\item The inclusion $(\conilsm, \cS^{\mathrm{mfd}}_{fs}) \to (\conil,\cS_{fs})$, and the forgetful functor $(\Mfd_\ast,{\cS_\ast}_{ss}) \to (\Mfd,\cS_{ss})$ are geometric functors that reflect covers.

\item \label{it:dist} The assignment \eqref{eq:dist} of a pointed manifold to its coalgebra of point distributions  
defines a geometric functor $\Dist \maps (\Mfd_\ast,{\cS_\ast}_{ss}) \to (\conil,\cS_{fs})$ 
\item \label{it:prim} The assignment \eqref{eq:prim-adj} of a formal manifold to its space of primitives defines a geometric functor $\prim \maps (\conilsm, \cS^{\mathrm{mfd}}) \to (\FDVect, \cS_{epi})$ that reflects covers.
\end{enumerate}
\end{corollary}
\begin{proof}
\mbox{}
\begin{enumerate}
\item Obvious
\item The functor $\Dist$ preserves covers as shown in Example \ref{ex:mfd-fsub}.
It remains to verify that it preserves pullbacks of pointed surjective submersions.
Let $X,Y,Z$ denote pointed manifolds, $Y \xto{\ph} Z$ a smooth pointed map, and $X \fib Z$ a submersion at the base point. We wish to show that 
the canonical map in $\clnAlg$ induces an isomorphism  $\hA_X \htensor_{\hA_Z} \hA_Y \cong \hA_P$, where $P$ is the pullback of the diagram $X \to Z \xleftarrow{\ph} Y$. 

It suffices to verify the statement for the local model. Let $X=\R^{d + d_Z}$, $Z=\R^{d_Z}$, and $Y=\R^{d_Y}$.
Let $\ph \maps Y \to Z$ be a smooth function with coordinates $\ph_i(\vv{y}) = \pr^Z_i \circ \phi(y_1,\ldots, y_{d_Y})$, $i=1,\ldots,d_Z$, 
and consider the pullback diagram
\[
\begin{tikzdiag}{2}{3}
{
P\& Y  \\
X \& Z  \\
};

\path[->,font=\scriptsize]
(m-1-1) edge node[auto,swap] {$\ti{\ph}$} (m-2-1)
(m-1-1) edge node[auto] {$\pr^Y$} (m-1-2)
(m-1-2) edge node[auto] {$\ph$} (m-2-2)
(m-2-1) edge node[auto] {$\pr^Z$} (m-2-2)
;
\begin{scope}[shift=($(m-1-1)!.4!(m-2-2)$)]
\draw +(-0.25,0) -- +(0,0) -- + (0,0.25);
\end{scope}
\end{tikzdiag}
\]
where $P = \R^{d + d_Y}$, and 
$\ti{\ph}(\vv{x},\vv{y}) =(x_1,\ldots,x_d, \ph_1(\vv{y}),\ldots, \ph_{d_Z}(\vv{y}))$. The universal property of the tensor product in $\clnAlg$ yields the unique morphism
\[
\hA_X \htensor_{\hA_Z} \hA_Y \xto{\psi} \hA_P .
\]
Hadamard's Lemma gives the equalities
\[
\begin{split}
\hA_X = \R[[x_1,\ldots,x_{d},z_1,\ldots,z_{d_Z}]], & \quad  \hA_Y = \R[[y_1,\ldots,y_{d_Y}]],\\
\hA_Z = \R[[z_1,\ldots,z_{d_Z}]], & \quad  \hA_P  = \R[[x_1,\ldots,x_d,y_1,\ldots,y_{d_Y}]].
\end{split}
\]
In terms of these identifications, $\psi$ is the identity on the generators $x_1 \htensor 1,\ldots,x_d \htensor 1$ and $1 \htensor y_1, \ldots, 1\htensor y_{d_Y}$. Furthermore, $\psi$ maps the generator $z_i$ to the $\infty$-jet of $\ph_i$ at $0$. Via Example \ref{ex:powerseries}, we obtain the natural isomorphisms
\[
\begin{split}
\hA_X \htensor_{\hA_Z} \hA_Y & \cong \bigl( \R[[x_1, \ldots, x_d]] \htensor_\R \hA_Z \bigr) \htensor_{\hA_Z} \hA_{Y}\\
& \cong \R[[x_1, \ldots, x_d]] \htensor_\R \hA_{Y} \cong \R[[x_1\ldots,x_d, y_1,\ldots,y_{d_Y}]].
\end{split}
\]
Hence, $\psi$ is also an isomorphism.

\item Follows from the second and fourth statements of Prop.\ \ref{prop:descent}.
\end{enumerate}
\end{proof}

\subsection{Homotopy theory of $\infty$-groupoids} \label{sec:hmtpy}
We recall the definition of a category of fibrant objects for a homotopy theory \cite{Brown}. Let  $\cC$ be a category with finite products, with terminal object $\ast \in \cC$, which is equipped with two wide subcategories:
\df{weak equivalences} and \df{fibrations}. A morphism which
is both a weak equivalence and a fibration is called an
\df{acyclic fibration}. Then $\cC$ is a
\df{category of fibrant objects (CFO)} if:
\begin{enumerate}
\item{Every isomorphism in $\cC$ is an acyclic fibration.}

\item{The class of weak equivalences satisfies  ``2 out of 3''. That is, if
    $f$ and $g$ are composable morphisms in $\cC$ and any two of $f,g, g
    \circ f$ are weak equivalences, then so is the third.}

\item{The pullback of a fibration/acyclic fibration exists, and is a fibration/acyclic fibration.
That is, if $Y \xto{g} Z \xleftarrow{f} X$ is a diagram in $\cC$ with $f$
    a fibration/acyclic fibration, then the pullback $X \times_{Z} Y$ exists, and
   the induced projection $X \times_{Z} Y \to Y$ is a  fibration/acyclic fibration.}

\item{For any object $X \in \cC$ there exists a (not necessarily
    functorial) \df{path object}, that is, an object
    $X^{I}$ equipped with morphisms
\[
X \xto{s} X^{I} \xto{(d_0,d_1)} X \times X,
\]
such that $s$ is a weak equivalence, $(d_0,d_1)$ is a fibration, and their
composite is the diagonal map.}
\item{All objects of $\cC$ are ``fibrant''. That is, for any $X \in \cC$ the unique map $ X \to \ast$ is a fibration.}
\end{enumerate}

Also, we recall that a functor $F \maps \cat{C} \to \cat{D}$ between categories of fibrant objects is \df{exact} if $F$ preserves the terminal object, fibrations, and acyclic fibrations, and $F$ maps any pullback square in $\cat{C}$ of the form
\[
\begin{tikzpicture}[descr/.style={fill=white,inner sep=2.5pt},baseline=(current  bounding  box.center)]
\matrix (m) [matrix of math nodes, row sep=2em,column sep=3em,
  ampersand replacement=\&]
  {  
P \& X \\
Z \& Y \\
}
; 
  \path[->,font=\scriptsize] 
   (m-1-1) edge node[auto] {$$} (m-1-2)
   (m-1-1) edge node[auto,swap] {$$} (m-2-1)
   (m-2-1) edge node[auto] {$$} (m-2-2)
  ;
  \path[->>,font=\scriptsize] 
   (m-1-2) edge node[auto] {$f$} (m-2-2)
;
  \begin{scope}[shift=($(m-1-1)!.4!(m-2-2)$)]
  \draw +(-0.25,0) -- +(0,0)  -- +(0,0.25);
  \end{scope}
\end{tikzpicture}
\]
in which $f$ is a fibration, to a pullback square in $\cat{D}$. It follows from these axioms that exact functors preserve finite products and weak equivalences.

\begin{remark}\label{rmk:no-weq}
In a CFO $\cC$ with functorial path objects $\Path(-) \maps \cC \to \cC$, the weak equivalences are uniquely determined by the acyclic fibrations appearing in the functorial factorization. This follows from Brown's factorization lemma, as emphasized in \cite[Def.\ 3.20]{BG}. Indeed, given a morphism $f \maps X \to Y$ we obtain a functorial span of fibrations
\begin{equation} \label{eq:factor}
\begin{tikzdiag}{2}{2}
{
\& X \times_Y \Path(Y) \& \\
X \&  \& Y \\
};

\path[->>,font=\scriptsize]
(m-1-2) edge node[auto] {$p_f$} (m-2-3)
(m-1-2) edge node[above,sloped,swap] {$\sim$} (m-2-1)
;
\path[->,font=\scriptsize]
(m-2-1) edge[bend left=30] node[auto] {$i_f$} node[auto,sloped, below]{$\sim$} (m-1-2)
;
\end{tikzdiag}
\end{equation}
in which the left leg is acyclic with a canonical right inverse $i_f$ such that $f=p_f \cc i_f$. 
Hence, $f$ is a weak equivalence if and only if $p_f$ is also acyclic. 
\end{remark}

The homotopy theory of $\infty$-groupoid objects in a descent category is characterized by a CFO structure. We call this the \df{canonical CFO structure on $\infty$-groupoids}.
\begin{theorem}[Thm.\ 3.6 \cite{BG}]\label{thm:descent-CFO}
Let $\cD$ be a descent category. The category $\Gpd_\infty(\cD)$ admits the structure of a category of fibrant objects with  functorial path objects in which the fibrations are the Kan fibrations, and the acyclic fibrations are the hypercovers.
\end{theorem}

\begin{remark}\label{rmk:path-obj}
The functorial path objects appearing in Thm.\ \ref{thm:descent-CFO} are constructed via Prop.\ \ref{prop:path-obj}. 
\end{remark}

Turning now to the $\infty$-groupoid objects in a category with covers,  a slightly weaker homotopical structure is needed, due to the lack of finite limits.

\begin{definition}[Def.\ 2.1 \cite{RZ}] \label{def:icfo}
Let $\cC$ be a category with finite products and terminal object $\ast
\in \cC$ equipped with two wide subcategories: 
\df{weak equivalences} and \df{fibrations}. A morphism which
is both a weak equivalence and a fibration is called an
\df{acyclic fibration}. We say $\cC$ is an
{\bf incomplete category of fibrant objects (iCFO)} if:
\begin{enumerate}
\item\label{it:icfo-1}{Every isomorphism in $\cC$ is an acyclic fibration.}

\item\label{it:icfo-2}{The class of weak equivalences satisfies  ``2 out of 3''.}

\item \label{it:icfo-3}{If the pullback of a fibration {\it exists}, then it is a fibration.}

\item\label{it:icfo-4}{The pullback of {\it any} acyclic fibration exists, and is an acyclic fibration.}

\item\label{it:icfo-5}{For any object $X \in \cC$ there exists a (not necessarily
    functorial) path object. }

\item\label{it:icfo-6}{For all $X \in \cC$ the unique map $ X \to \ast$ is a fibration.}
\end{enumerate}
\end{definition}

\begin{remark}\label{rmk:icfo-good}
An iCFO retains many of the good homotopical properties of a CFO. In particular, when the iCFO has functorial path objects, the factorization properties 
described in Rmk.\ \ref{rmk:no-weq} for CFOs also hold for in an iCFO.  We refer the reader to \cite[Sec.\ 2]{RZ} for additional details.
\end{remark}

\begin{definition}[Def.\ 2.8 \cite{RZ}] \label{def:iCFO-exact}
Let $\cC$ be a iCFO and $S \sse \cC$ a subcategory of fibrations containing all 
acyclic fibrations and morphisms to the terminal object. A functor $F \maps \cC \to \cC'$ between iCFOs is \df{exact with respect to $S$} if: 
\begin{enumerate}
\item $F$ preserves the terminal object and acyclic fibrations, 
\item $F$ maps every fibration in $S$ to a fibration in $\cC'$, and
\item any pullback square in $\cC$ of the form
\[
\begin{tikzpicture}[descr/.style={fill=white,inner sep=2.5pt},baseline=(current  bounding  box.center)]
\matrix (m) [matrix of math nodes, row sep=2em,column sep=3em,
  ampersand replacement=\&]
  {  
P \& X \\
Z \& Y \\
}
; 
  \path[->,font=\scriptsize] 
   (m-1-1) edge node[auto] {$$} (m-1-2)
   (m-1-1) edge node[auto,swap] {$$} (m-2-1)
   (m-1-2) edge node[auto] {$f$} (m-2-2)
   (m-2-1) edge node[auto] {$$} (m-2-2)
  ;

  \begin{scope}[shift=($(m-1-1)!.4!(m-2-2)$)]
  \draw +(-0.25,0) -- +(0,0)  -- +(0,0.25);
  \end{scope}
\end{tikzpicture}
\]
with $f$ a fibration in $S$, is mapped by $F$ to a pullback square in $\cC'$.
\end{enumerate}
For the case when $S$ consists of all fibrations in $\cC$,  we simply say $F$ is \df{exact}.
\end{definition}
As in the CFO case, functors between iCFOs which are exact with respect to a class of fibrations necessarily preserve finite products and weak equivalences \cite[Lem.\ 2.9]{RZ}.

\subsubsection{iCFOs and exact functors via geometric functors} \label{sec:icfo-geo-func} 

The framework of geometric functors \eqref{def:geo-func} provides a useful construction of iCFOs and exact functors. In what follows, we equip the category of $\infty$-groupoid objects in a descent category $\cD$ with the CFO structure given in Thm.\ \ref{thm:descent-CFO}.

\begin{proposition}\label{prop:icfo-exist}
Let $F \maps (\cC,\SC) \to (\cD,\SD)$ be a geometric functor that reflects covers. Define $f \maps X_\bl \to Y_\bl$ in $\Gpd_\infty(\cC)$ to be a weak equivalence if and only if $F_\bl(f)$ is a weak equivalence in $\Gpd_{\infty}(\cD)$.
Define $f \maps X_\bl \to Y_\bl$ to be a fibration if and only if it is Kan fibration. Then the category $\Gpd_\infty(\cC)$ admits the structure of an iCFO 
with weak equivalences and fibrations as above, and with functorial path objects given by Prop.\ \ref{prop:path-obj}.  Moreover, the acyclic fibrations in $\Gpd_\infty(\cC)$ are exactly the hypercovers.
\end{proposition}
\begin{proof}
We begin by proving the last assertion. Statements \eqref{it:gf-1} and \eqref{it:gf-2} of Prop.\ \ref{prop:geo-func} imply that $F_\bl \maps \Gpd_\infty(\cC) \to \Gpd_\infty(\cD)$ preserves and reflects Kan fibrations and hypercovers. Hence, $f \maps X_\bl \to Y_\bl$ is both a Kan fibration and weak equivalence if and only if $F_\bl(f)$ is a hypercover if and only if $f$ is a hypercover. 

Axioms \eqref{it:icfo-3} and \eqref{it:icfo-4} of Def.\ \ref{def:icfo} follow
from statements \eqref{it:life-saver2} and \eqref{it:life-saver4} of Lemma \ref{lem:life-saver}.  It remains to verify Axiom \eqref{it:icfo-5}, the existence of path objects. 
Let $X_\bl \in \Gpd_{\infty}(\cC)$. Since $F$ preserves pullbacks along covers, statement \eqref{it:path-1} of Prop.\ \ref{prop:path-obj} implies 
\[
F\bigl(\hom(\Delta^n \times \Delta^1,X_\bl) \bigr) \cong  \hom(\Delta^n \times \Delta^1,F_\bl(X)),
\]
for all $n \geq 0$. The naturality of the morphisms $s^0_{(-)}$, $d^0_{(-)}$, and $d^1_{(-)}$ defined in statement \eqref{it:path-2} of Prop.\ \ref{prop:path-obj} imply that $F_\bl(s^0_{X})$ and $F_\bl(d^0_{X},d^1_{X})$ differ from $s^0_{F_\bl(X)}$ and $(d^0_{F_\bl(X)},d^1_{F_\bl(X)})$, respectively, via conjugation by isomorphisms. By Rmk.\ \ref{rmk:path-obj}, $F_\bl(X)^{I}$ is a path object for $F_\bl(X)$. Therefore, $s^0_{F_\bl(X)}$ is a weak equivalence, and $(d^0_{F_\bl(X)},d^1_{F_\bl(X)})$ is a Kan fibration. Hence, $s^0_{X}$ is weak equivalence in $\Gpd_\infty(\cC)$, by definition, and $(d^0_{X},d^1_{X})$ is a Kan fibration, by Prop.\ \ref{prop:geo-func}.
\end{proof}

In analogy with the descent category context, if $\cC$ is a category with covers, we call the iCFO structure featured in Prop.\ \ref{prop:icfo-exist} -- whenever it exists -- the \df{canonical iCFO structure on $\infty$-groupoids}.

\begin{remark}\label{rmk:no-weq-gpd}
The canonical iCFO structure, if it exists, is uniquely determined and independent of the geometric functor $F \maps \cC \to \cD$ used in Prop.\ \ref{prop:icfo-exist}. In more detail,
it follows from Rmk.\ \ref{rmk:icfo-good} and Rmk.\ \ref{rmk:no-weq} that
the weak equivalences for the canonical iCFO structure on $\Gpd_\infty(\cC)$ 
are completely determined by the hypercovers. Hence, the functor $F$ is no longer needed once existence has been established.
\end{remark}

\begin{proposition} \label{prop:icfo-exact-geo}
Let $(\cC,\SC)$ be a category with covers. Suppose that the category $\Gpd_\infty(\cC)$ admits the canonical iCFO structure. Let $F \maps (\cC,\SC) \to (\cD,\SD)$ be any geometric functor. 
\begin{enumerate}
\item \label{it:exact-cover} $F_\bl \maps \Gpd_\infty(\cC) \to \Gpd_\infty(\cD)$ is an exact functor with respect to the class of covering fibrations. 
Moreover, $F_\bl$ preserves all fibrations, and hence path objects.
\item\label{it:exact-reflect}  If $F$ reflects covers, then $F_\bl$ reflects all fibrations and weak equivalences.

\item \label{it:exact-exact} If $F$ preserves finite limits, then $F_\bl$ is exact. 
\end{enumerate}
\end{proposition}
\begin{proof}
\mbox{}
\begin{enumerate}

\item All assertions follow from Prop.\ \ref{prop:geo-func}.

\item If $F$ reflects covers, then $F_\bl$ reflects hypercovers. Hence, by applying the factorization \eqref{eq:factor} to a morphism $f$ in $\Gpd_\infty(\cC)$, we conclude that $F_\bl(f)$ is a weak equivalence if and only if $f$ is.

\item Clear.   
 
\end{enumerate}
\end{proof}

\subsubsection{The canonical iCFO structure on pointed Lie $\infty$-groupoids} \label{sec:canonical-iCFO}
The above framework provides a concise construction of the canonical iCFO structure on the category of pointed finite-dimensional Lie $\infty$-groups, i.e.\ $\Gpd_{\infty}(\Mfd_\ast)$. As a byproduct, we will show that the iCFO structure on Lie $\infty$-groupoids constructed ``by hand'' in \cite[Thm.\ 7.1]{RZ} also arises  via the more conceptual approach featured in Prop.\ \ref{prop:icfo-exist}.  

Since every cover in a category with covers $(\cC,\SC)$ is an effective epimorphism, the Yoneda embedding lands in the category of sheaves $\Sh(\cC)$ with respect to the Grothendieck topology induced by $\SC$. 
Recall that the subcategory of epimorphisms $\cS_{epi}$ gives $\Sh(\cC)$ the structure of a descent category. We begin by placing a descent category structure on $\Sh(\cC)$ which is more fine-grained than $(\Sh(\cC),\cS_{epi})$.

\begin{definition} \label{def:mixed}
Let $(\cC,\SC)$ be a category with covers. A morphism $\ph \maps F \to G$ in $\Sh(\cC)$ is a \df{mixed cover} if for all $X \in \cC$ and each point $X \to G$, there exists an epimorphism of sheaves $\ti{F} \xto{f} F$ such that the morphism $\pr_X$ in the pullback diagram
\[
\begin{tikzdiag}{2}{2}
{
X \times_G \ti{F}\& \ti{F}  \\
\&   F \\
X \&   G \\
};

\path[->,font=\scriptsize]
(m-1-1) edge node[auto] {$$} (m-1-2)
(m-1-1) edge node[auto,swap] {$\pr_X$} (m-3-1)
(m-1-2) edge node[auto] {$f$} (m-2-2)
(m-2-2) edge node[auto] {$\ph$} (m-3-2)
(m-3-1) edge node[auto] {$$} (m-3-2)
;
\begin{scope}[shift=($(m-1-1)!.4!(m-2-2)$)]
\draw +(-0.25,0) -- +(0,0) -- + (0,0.25);
\end{scope}
\end{tikzdiag}
\]  
is represented by a cover in $(\cC,\SC)$. We denote by $\cS_{\mix} \sse \Sh(\cC)$
the subcategory of mixed covers.
\end{definition}

\begin{proposition} \label{prop:mix-desc}
$(\Sh(\cC),\cS_{\mix})$ is a descent category, and the Yoneda embedding defines a geometric functor $\yo \maps (\cC,\SC) \to (\Sh(\cC),\cS_{\mix})$. 
\end{proposition}
\begin{proof}
Every cover in $\SC$ is a mixed cover between the corresponding representable functors. In particular, axiom \ref{ax2} of Def.\ \ref{def:descent} is satisfied. Furthermore, we observe that every mixed cover is an epimorphism. Combining this with the pasting lemma for pullback squares, we then deduce that axiom \ref{ax3} is satisfied. For axiom \ref{ax4}, if $f$ and $g\circ f$ are in $\cS_{\mix}$, then $g$ is as well, since $f$ is an epimorphism. Finally, axiom \ref{ax5} is satisfied, since every epimorphism in a topos is effective.    
\end{proof}

\begin{remark}\label{rmk:mix}
As already mentioned, every mixed cover is an epimorphism of sheaves, and every cover in $\SC$ is a mixed cover. However, in geometrically interesting cases, an epimorphism between representable sheaves need not be a cover in $\SC$, e.g.\ \cite[Rmk.\ 6.10]{RZ}. This explains our rationale for introducing the descent category $(\Sh(\cC),\cS_{\mix})$ and the terminology ``mixed''. Furthermore, every surjective submersion of sheaves in the sense of \cite[Def.\ 5.7]{Lie3} is a mixed cover, and every ``local stalkwise cover'' in the sense of \cite[Def.\ 6.4]{RZ} is a mixed cover.   
\end{remark}
\begin{question}
$\cS_\mix$ is the class of covers for a descent category which contains $\yo(\SC)$
and is contained in $\cS_{epi}$. It also satisfies left-cancellation with respect to all epimorphisms. What are the covers for the descent category that is universal with respect to these properties?
\end{question}

Recall from Sec.\ \ref{sec:geo-func-ex} that 
$(\Mfd,\cS_{ss})$  denotes the category with covers whose objects are finite-dimensional manifolds with surjective submersions as covers.

\begin{lemma}\label{lem:surj-sub}
The geometric functor $\yo \maps (\Mfd,\cS_{ss}) \to (\Sh(\Mfd),\cS_{\mix})$
reflects covers.
\end{lemma}
\begin{proof}
Suppose a morphism $\ph \maps Y \to X$ in $\Mfd$ is a mixed cover in $\Sh(\Mfd)$.
In particular, $\phi$ is an epimorphism.  Taking the pullback of $\phi$ along itself, we observe that the induced morphism of sheaves $Y \times_X Y \xto{\ph'} Y$ is an epimorphism. By definition of mixed cover, there exists an epimorphism of sheaves $f \maps F \to Y$ such that $\pr_X \maps Y \times_X F \to Y$ is a surjective submersion. Pasting vertical pullback squares gives a factorization of $\pr_X$ into two epimorphisms:
\[
Y \times_X F \xto{f'} Y \times_XY \xto{\ph'} Y.
\]
Therefore, we have exhibited a morphism $\ph'$ with representable codomain, and an epimorphism of sheaves $f'$ whose composition is a surjective submersion. In other words, $\ph'$ is a ``local stalkwise cover'', in the sense of \cite[Def.\ 6.4]{RZ}, 
that sits in a pullback square
\[
\begin{tikzdiag}{2}{2}
{
Y \times_XY\& Y \& \\
Y \& X  \& \\
};

\path[->,font=\scriptsize]
(m-1-1) edge node[auto,swap] {$\ph'$} (m-2-1)
(m-1-2) edge node[auto] {$\ph$} (m-2-2)
(m-1-1) edge node[auto] {$$} (m-1-2)
(m-2-1) edge node[auto] {$\ph$} (m-2-2)
;
\begin{scope}[shift=($(m-1-1)!.4!(m-2-2)$)]
\draw +(-0.25,0) -- +(0,0) -- + (0,0.25);
\end{scope}
\end{tikzdiag}
\]
whose horizontal arrows are epimorphisms (cf.\ \cite[Def.\ 6.5(2)]{RZ}). It then follows from \cite[Prop.\ 6.12]{RZ} that $\ph$ is a surjective submersion.  
\end{proof}

\begin{remark}\label{rmk:RZ}
Although the statement and proof of \cite[Prop.\ 6.12]{RZ} involve ``stalk-wise'' constructions and jointly conservative points for the category of sheaves $\Sh(\cat{Mfd})$, the essential ingredients needed for the proof of Lemma \ref{lem:surj-sub} can be all recast in terms of epimorphisms of sheaves. Arguments involving points or stalks are not required.  
\end{remark}

As in Example \ref{ex:CWC}, we have the pointed analogs: the category of pointed manifolds with covers $(\Mfd_\ast,{\cS_\ast}_{ss})$, and the descent category of pointed sheaves $(\Sh(\Mfd)_\ast,{\cS_\ast}_{\mix})$.
Let $\Gpd_\infty(\Sh(\Mfd))$ denote the category of $\infty$-groupoid objects in the descent category $(\Sh(\Mfd),\cS_{\mix})$ equipped with the canonical CFO structure. Let  $\Gpd_\infty(\Sh(\Mfd)_\ast)$ denote the analogous CFO of pointed $\infty$-groupoids.

We recover a strengthened version of \cite[Thm.\ 7.1]{RZ}. 

\begin{proposition}\label{prop:icfo-Mfd}
\mbox{}
\begin{enumerate}
\item The category of finite-dimensional Lie $\infty$-groupoids
$\LinfGpd:=\Gpd_\infty(\Mfd)$ and pointed Lie $\infty$-groupoids
\[
\LinfGpd_\ast:=\Gpd_\infty(\Mfd_\ast)
\]
 both admit the canonical iCFO structure. 

\item The forgetful functor 
\[
\LinfGpd_\ast \to \LinfGpd 
\]
and the Yoneda embedding 
\[
\LinfGpd \to \Gpd_\infty(\Sh(\Mfd))
\]
are exact functors.
\end{enumerate}
\end{proposition}
\begin{proof}
All assertions follow directly from Prop.\ \ref{prop:icfo-exact-geo}.
\end{proof}

\subsection{Homotopy theory of formal $\infty$-groupoids}
We combine the results of the previous sections to construct exact functors between iCFOs of formal $\infty$-groupoids, Lie $\infty$-groupoids, and related categories.

Recall from Sec.\ \ref{sec:geo-func-ex} that $\FDVect$ admits the structure of a descent category with surjections as covers. By Thm.\ \ref{thm:descent-CFO}, $\Gpd_\infty(\FDVect)$ admits the canonical CFO structure. Clearly, we have    
an isomorphism of categories of fibrant objects 
\[
\Gpd_\infty(\FDVect) = s\FDVect. 
\]
Indeed, the canonical CFO structure is identical to the one induced by the Dold-Kan model structure on simplicial vector spaces. In particular, the weak equivalences are the usual simplicial weak homotopy equivalences. 

\begin{remark}\label{rmk:formal-tan}
For later reference, we recall that a morphism $V_\bl \to W_\bl$ in $\sFDVect$ is a weak equivalence if and only if 
$N_\ast(V) \to N_\ast(W)$ is a quasi-isomorphism in $\Chain$. It is a fibration if and only if    
$N_k(V) \to N_k(W)$ is surjective in degrees $k \geq 1$ \cite[Lem.\ III.2.11]{GJ}.
\end{remark}

From Def.\ \ref{def:form-inf-gpd}, we recall that the category of formal manifolds admits the structure of a category with covers $(\conilsm, \cS^{\mathrm{mfd}}_{fs})$ with formal submersions as covers. 
\begin{theorem}\label{thm:formal-hmtpy}
\mbox{}
\begin{enumerate}

\item The category of formal $\infty$-groupoids $\FGpd := \Gpd_\infty(\conilsm)$ 
admits the canonical iCFO structure. 

\item A morphism $f \maps \cX_\bl \to \cY_\bl$ between formal $\infty$-groupoids is a weak equivalence (resp.\ fibration) if and only if $\prim_\bl(f)$ is a weak equivalence (resp.\ fibration) in $\sFDVect$. 

\item The primitives functor $\prim_\bl \maps \FGpd \to \sFDVect$ is an exact functor of iCFOs with respect to covering fibrations.

\item The point distributions functor \eqref{eq:dist} induces an exact functor 
with respect to covering fibrations
\begin{equation} \label{eq:Dist-Gpd}
\Dist_\bl \maps \LinfGpd_\ast \to \FGpd.
\end{equation}
Moreover, $\Dist_\bl$ preserves all fibrations and path objects.
\end{enumerate}
\end{theorem}
\begin{proof}
\mbox{}
\begin{enumerate}
\item By statement \eqref{it:prim} of Cor.\ \ref{cor:geo-func-ex}, 
the primitives functor $\prim \maps (\conilsm, \cS^{\mathrm{mfd}}) \to (\FDVect, \cS_{epi})$ is a geometric functor that reflects covers. Hence, the canonical iCFO structure exists by Prop.\ \ref{prop:icfo-exist}.

\item Follows from statement \eqref{it:exact-reflect} of Prop.\ \ref{prop:icfo-exact-geo}.

\item Follows from statement \eqref{it:exact-cover} of Prop.\ \ref{prop:icfo-exact-geo}.


\item Cor.\ \ref{cor:geo-func-ex} combined with Prop.\ \ref{prop:icfo-exact-geo} yields an exact functor with respect to covering fibrations $\Dist_\bl \maps \LinfGpd_\ast \to \Gpd_\infty(\conil)$ which preserves all fibrations and path objects. Since
$(\conilsm, \cS^{\mathrm{mfd}}_{fs}) \to (\conil,\cS_{fs})$ is a fully faithful geometric functor that reflects covers, we conclude by Prop.\ \ref{prop:icfo-exact-geo} that the functor $\Dist_\bl$ factors through the full subcategory $\FGpd$ with the desired properties.   
\end{enumerate}
\end{proof}

\subsection{Formal $\infty$-groups from Lie $\infty$-groups} \label{sec:LG-FG} 
Restricting the functor \eqref{eq:Dist-Gpd} to reduced objects, yields
\[
\Dist_\bl  \maps \LinfGrp \to \FG.
\]
which preserves weak equivalences and all fibrations. Since a fibration between reduced objects is a covering fibration, $\Dist_\bl$ preserves pullbacks along all fibrations in $\LinfGrp$. In this section, we characterize the image of $\Dist_\bl$ for an important class of examples: those Lie $\infty$-groups of the form $\ov{\cW}(G_\bl)$, where $G_\bl$ is a simplicial Lie group. 

We begin by expanding on Example \ref{ex:mfd-coalg}. Let $G \in \Grp(\Mfd_\ast)$ be a Lie group with Lie algebra $\g$. By Cor.\ \ref{cor:geo-func-ex},
$\Dist \maps (\Mfd_\ast,{\cS_\ast}_{ss}) \to (\conil,\cS_{fs})$ is a geometric functor, with image in the full subcategory $\conilsm$.  Hence, by statement \eqref{it:grp-obj} of Prop.\ \ref{prop:geo-func}, we have 
\[
\Dist(G)=\hom^\cont_\kk(\hA_G,\kk) \in \Grp(\conilsm),
\]
i.e.\ $\Dist(G)$ is cocommutative Hopf algebra. The action of $\g$ as derivations on $\O_{G,e}$ induces a linear map
\[
\g \to  \Dist(G) \quad \g \ni x \mapsto \Bigl(f \mapsto \frac{d}{dt}f(\exp(tx))\vert_{t=0} )\Bigr).
\] 
It follows from \cite[Thm.\ II.V.1]{Serre:Lie} that this assignment is a Lie algebra homomorphism, where the Lie bracket on $\Dist(G)$ is the usual commutator bracket. Clearly, it is natural in $G$. Let $\psi$ denote the unique lift of this homomorphism to the universal enveloping algebra of $\g$.
\begin{theorem}[Thm.\ II.V.2 \cite{Serre:Lie}]\label{thm:serre1}
The morphism $\psi \maps \cU(\g) \to \Dist(G)$ is an isomorphism of Hopf algebras.
\end{theorem}

As a corollary of this classical fact, we obtain our desired result.

\begin{corollary}\label{cor:WG-Ug}
Let $G_\bl$ be a simplicial Lie group with simplicial Lie algebra $\g_\bl$.  
The isomorphism of Thm.\ \ref{thm:serre1} induces a natural isomorphism of formal $\infty$-groups
\[
\psi_{G_\bl} \maps \Dist_\bl(\ov{\cW}G) \xto{\cong} \Wbar{\cU(\g)}.
\]
\end{corollary} 
\begin{proof}
Since $\Dist \maps (\Mfd_\ast,{\cS_\ast}_{ss}) \to (\conil,\cS_{fs})$ is a geometric functor, the assertion follows directly from statement \eqref{it:grp-obj} of Prop.\ \ref{prop:geo-func}, along with the naturality of the isomorphism $\psi$.
\end{proof}

\subsection{PBW bases for pointed formal $\infty$-groupoids}

\subsubsection{Almost simplicial objects} \label{sec:almost}
Denote by $\Del_+ \sse \Del$ the wide subcategory of ordinals whose morphisms fix $0$. Let $s^{+}\cC$ denote the category of \df{almost simplicial objects} of a category $\cC$ i.e.\  the category of functors $\Del_{+}^{\op} \to \cC$. 
From the definition, it follows that an almost simplicial object is equipped with all of the degeneracies and the ``positive'' face maps $d_i \neq d_0$ of a simplicial object, satisfying the usual identities. In particular, restriction induces a  functor 
\begin{equation} \label{eq:s+}
U^+_\bl \maps s\cC \to s^{+}\cC 
\end{equation}
that forgets the face map $d_0$.

\subsubsection{The classical and inhomogeneous $\Wbar$ functors} \label{sec:Wbar}
Let $\catGrp$ denote the category of simplicial group objects in a cartesian monoidal category $(\cC,\times,\one{\cC})$. 
We begin by recalling the definition of Moore's \cite{Moore1,Moore2} total space functor 
\[
\cW_\bl(-) \maps \catGrp \to s\cC, 
\]
for the Eilenberg-Mac Lane construction $\Wbar \maps \catGrp \to s\cC$. We follow the conventions from \cite[\Sec V.4]{GJ}. Given $G \in \catGrp$, define
\[
\cW_{-1}G:= \kk, \qquad \cW_nG:= G_n \times G_{n-1} \times \cdots \times G_0, \quad n \geq 0
\]
with face and degeneracy maps
\begin{equation} \label{eq:W}
\begin{split}
d^W_i(g_n,g_{n-1},\ldots,g_0)&=
\begin{cases}
(d^G_i g_n, d^G_{i-1} g_{n-1}, \ldots, d^G_0 g_{n-i} \ast g_{n-i-1}, g_{n-i-2},\ldots,g_0), & \text{if $i < n$},\\
(d^G_n g_n , d^G_{n-1} g_{n-1}, \ldots,d^G_1g_1\eps(g_0)), & \text{if $i=n$}.\\
\end{cases}\\
s^W_i (g_n , g_{n-1},\ldots, g_0) &= (s^G_i g^G_{n} , s^G_{i-1} g_{n-1} ... s^G_0 g_{n-i}, e_G , g_{n-i-1},\ldots, g_0)
\end{split}
\end{equation}
Above $\eps \maps G_0 \to \one{\cC}$ denotes the canonical morphism to the terminal object, while $d^{G}_1g_{1}\eps(g_0)$ is the image of the composition $G_0 \times G_{0} \to G_{0}\times \one{\cC} \cong G_{0}$. Observe that the face maps $d^{\cW}_{i \geq 1}$ satisfy the recursive identity
\begin{equation} \label{eq:W-recurse}
d^W_i(g_n,g_{n-1},\ldots,g_0)= \bigl(d^G_i g_n, d^W_{i-1}(g_n,g_{n-1},\ldots,g_0) \bigr) \quad \forall i \geq 1 
\end{equation} 

Furthermore, note that $\cW_\bl G$ is a left $G$-space with $g \cdot (g_n , g_{n-1},\ldots, g_0): = (g\ast g_n , g_{n-1},\ldots, g_0)$. The classical $\Wbar$ construction applied to $G$ is the quotient of $\cW_\bl G$ by this left action. We denote the corresponding quotient map by
\begin{equation} \label{eq:piW}
\pi^\cW \maps \cW_nG \to \Wb_{n}G, \quad \pi^{\cW}(g_n,g_{n-1},\ldots,g_0):=( \eps_{G_n}(g_n)g_{n-1},g_{n-2},\ldots,g_0).
\end{equation}

Recall that the simplicial object $\cW_\bl G$ admits an extra degeneracy \cite[\Sec III.5; Lem.\ 4.6]{GJ}. For $n \geq -1$, let $\si_n \maps \cW_{n}G \to \cW_{n+1}G$ denote the following morphisms in $\cC$:
\begin{equation} \label{eq:extra1}
\si_{-1} \maps \one{\cC} \xto{e_{G_0}} G_{0} \qquad \si_{n \geq 0}(g_{n},g_{n-1},\ldots,g_1,g_0):= (e_{G_{n+1}},g_{n},g_{n-1},\ldots,g_1,g_0).  
\end{equation}
It is straightforward to verify that following equalities are satisfied:
\begin{align} 
d^{\cW}_0 \si_n =& \id_{\cW_nG} & \quad \forall n \geq 0 \label{eq:extra2} \\
d^{\cW}_{i+1} \si_n  =& \si_{n-1}d^{\cW}_{i} &\quad \forall n \geq  1 ~ \forall i \geq 0 \label{eq:extra3}\\
s^{\cW}_{j+1} \si_n  = &\si_{n+1} s^{\cW}_{j} & \quad \forall n \geq  0 ~ \forall j \geq 0 \label{eq:extra4}\\
s^{\cW}_{0} \si_{n-1} = &\si_{n} \si_{n-1} &\quad \forall n \geq  0. \label{eq:extra5}
\end{align}
The extra degeneracy $\si_\bl$ for $\cW_\bl G$ will play an important role in Sec.\ \ref{sec:TheMap}.

A key ingredient for Prop.\ \ref{prop:Winhom}, the main result of this section,  is an ``inhomogeneous'' analog of the $\Wbar$ functor. We recall the construction from \cite[Def.\ 3.6]{Pridham:ProAHT}. 

\begin{notation}\label{note:red_diag}
Given an object $X \in \cC$ as above, let $\diag \maps X \to X \times X$ denote its diagonal. The $n$th iterated diagonal $\diag_{n} \maps X \to X^{\times(n)}$ is the morphism recursively defined as $\diag_{1}:=\id$,    
$\diag_{2}:=\diag$, and $\diag_{n} := (\diag \times \id) \cc \diag_{n-1}$. For a generalized element $x$ of $X$, we denote its image under $\diag_n$ as
\begin{equation} \label{eq:red_diag}
\diag_n(x) = \bigl(x^{[1]},x^{[2]},\ldots,x^{[n]}\bigr).
\end{equation}
\end{notation}
We define a functor 
\[
\Wbarin \maps \cat{Grp}(s\cC) \to s\cC
\]
by the following construction. Given $(G_\bl,\ast,e_\bl) \in \cat{Grp}(s\cC)$,
set $\Wbin_0(G):= \one{\cC}$, and for $n \geq 1$ define
objects
\begin{equation} \label{eq:Winh}
\Wbin_n(G):= G_{n-1} \times G_{n-2} \times \cdots \times G_1 \times G_0,
\end{equation} 
and morphisms
\begin{eqnarray} \label{eq:Winh-maps}
d_0(g_{n-1},\ldots,g_0)&:=& \Bigl((d^G_0g^{[1]}_{n-1})^{-1}\ast g_{n-2}, \ldots, \bigl((d^G_0 )^{n-1}g^{[n-1]}_{n-1}\bigr)^{-1}\ast g_0 \Bigr)\\
d_i(g_{n-1},\ldots,g_0)&:=&(d^G_{i-1}g_{n-1},\ldots, d^G_1g_{n-i+1}, g_{n-i-1},g_{n-i-2},\ldots, g_0), \quad 1 \leq i \leq n \notag\\
s_0(g_{n-1},\ldots,g_0)&:=& (e_{n},g_{n-1},\ldots,g_0) \notag\\
s_i(g_{n-1},\ldots,g_0)&:=&(s^G_{i-1}g_{n-1},\ldots, s^G_0g^{[1]}_{n-i}, g^{[2]}_{n-i},g_{n-i-1},\ldots, g_0) , \quad 1 \leq i \leq n. \notag
\end{eqnarray}
Above, in the formulae for $d_0$ and $s_0$, the elements $g^{[k]}_{n-1}$ and $g^{[k]}_{n-i}$ denote the $k$th components of $\diag_{n-1}(g_{n-1})$ and $\diag(g_{n-i})$, respectively.  

\begin{remark}\label{rmk:Wbarin}
We emphasize that the zeroth face map of $\Wbarin(G)$ is the only morphism in \eqref{eq:Winh-maps} that involves the group operation of $G_\bl$ or the zeroth face map $d^G_0$. The rest depend only on the almost simplicial structure of $U^+_\bl(G) \in s^+\cC$, the identity operations, or the diagonal map $\diag \maps G_\bl \to G_\bl \times G_\bl$. 
\end{remark}

\begin{lemma}[Lemma 6.22 \cite{JP:Dmod}]\label{lem:Wbarinh}
The morphisms \eqref{eq:Winh-maps} give $\Wbarin(G)$ the structure of a simplicial object in $\cC$, and the simplicial morphism
\[
\phi_\Wb \maps \Wbarin(G) \xto{\iso} \Wbar(G) 
\]
\[
\begin{split}
\phi_\Wb\bigl(g_{n-1},&\ldots,g_0 \bigr):= \\
&\Bigl((g^{[1]}_{n-1})^{-1}, d^G_{0}(g^{[2]}_{n-1}) \ast (g^{[1]}_{n-2})^{-1}, d^G_{0}(g^{[2]}_{n-2}) \ast (g^{[1]}_{n-3})^{-1}, \ldots,  
d^G_0(g^{[2]}_{1})\ast g^{-1}_0 \Bigr)
\end{split}
\]
induces a natural isomorphism $\Wbarin(-) \xto{\cong} \Wbar(-)$.
\end{lemma}

\subsubsection{The classical PBW theorem} \label{sec:PBW-classic}
Given an almost simplicial vector space $V_\bl \in \spFDVect$, recall $\cospSym(V) \in \spconil$ denotes the cofree (conilpotent cocommutative) almost simplicial coalgebra cogenerated by $V_\bl$.  

\begin{definition} \label{def:PBW} 
A \df{PBW basis} for a pointed formal $\infty$-groupoid $\cX_\bl \in \FGpd$ is an isomorphism of almost simplicial coalgebras $U^+_{\bl}(\cX) \xto{\iso} \cospSym(V)$, where $V_\bl=U^{+}_{\bl}\prim(\cX)$.
\end{definition}

The rationale behind this definition is the classical PBW theorem for simplicial Lie algebras. More precisely:
\begin{proposition} \label{prop:Winhom} 
Let $\g_\bl \in \sLie$. The natural coalgebra isomorphisms
\begin{equation} \label{eq:pbw} 
\pbw_n \maps \coSym(\g_n) \xto{\iso} \cU(\g_n), \quad \pbw_n(x_1x_2 \cdots x_k):=
\frac{1}{k!} \sum_{\si \in S_k} x_{\si(1)} \ast x_{\si(2)} \ast \, \cdots \, \ast x_{\si(k)}
\end{equation}
yield a canonical PBW basis for the formal $\infty$-group $\Wbar \cU(\g)$.    
\end{proposition}

\begin{proof}
Let $\und{\g}_\bl$ denote the underlying simplicial vector space of $\g_\bl$, considered as an abelian group object in $\sVect$. Then $\cosSym(\und{\g})$ is an abelian group object in $\sconil$. Indeed, the functor 
\[
\cosSym(-) \maps \sVect \to \scoccom
\]
 is right adjoint and hence preserves group objects. This gives an identification 
\begin{equation} \label{eq:Winhom0}
\cosSym(\Wbin (\und{\g})) \cong \Wbarin \coSym(\und{\g}).
\end{equation}
Since $U^+_\bl\cc \cosSym(-) = \cospSym \cc U^+_\bl(-)$, we obtain a natural isomorphism of almost simplicial coalgebras
\begin{equation} \label{eq:Winhom1}
\cospSym(U^+ \Wbin(\und{\g})) \cong U^+_\bl\Wbin \coSym(\und{\g}).
\end{equation}
The coalgebra morphisms $\pbw_n$ from \eqref{eq:pbw} are natural isomorphisms between functors from Lie algebras to coalgebras. By Rmk.\ \ref{rmk:Wbarin}, we see that all degeneracies and positive face maps for $\Wbarin \coSym(\und{\g})$ and   
$\Wbarin \cU(\g)$ are constructed from morphisms which lie in the image of the functors $\coSym(-)$ and $\cU(-)$, respectively. From this, we conclude that the morphisms
\[
\begin{split}
\pbw_{n-1} \tensor \pbw_{n-2}\tensor  \cdots \tensor \pbw_0  \maps& \\
\Sym^{\coalg}(\g_{n-1}) \tensor \Sym^{\coalg}(\g_{n-2}) &\tensor \cdots \tensor  
\Sym^{\coalg}(\g_{0})  \longrightarrow 
\cU(\g_{n-1}) \tensor \cU(\g_{n-2}) \tensor \cdots \tensor  
\cU(\g_{0}) 
\end{split}
\]
assemble into an  isomorphism of almost simplicial coalgebras
\[
{\pbw^\tensor} \maps U^+_\bl\Wbin \coSym(\und{\g})  \xto{\cong} U^+_\bl\Wbin \cU(\g).
\]
Combining this with \eqref{eq:Winhom1} and Lemma \ref{lem:Wbarinh} gives the desired canonical PBW basis:
\[
\cospSym(V) \xto{\cong} U^+_\bl\Wb\, \cU(\g)
\]
where $V_\bl:=U^+_\bl \Wbin(\und{\g})$.
\end{proof}a

\begin{remark}\label{rmk:dsum-tensor}
The identification \eqref{eq:Winhom0}
in the above proof involves the usual isomorphism 
\[
\Sym^{\coalg}(V \dsum V') \xto{\cong} \Sym^{\coalg}(V) \tensor \Sym^{\coalg}(V'), \quad (v,v') \mapsto v \tensor 1 + 1 \tensor v'.
\]
which is natural with respect to the vector spaces $V$ and $V'$.
\end{remark}

Proposition \ref{prop:Winhom} implies that there is a natural isomorphism
of almost simplicial vector spaces $U^+_\bl \Wbin(\und{\g}) \cong 
U^+_\bl\prim (\Wb\, \cU(\g))$. In the following corollary, we show that this isomorphism lifts to simplicial vector spaces. 
 
\begin{corollary}\label{cor:prim-iso}
Let $\gb \in \sLie$ with underlying simplicial vector space $\und{\g}_\bl$. Then there is a natural isomorphism of simplicial vector spaces
\[
\xi_{\g} \maps \Wbar{(\und{\g})} \xto{\cong} \prim_\bl \Wb{\cU(\g)} \quad \text{in $\sVect$}.
\]
\end{corollary}
\begin{proof}
We will exhibit an isomorphism $\Wbarin{(\und{\g})} \xto{\cong} \prim_\bl\Wbarin{\cU(\g)}$. As in Rmk.\ \ref{rmk:dsum-tensor}, the linear maps
\[
\begin{split}
\ga_n \maps \g_{n-1} \dsum \g_{n-2} \dsum \cdots \dsum \g_0 &\to 
\coSym(\g_{n-1}) \tensor \coSym(\g_{n-2})\tensor \cdots \tensor \coSym(\g_0).\\
\ga_n(x_{n-1},\ldots,x_{0}) &= \sum_{k=1}^n 1^{\tensor k-1}\tensor x_{n-k} \tensor 1^{\tensor n-k} \\
\end{split}
\] 
assemble into isomorphism of simplicial vector spaces 
\[
\ga \maps \Wbarin{(\und{\g})} \to \prim_\bl(\Wbin\coSym(\g)).
\]
The isomorphism of almost simplicial coalgebras
\[
{\pbw^\tensor} \maps U^+_{\bl}\Wbin \bigl(\coSym(\und{\g}) \bigr) \xto{\cong} U^+_\bl\Wbin \bigl(\cU({\g}) \bigr)
\]  
exhibited in the proof of Prop.\ \ref{prop:Winhom} induces the identity on primitives, and hence an equality of almost simplicial vector spaces  
\[
U^{+}_{\bl}\prim \Wbin(\coSym(\g)) = U^{+}_{\bl} \prim \Wbin(\cU(\g)).
\]
Therefore, it remains to verify that the equality 
\begin{equation} \label{eq:cor-prim-pf1}
d_{0}^{\Wbin\cU} \cc \ga_{n} = \ga_{n-1} \cc d_{0}^{\Wbin \und{\g}}   
\end{equation}
holds for all $n \geq 2$. Following the definition \eqref{eq:Winh-maps} for $d_0^{\Wbin}$, we first compute the diagonals. If $x \in \g_{n-1} \sse \cU(\g_{n-1})$, then
\[
\begin{split}
\diag^{\cU(\g)}_{n-1}(x) &= \sum_{\el=1}^{n-1} 1^{\tensor \el -1} \tensor x \tensor 1^{n-1-\el} \in \cU(\g_{n-1})^{\tensor n-1}\\
\diag^{\und{\g}}_{n-1}(x) &= (x,x,\ldots,x) \in {\g}_{n-1}^{\times n-1}.
\end{split}
\]   
Therefore, we obtain for the left-hand side of \eqref{eq:cor-prim-pf1}:
\begin{equation} \label{eq:cor-prim-pf2}
\begin{split}
d_{0}^{\Wbin\cU} \bigl(\ga_{n}&(x_{n-1},x_{n-2},\ldots,x_0) \bigr) \\ 
&=d_{0}^{\Wbin\cU}(x_{n-1}\tensor 1^{\tensor n-1}) 
+ \sum_{\el=2}^n 1^{\tensor \el-2} \tensor x_{n-\el} \tensor 1^{\tensor n-\el}\\
&= - \sum_{\el=1}^{n-1} 1^{\tensor \el -1} \tensor (d^{\g}_0)^{\el} x_{n-1} \tensor 1^{\tensor n-1 -\el}  + \sum_{\el=2}^n 1^{\tensor \el-2} \tensor x_{n-\el} \tensor 1^{\tensor n-\el}.
\end{split}
\end{equation}
The right-hand side is
\[
\begin{split}
\ga_{n-1} \bigl(d_{0}^{\Wbin \und{\g}}&(x_{n-1},x_{n-2},\ldots,x_0) \bigr) \\  
&= - \sum_{\el=1}^{n-1}\ga_{n-1}\bigl(\underbrace{0,\ldots,0}_{\el-1}, (d^{\g}_0)^{\el}x_{n-1},\underbrace{0,\ldots,0}_{n-\el-1} \bigr) + \ga_{n-1}(x_{n-2},\ldots,x_0)\\
&= - \sum_{\el=1}^{n-1}1^{\tensor \el-1}\tensor (d^{\g}_0)^{\el}x_{n-1}\tensor 1^{n-\el-1} + \sum_{k=1}^{n-1} 1^{\tensor k-1}\tensor x_{n-k-1} \tensor 1^{\tensor n-k-1}.
\end{split}
\]
After reindexing the second summation by setting $k=\el -1$, we obtain
\eqref{eq:cor-prim-pf2}. Hence, the equality \eqref{eq:cor-prim-pf1} holds.
Composing with the isomorphism of Lem.\ \ref{lem:Wbarinh} gives the desired isomorphism: 
\begin{equation} \label{eq:cor-prim-map}
\xi_\g:=\phi_{\Wb \cU(\g)} \cc \ga \cc \phi^{-1}_{\Wb \und{\g}} \maps
\Wbar{(\und{\g})} \xto{\cong} \prim_\bl \Wb{\cU(\g)}.
\end{equation}
\end{proof}

\subsubsection{Existence of PBW bases} \label{sec:formal-pbw}
In this section, we prove 
\begin{theorem}\label{thm:pbw}
Every pointed formal $\infty$-groupoid $\cX_\bl$ admits a PBW basis.
\end{theorem} 

We first setup the needed framework. It will be convenient to extend the equivalence of Prop.\ \ref{prop:dual} to simplicial objects
\begin{equation} \label{eq:simp-dual}
\kk[-]^{\bl} \maps \sconil \overset{\simeq}{\longleftrightarrow} \bigl(\cclnAlg \bigr)^\op \maps \spf_{\bl}(-),
\end{equation}
and work in dual category of cosimplicial algebras $\cclnAlg$. Also, in analogy with Sec.\ \ref{sec:almost}, let $\cpclnAlg$ denote the category of almost cosimplicial algebras with forgetful functor
\[
U^\bl_{+} \maps \cclnAlg \to \cpclnAlg.
\]
We dualize the matching objects introduced in Sec.\ \ref{sec:simpdesc}. Let $S$  be a finite set, and $A \in \clnAlg$. We denote the tensor of $S$ with $A$ by
\[
S \cast A := \bightensor_S A.  
\]
In particular, given a set of generators $\{x_1,\ldots,x_k\}$ for $A$, and a finite indexing set $\cI$ we denote by  
\begin{equation} \label{eq:gen-tensor}
\{x^\al_{i}\} \sse \bightensor_{\al \in \cI} A
\end{equation}
the corresponding set of generators for $\bightensor_{\al \in \cI} A$.

Extending this to cosimplicial objects, let $S \in \sSet$ be finitely-generated, $A^\bl \in c\clnAlg$, and define $S \cast A^\bl \in \clnAlg$ as the coequalizer of the diagram
\[
\bightensor_{\substack{[m] \to [n] \\ m,n \leq d} } S_n \cast A^m \Eq  \bightensor_{m \leq d} S_m \cast A^m
\]
where $d$ is the dimension of $S$. In particular, for $C \in \sconil$ and $A^\bl_C:= \kk[C]^{\bl}$, we have the natural isomorphism
\[
S \cast A_C^\bl \cong \hom_{\kk} \bigl( \hom(S,C_\bl),\kk). 
\]    
When $S=\Lam^n_k$ is a horn, we will use the notation
\[
A_{C_{n,k}}:= \Lam^{n}_k \cast A_C^\bl \cong \hom_\kk(C_{n,k},\kk).
\]
\begin{construction}[Presentation for $A_{\cX_{n,0}}$] \label{con:horn}
For $n \geq 1$, the zeroth horn of $\Del^n$ can be written as a coequalizer
\[
\coprod_{1 \leq i < j \leq  n} \Del^{n-2} \rightrightarrows  \coprod_{1 \leq r \leq n} \Del^{n-1} \to \Lam^n_0 
\] 
Recall, that if $\Del^{n-2}_{(i,j)}$ and $\Del^{n-1}_{(k)}$ denote, respectively, the $(i,j)$th and $k$th summands of the above coproducts, then the two parallel arrows are induced by the coface maps
\begin{equation} \label{eq:para-arrow}
\Del^{n-2}_{(i,j)} \xto{d^{j-1}} \Del^{n-1}_{(i)}, \qquad \Del^{n-2}_{(i,j)} \xto{d^{i}} \Del^{n-1}_{(j)}.
\end{equation}
Let $\cX_\bl \in \FGpd$. 
Using Notation \ref{note:hom}, we have $\cX_{n,0} \in \conilsm$. Hence, $A_{\cX_{n,0}}$ is smooth and is the coequalizer of the diagram
\[
\bightensor_{1 \leq i < j \leq  n} A^{n-2}_\cX \rightrightarrows  \bightensor_{1 \leq r \leq n} A^{n-1}_\cX
\] 
in which the parallel arrows are induced by the coface maps of $A^{\bl}_{\cX}$, as in \eqref{eq:para-arrow}. Let $\{x_1, \ldots, x_t\} \sse A^{n-1}_\cX$ and $\{y_1,\ldots,y_m\} \sse A^{n-2}_\cX$ denote generators such that $A^{n-1}_\cX = \kk[[\{x_k\}]]$ and $A^{n-2}_\cX = \kk[[\{y_\el\}]]$. Then, using the notation \eqref{eq:gen-tensor}, we have
\[
A_{\cX_{n,0}} = \kk[[\{x^r_{k}\}]]  \bigm / I
\]
where $I = \gen{ \{g_{\el,(i,j)}\} }$ is the ideal generated by the elements
\[
g_{\el,(i,j)} := d^i {y}_{\el}^{(i,j)} - d^{j-1}{y}_{\el}^{(i,j)}, \quad \el=1,\ldots, m, \quad 1 \leq i < j \leq n.
\] 
\end{construction}

\myspace

\begin{proof}[{\bf Proof of Theorem \ref{thm:pbw}}]
Let $d^i$ and $s^j$ denote the coface/codegeneracy maps of $A^\bl_\cX$. Throughout, $\mm_n \ideal A^n_\cX$ denotes the unique maximal ideal in dimension $n$. 
We will construct, inductively, vector spaces $W^{n}$ spanned by
finite subsets $S^{m} \sse \mm_m$ for all $m \geq 0$ such that
\begin{equation}\label{eq:induct}
\begin{split}
d^i(S^{m-1}) \sse W^m, & \quad i=1, \ldots, m \\
s^{j}(S^m) \sse W^{m-1} & \quad j = 0, \ldots, m-1, \\
 A^m_{\cX} &= \wh{\Sym}(W^m)
\end{split}
\end{equation}
for each $m \geq 1$. From this we will obtain an isomorphism of almost cosimplicial algebras
\[
\wh{\cpSym}(W) \cong U^\bl_{+} (A_{\cX}). 
\]
with $W^{\bl} =U^\bl_{+}(\mm/\mm^2)$. Applying $\spf_\bl(-)$ will then yield the desired result 
\[
\cospSym(V) \cong U^{+}_{\bl}(\cX),
\]
with $V_\bl := (W^\bl)^{\vee}=U^{+}_{\bl}\prim(\cX)$.  

For the base case, $n=1$, observe that we have retraction in $\clnAlg$
\[
A^0_{\cX} \xto{d^1} A^1_{\cX} \xto{s^0} A^0_{\cX}.
\]
We consider the induced idempotent endomorphism on $\mm_1$ as in \cite[Thm.\ 6.7]{Chak}. From this we conclude, that there exist sets of generators 
$S^0 \sse \mm_0$ and $S^1 \sse \mm_1$ such that: 
\[
A^0_{\cX}= \kk[[S^0]], \quad  A^1_{\cX}= \kk[[S^1]], \quad d^1(S^0) \sse S^1
\]
and 
\[
s^0\bigl(S^1 \setminus  d^1(S^0) \bigr) =0. 
\]
Therefore, we set $W^0:= \spann_\kk S^0$, and $W^1:= \spann_\kk S^1$.

Let $n > 1$ and suppose for all $k < n$ we have $S^{k} \sse  A^k_{\cX} $
satisfying \eqref{eq:induct}. 
We proceed in three steps.    
\begin{enumerate}
\item Since $\cX_\bl \in \FGpd$, the horn filling map $\io \maps A_{\cX_{n,0}} \to A^n_{\cX}$ is formally smooth, and so it admits a left inverse as in the verification of axiom \ref{ax5}  in the proof of Prop.\ \ref{prop:descent}. Therefore, we have a retraction in $\clnAlg$ 
\begin{equation} \label{eq:retract}
 A_{\cX_{n,0}} \xto{\io} A^n_{\cX} \xto{\si}  A_{\cX_{n,0}}, \qquad \si \cc \io = \id.  
\end{equation}
Combining this with Construction \ref{con:horn} yields the commutative diagram
\begin{equation} \label{eq:diag}
\begin{tikzdiag}{2}{3}
{
A^{n-1}_\cX \wh{\tensor} A^{n-1}_\cX \wh{\tensor} \cdots \wh{\tensor} A^{n-1}_\cX \&  \& \\
A_{\cX_{n,0}} \&  \& A^n_\cX \\
};
\path[->,font=\scriptsize]
(m-1-1) edge node[auto] {$d^1 \vee d^2 \vee \cdots \vee d^n$} (m-2-3)
(m-2-1) edge node[auto] {$\io$} (m-2-3);
;
\path[->,font=\scriptsize]
(m-1-1) edge node[auto,swap] {$\ro$} (m-2-1);
\end{tikzdiag}
\end{equation}
We emphasize that the vertical map $\ro$ is surjective. Let $S^{n-1} = \{ e_1, \ldots, e_m\}$. By the induction hypothesis $A^{n-1}_\cX \wh{\tensor}  \cdots \wh{\tensor} A^{n-1}_\cX$ is isomorphic to the algebra of formal power series in variables
\[
\{ y^i_j \st i=1,\ldots,n, \, j=1,\ldots m\},
\]
such that $\bigl(d^1 \vee d^2 \vee \cdots \vee d^n\bigr)(y^i_j) = d^i e_{j}$. Define 
\[
z^i_j: = \ro(y^i_j).
\]
Then the $z^i_j$ generate the maximal ideal of $A_{\cX_{n,0}}$. Let $z^{i_1}_{j_1}, \ldots,z^{i_r}_{j_r}$ denote a minimal generating subset. Combining this with the retraction \eqref{eq:retract} gives a splitting of the cotangent space of $A^n_\cX$:
\[
\mm_{n}/\mm^2_n = \ker \ba{\si} \dsum \im \ba{\io} = \ker \ba{\si} \dsum \spann_\kk \{ d^{i_1}e_{j_1} + \mm^2_n, \ldots,  d^{i_r}e_{j_r} + \mm^2_n\},  
\]
with $\{ d^{i_1}e_{j_1} + \mm^2_n, \ldots,  d^{i_r}e_{j_r} + \mm^2_n\}$ linearly independent.
\myspace

\begin{claim} \label{claim:linear}
For each generator $e_k \in S^{n-1}$, its image $d^\ell e_k \in \mm_n$ is a $\kk$-linear combination of $d^{i_1}e_{j_1}, \ldots,  d^{i_r}e_{j_r}$ for all $\ell =1,\ldots, n$.  
\end{claim}

\paragraph{{\it Proof of Claim \ref{claim:linear}}}
It suffices to show that  $z^\ell_k \in \spann_\kk\{z^{i_1}_{j_1}, \ldots,  z^{i_r}_{j_r}\} \sse A_{\cX_{n,0}}$. The surjectivity of $\ro$ in \eqref{eq:diag}, along with the induction hypothesis, implies that
$z^\ell_k = \sum^N_{t=0}f_t z^{i_t}_{j_t}$ where each $f_t$ is the image of a formal power series in the variables $y^i_j$. Therefore, it suffices to verify that if $h \in A^{n-1}_\cX \wh{\tensor} \cdots \wh{\tensor} A^{n-1}_\cX$ such that 
\begin{equation} \label{eq:claimlin1}
y^\ell_k - h \in \ker \ro,
\end{equation}
then there exists a linear polynomial $q$ such that 
\begin{equation} \label{eq:claimlin1-1}
y^\ell_k - q \in \ker \ro.
\end{equation}
Let $S^{n-2} = \{w_1, \ldots, w_{m'}\}$.  
Then $A^{n-2}_\cX = \kk[[S^{n-2}]]$. Construction \ref{con:horn}
implies that $\ker \ro$ is generated by elements
\[
g_{k,(i,j)} = d^i {w}^{(i,j)}_k - d^{j-1}{w}^{(i,j)}_k, \quad k=1,\ldots, m', \quad 1 \leq i < j \leq n,
\] 
where $\{{w}^{(i,j)}_k \}_{k=1}^{m'} \cong S^{n-2}$ is the set of generators of the $(i,j)$ factor of  $\bigotimes_{i < j} A^{n-2}_\cX$. The induction hypothesis implies that $d^i(S^{n-2}) \sse W^{n-1}=\spann_\kk S^{n-1}$
for  $i \geq 1$. Hence, each $g_{k,(i,j)}$ is a homogeneous linear polynomial in the variables $y^{i'}_{j'}$. Hence, given $y^\el_k - h$ as in \eqref{eq:claimlin1}, by degree reasons, there must exist a linear polynomial $q$ as in \eqref{eq:claimlin1-1}. This completes the proof of the claim.

\item Next, in analogy with the argument used for the base case,  we consider the retracts
\[
A^{n-1}_\cX \xto{ d^1} A^{n}_\cX \xto{s^0} A^{n-1}_\cX,  \quad   
A^{n-1}_\cX \xto{ d^2} A^{n}_\cX \xto{s^1} A^{n-1}_\cX,  \quad   
\cdots \quad
A^{n-1}_\cX \xto{ d^n} A^{n}_\cX \xto{s^{n-1}} A^{n-1}_\cX.
\]
and split $\mm_n$ as a $\kk$-linear space
\begin{equation} \label{eq:split3}
\mm_n = \bigcap_{j=0}^{n-1} \ker s^j \oplus  \spann_\kk \Bigl(\bigcup_{i=1}^n \ov{\im d^i} \Bigr),
\end{equation}
as is done in constructing the normalized cochain complex. Above, $\ov{\im d^i}$ is the augmentation ideal of the subalgebra $\im d^i \sse A^n_\cX$.  

\myspace

\begin{claim} \label{claim:extend}
There exist
$f_1,\ldots, f_t \in \bigcap_{j=0}^{n-1} \ker s^j + \spann_\kk \{ d^{i_1}e_{j_1}, \ldots,  d^{i_r}e_{j_r}\}$ 
such that 
\[
\{f_1 + \mm^2_n,  \ldots,
f_t + \mm^2_n, d^{i_1}e_{j_1} + \mm^2_n, \ldots,  d^{i_r}e_{j_r} + \mm^2_n\}
\]  
is a basis for $\mm_n/\mm_n^2$. 
\end{claim}
\paragraph{{\it Proof of Claim \ref{claim:extend}}}
Let $g_1,\ldots, g_t \in \mm_n$ be {any}    
elements such that the set $\{g_i + \mm^2_n\}_{i \geq 1}$ extends $\{d^{i_\el}e_{j_\el} + \mm^2_n\}_{\el \geq 1}$ to a basis. Using the splitting \eqref{eq:split3} write $g_i = q_i + h_i$ with $q_i \in \bigcap \ker s^j$. By the induction hypothesis, each $h_i$ can be written as a linear combination of formal power series in the variables $d^\el e_k$ with trivial constant term. Let $\ti{h}_i \in \spann_\kk\{ d^\el e_k \st  1 \leq \el \leq n, \, 1 \leq k \leq m\}$  denote the linear term of $h_i$. Then by Claim \ref{claim:linear}, $\ti{h}_i \in \spann_\kk \{ d^{i_1}e_{j_1}, \ldots,  d^{i_r}e_{j_r}\}$.
Define $f_i:= q_i + \ti{h}_i$. Then, by construction, $g_i - f_i \in \mm^2_n$.
This completes the proof of the claim.

\item Finally, define
\[
S^n:=\{f_1,\ldots,f_t, d^{i_1}e_{j_1}, \ldots,  d^{i_r}e_{j_r} \}, \qquad W^{n}:=\spann_{\kk} S^n,
\]
where the $f_i$ are as in Claim \ref{claim:extend}. By construction, the set $S^n$ satisfies the three criteria listed in \eqref{eq:induct}. Indeed, Claim \ref{claim:linear} implies that $d^i(S^{n-1}) \sse W^n$ for $i=1,\ldots, n$. To verify $s^j(S^n) \sse W^{n-1}$ for each $j=0,\ldots, n-1$, it follows from Claim \ref{claim:extend} that it is sufficient to verify that $s^j(d^\el e_k) \in W^{n-1}$ for any $\el =1,\ldots,n$ and $k=1,\ldots,m$. For the cases in which $s^j d^\el \neq \id$, the cosimplicial identities imply that $s^j d^\el = d^{\el'}s^{j'}$ with $\el' \geq 1$. Hence, in these cases $s^j(d^\el e_k) \in W^{n-1}$ by the induction hypothesis. Finally, since the image of $S^n$ in $\mm_n/\mm^2_n$ is a basis, Lemma \ref{lem:gen} implies that $A^n_{X} = \kk[[S^n]]$.
\end{enumerate} 
This concludes the proof of the proposition. 
\end{proof}

\section{Differentiation of formal $\infty$-groups} \label{sec:FD}

\subsection{Simplicial and almost simplicial Dold-Kan correspondence}\label{sec:simp-DK}
Let $N_\ast \maps \sVect \adjunct \Chain \maps K_\bl$ denote the classical Dold-Kan adjoint equivalence\footnote{The functor $K_\bl$ is denoted by $\Gamma_\bl$ in \cite[\Sec 5.0]{F2}.}. We denote the counit and unit of the adjunction by
\[
\eps_{\DK} \maps N_\ast K_\bl \xto{\cong} \id_{\Ch_\ast}, \quad  \eta_{\DK} \maps \id_{\sVect} \xto{\cong}  K_\bl N_\ast
\]

As in \cite[\Sec 5.0.5]{F2}, we define $N_0(V_\bl):=V_0$, and for $n\geq 1$
\begin{equation} \label{eq:simp-norm}
N_n(V_\bl):= \bigcap_{i=1}^n\ker d_i, \qquad d_{N_\ast(V_\bl)}:=d_0 \maps N_{n}(V_\bl) \to N_{n-1}(V_\bl). 
\end{equation} 
Given $(C,d_C) \in \Ch_\ast$, we recall the construction of the simplicial vector space $K_\bl(C)$. Define $K_0(C):=C_0$. Let $n \geq 1$. For each $\el$ such that $1 \leq \el \leq n$, and each strictly increasing sequence $0 \leq j_1 < \cdots < j_\el \leq n -1$ of length $\el$, let $s_{j_\el} s_{j_{\el -1}} \cdots s_{j_1}C_{n-\el}$ denote an isomorphic copy\footnote{Here  $s_{j_\el} s_{j_{\el -1}} \cdots s_{j_1}$ is a formal symbol.} of the vector space $C_{n-\el}$. Then define, as a vector space, 
\begin{equation} \label{eq:simp-K}
K(C)_n:= C_n \oplus \bigoplus^n_{\el =1} \Bigl (~ \bigoplus_{0 \leq j_1 < \cdots < j_\el \leq n -1} s_{j_\el} s_{j_{\el -1}} \cdots s_{j_1}C_{n-\el}    \Bigr).
\end{equation}
The face maps $d_i \maps K(C)_n \to K(C)_{n-1}$ are defined on the summand $C_n$ by
\[
d_0 \vert_{C_n}:=d_C, \quad d_{i \geq 1} \vert_{C_n}:=0.
\]  
The values of the face maps on the remaining summands, as well as the  
degeneracy operators $s_i \maps K_{n-1}(C) \to K_n(C)$  are defined in the obvious way so that the simplicial identities are satisfied. We note that this presentation of $K_\bl(C)$ exhibits the natural splitting of $K_n(C)$ into non-degenerate and degenerate simplices.

Recall that the Eilenberg-MacLane morphism, or ``shuffle map'' \cite{EM}
\begin{equation} \label{eq:EM}
\begin{split}
\EM_{V,W} &\maps N_\ast(V_\bl) \tensor N_\ast(W_\bl)  \xto{\simeq} N_\ast(V_\bl \tensor W_\bl)\\
(\EM_{V,W})_{n}:= \sum_{p+q =n} \EM_{p,q} \, , & \qquad \EM_{p,q} \maps N_p(V_\bl) \tensor N_q(W_\bl) \to N_{p+q}(V_\bl \tensor W_\bl),  
\end{split}
\end{equation}  
is a symmetric monoidal natural homotopy equivalence that exhibits $N_\ast$ as a lax symmetric monoidal functor. Here, given $v\tensor w \in N_p(V) \tensor N_q(W)$, we have
\[
\EM_{p,q}(v \tensor w):= \sum_{I,J}\sgn(I,J) s_{j_q} s_{j_{q-1}} \cdots s_{j_1} (v) \tensor s_{i_p} s_{i_{p-1}} \cdots s_{i_1} (w),
\]
where the sum is over all partitions $(I,J)$ of the set $\{0,\ldots,p+q -1\}$, 
with $I=\{i_1 < \ldots < i_p\}$ and $J=\{j_1 < \ldots < j_q\}$, and $\sgn(I,J)$ denotes the sign of the corresponding $(p,q)$ shuffle. We refer to the proof of \cite[Thm.\ 5.2.3]{F2} for further details.

\subsubsection{Almost simplicial Dold-Kan correspondence}\label{sec:al-simp-DK}
Let $\cC$ be an abelian category. As in \eqref{eq:s+}, we have the forgetful functor $U^{+}_\bl \maps s\cC \to s^+\cC$. Analogously, denote by 
\[
U^+_\ast \maps \Chain(\cC) \to \grC
\] 
the forgetful functor from complexes to graded objects in $\cC$ that drops the differential.
As emphasized in \cite[Rmk.\ 5.11]{Pridham:DDArtin}, there is an equivalence of categories $N^+_\ast \maps s^+\cC \to \grC$ which is given by exactly the same formula as \eqref{eq:simp-norm}.  Its inverse functor is $K^+_\bl(V_\ast):=U^+_\bl K_\bl(V_\ast)$. Here we identify $\grC$ with chain complexes with trivial differential. The unit $\eta^+_{\DK}$ and counit $\eps^+_{\DK}$  are defined analogously. In particular, we have the following compatibilities
\[
\begin{split}
N^+_\ast U^+_\bl  = U^+_\ast N_\ast, &\quad K^+_\bl U^+_\ast  = U^+_\bl K_\bl \\
U^+_\bl \eta_{\DK} = \eta^+_{\DK} U^+_\bl, &\quad U^+_\ast \eps_{\DK} = \eps^+_{\DK} U^+_\ast
\end{split}
\]
 
\subsubsection{Comparing $K_\bl(-)$ with $\Wbar{(-)}$} \label{sec:K-vs-Wb}
Let us record here the following basic fact, e.g.\ \cite[\Sec V.1]{GJ}, for constructing morphisms between simplicial objects in a bicomplete category $\cC$.
We will use it frequently throughout the remainder of the paper.
\begin{lemma}\label{lem:sk}
Let $X_\bl \in s\cC$ and $f \maps \sk_n X_\bl \to Y_\bl$ a morphism.
Suppose $N \in \cC$ such that $(\sk_n X)_{n+1} \coprod N = X_{n+1}$. Then extensions of $f$ to a morphism $g \maps \sk_{n+1} X_\bl \to Y_\bl$ are in one to one correspondence with morphisms $\ti{f} \maps N \to Y_{n+1}$ in $\cC$ such that 
\begin{equation} \label{eq:sk}
d^Y_{i} \circ \ti{f} = f \circ d^{X}_{i} \vert_{N}
\end{equation}
for all $0 \leq i \leq n$.  
\end{lemma}

Given a non-negatively graded vector space $V \in \gVect$, let $\bs V$ denote the graded vector space $(\bs V)_i:=V[1]_{i}=V_{i-1}$. For a complex $(V,d)$, our convention\footnote{This aligns with the standard conventions for suspension found in the literature for $L_\infty$-algebras.} for the induced differential on $\bs V$ is 
\[
d_{\bs V}(\bs v):= \bs dv.
\] 
A simplicial vector space $V_\bl$ is, in particular, a simplicial abelian group. The reduced simplicial vector space $\Wbar{(V_\bl)}$ can be characterized in terms of the suspension of the normalized complex of $V_\bl$.

\begin{proposition} \label{prop:K-Wbar-vect}
Let $V_\bl \in \sVect$. There is an isomorphism of simplicial vector spaces
\[
\ph_V \maps K_\bl( \bs N_\ast V_\bl) \xto{\cong} \Wbar{(V)}
\]
natural in $V$  whose induced isomorphism on normalized complexes is the chain map
\[
N_n(\ph_V)(\bs x) = (-1)^n(x,0,\ldots,0) + (-1)^{n-1}(0,d_0 x, 0, \ldots,0) \in 
N_n\Wbar{(V)}.
\] 
\end{proposition}
\begin{proof}
We apply Lemma \ref{lem:Wbarinh} and first exhibit a natural isomorphism 
\[
\ph'_V \maps K_\bl( \bs N_\ast V) \xto{\cong} \Wbarin{(V)}
\]
into the inhomogeneous $\Wb$ construction \eqref{eq:Winh}. From Eq.\ \ref{eq:simp-K}, we see that the subspace of non-degenerate $n$-simplices of $K_\bl( \bs N_\ast V)$ is $\bs N_n(V)$. We will show that $\ph'_V$ is the unique simplicial map whose restriction is
\begin{equation} \label{eq:K-Wbar-vect1}
{\ph'_V}\vert_{\bs N_n(V)}(\bs x):= (-1)^{n-1}(x,0,\ldots,0) \in \Wbin_n(V).
\end{equation}
Proceed by induction, and assume that Eq.\ \ref{eq:K-Wbar-vect1} induces a well-defined morphism $\ph'_V$ on the $n$-skeleton of $K_\bl( \bs N_\ast V)$. We extend to $\sk_{n+1}$ by showing ${\vph_V}\vert_{\bs N_{n+1}(V)}$ satisfies the conditions \eqref{eq:sk} of Lemma \ref{lem:sk}. If $i \geq 1$, then $d_i\vert_{N_{n+1}(V)} =0$, and $d^{\Wbin}_i(x,0,\ldots,0) =0$ by definition of the face maps \eqref{eq:Winh-maps}. For the zeroth face map, writing $d^{\Wbin}_0$ additively, gives 
$d^{\Wbin}_0(x,0,\ldots,0) = -d_{\bs N(V)}x$. Hence, we have ${\ph'_V}_n(d_0 \bs x) = d^{\Wbin}_0{\ph'_V}_{n+1}(\bs x)$. The simplicial map $\ph'_V$ is an isomorphism since, by definition of the face maps \eqref{eq:Winh-maps}, we have the equality 
\[
N_n\Wbarin(V) = \{ (v,0,\ldots,0) \in \Wbin_n(V) \st \ v \in N_{n-1}(V_\bl)\}.
\]    
for all $n \geq 0$. Finally, define  $\ph_V \maps K_\bl( \bs N_\ast V) \xto{\cong} \Wbar{(V)}$ as the composition $\phi_{\Wb} \cc \ph'_V$,
where $\phi_{\Wb} \maps \Wbarin{(V)} \xto{\cong} \Wbar{(V)}$ is the natural isomorphism from Lemma \ref{lem:Wbarinh}. Let $\bs x \in \bs N_n(V_\bl)$.
Writing $\phi_{\Wb}$ additively, we conclude that
\[
\begin{split}
N_n(\ph_V)(\bs x) & = (-1)^{n-1}\phi_{\Wb}((x,0,\ldots,0))\\ 
&= (-1)^n(x,0,\ldots,0) + (-1)^{n-1}(0,d_0 x, 0, \ldots,0). 
\end{split}
\]  
\end{proof}

\subsection{Simplicial Dold-Kan adjunction for coalgebras} \label{sec:sim-dg-coalg}
Let $\dgcocom$ denote the category of homologically graded conilpotent counital cocommutative dg coalgebras concentrated in non-negative degrees. Given such a coalgebra $(C,d_C,\Delta_C)$, we construct a cocommutative comultiplication 
\[
\Delta_{K_\bl C} \maps K_\bl(C) \to K_{\bl}(C) \tensor K_\bl(C)
\]
in the following way. Keeping in mind the decomposition \eqref{eq:simp-K}, if $x \in K_n(C)=C_n$ is a non-degenerate $n$-simplex then define
\begin{equation} \label{eq:K-comult}
\Delta_{K_nC}(x):= (\EM_{KC,KC})_n \circ \Delta_C(x) \in K_n(C) \tensor K_n(C).
\end{equation}
If $x = s_{j_\el} s_{j_{\el -1}} \cdots s_{j_1} c$ is a degenerate $n$-simplex with $c \in C_{n-\el}$ then we necessarily define
\[
\Delta_{K_nC}(x):= (s_{j_\el} s_{j_{\el -1}} \cdots s_{j_1} \tensor s_{j_\el} s_{j_{\el -1}} \cdots s_{j_1}) \Delta_{K_{n-\el}C}(c). 
\]
Since $\EM_{KC,KC}$ is both a symmetric and monoidal natural transformation, $\Delta_{KC}$ is a cocommutative and coassociative comultiplication. It is clear from the explicit construction that, level-wise, each counital coalgebra $(K_nC,\Delta_{K_n C})$ is conilpotent with coaugmentation $\kk = K_n(\kk) \to K_nC$. As a result we have defined a functor
\[
K^{\coalg}_\bl \maps \dgcocom \to \scocom.
\]

Transporting the observations of Schwede-Shipley on the monoidal Dold-Kan correspondence \cite{SS} over to coalgebras demonstrates that $N_\ast$ is not an adjoint of $K^\coalg_\bl$, nor is the unit $\eta_\DK$ a comonoidal transformation. On the other hand, we have   
\begin{lemma}\label{lem:coalg-adjoint}
The natural isomorphism $\eps_\DK \maps N_\ast K_\bl \to \id_{\Ch_\ast}$ lifts to a comonoidal transformation $\eps^{\coalg}_{\DK} \maps N_\ast K^{\coalg}_\bl \to \id_{\cat{dgCoCom}}$, and the functor $K^{\coalg}_\bl(-)$ is fully faithful. Furthermore, 
$K^{\coalg}_\bl(-)$ preserves small colimits. 
\end{lemma}
\begin{proof}
The first two assertions follow by reversing the arrows in the proofs of \cite[Lem.\ 2.11]{SS} and \cite[Prop.\ 2.13]{SS}, respectively. Colimits in $\dgcocom$ (resp.\ $\scocom$), are created by the forgetful functors to $\Chain$ (resp.\ $\sVect$). Hence, since $K_\bl \maps \Chain \to \sVect$ is an equivalence, $K^\coalg_\bl$ preserves colimits.
\end{proof}

\subsubsection{Chevalley-Eilenberg versus simplicial universal enveloping} \label{sec:TheMap}

Let $(\g_\bl,\brac_\bl)$ be a simplicial Lie algebra. Its normalized complex 
is a dg Lie algebra $(N_\ast\g_\bl, d^\g_0, \Brac)$ where 
\[
\bbrac{x}{y}:= \brac \cc \EM_{\g,\g}(x \tensor y).  
\]

Let $\CE_\ast(N \g_\bl):=\bigl(\codgSym(\bs N \g_\bl),\del_\CE \big) \in \dgcocom$ denote the Chevalley-Eilenberg coalgebra of $N_\ast\g_\bl$. 
Recall, e.g.\ \cite[\Sec 2.5]{R}, that the codifferential is uniquely determined by its restrictions to length 1 and length 2 tensors followed by projection to primitives: $(\del_\CE)^1_1 \maps \bs N_\ast\g_\bl \to \bs N_\ast\g_\bl$, and $(\del_\CE)^1_2 \maps \Sym^2_\ast(\bs N \g_\bl) \to \bs N_\ast\g_\bl$, respectively. In this case, we have 
\begin{equation} \label{eq:CE-diff}
(\del_\CE)^1_1(\bs x) = \bs d^\g_0 x, \qquad (\del_{\CE})^1_2(\bs x \, \bs y) = (-1)^{\deg{x} - 1} \bs \bbrac{x}{y}.
\end{equation}
Let $(\cU(\Ng) ,d_{\cU})$ denote the dg cocommutative coalgebra underlying the universal enveloping algebra of $\Ng$. As in \cite[Thm. 2.3]{Quillen:RHT}, the PBW map is an isomorphism of dg coalgebras
\[
\begin{split}
\pbw_{\ast} \maps \codgSym(N\gb) &\xto{\iso} \cU(\Ng), \\
\pbw_{\ast}(x_1 x_2\cdots x_k) &:=  \frac{1}{k!} \sum_{\si \in S_n} \eps(\si) x_{\si(1)} \mdot x_{\si(2)} \mdot \, \cdots \, \mdot x_{\si(k)}.
\end{split}
\]

\subsubsubsection{Comparing $\cU(\Ng)$ with $\cU(\gb)$}
We construct a morphism in $\sconil$ between the simplicial coalgebra $K^\coalg_\bl \cU(N \g)$ and the simplicial universal enveloping algebra $\cU(\g_\bl)$. Let $\pr_{N\g} \maps \codgSym(N\gb) \to \Ng$ in $\Chain$ denote the projection to primitives. Let us consider the morphism in $\sVect$:
\[
f \maps K_\bl\cU(N_\ast\g) \to K_\bl N_{\ast} \g = \g_\bl, \quad f:= K_\bl( \pr_{N\g} \cc \pbw^{-1}_{\ast}).
\] 
Let $F \maps K^\coalg_\bl\cU(N_{\ast}\g) \to \Sym^{\coalg}_{\bl}(\g)$ denote the unique simplicial coalgebra morphism such that $\pr_{\g_\bl}F =f$.  Composing this with simplicial PBW isomorphism, e.g.\ Eq.\ \ref{eq:pbw} gives the desired morphism in $\sconil$
\begin{equation} \label{eq:rho}
\rho \maps K^\coalg_\bl\cU(N_{\ast}\g) \to \cU(\g_\bl), \qquad \rho:= \pbw_\bl \cc F.
\end{equation}  

\begin{lemma}\label{lem:rho}
Let $y \in N_p(\g) \sse \g_p$ and $x \in N_q(\g) \sse \g_q$. Denote by $y \mdot \, x \in \cU(N\g)_{p+q}$ their product in the dg Hopf algebra $\cU(N_\ast\g)$. Then $\rho(y) = y  \in \cU(\g_p)$ and
\begin{equation} \label{eq:rho-xy}
\rho(y \mdot x) = \hspace{-2ex} \sum_{(I,J) \in \Sh(p,q)} \hspace{-3ex} \sgn(I,J) \, s^\g_{J} y \ast s^{\g}_I x ~ \in \cU(\g_{p+q}),
\end{equation}
where $\ast \maps \cU(\g_\bl) \tensor \cU(\g_\bl) \to \cU(\g_\bl)$ denotes the product in the simplicial Hopf algebra $\cU(\g_\bl)$. 
\end{lemma}
\begin{proof}
Let $n=p+q$. We recall the explicit formula, e.g.\ \cite[Lem.\ 22.1]{FHT}, for $F$. For $z \in K^{\coalg}_n(\cU(\Nag))$ and $z \notin \kk$,  
\begin{equation} \label{eq:FHT}
F(z) = \sum_{k=1}^\infty f^{(k)} \cc \ov{\diag}_k(z),
\end{equation}
where $f^{(k)} \maps K_{n}(\cU(\Nag))^{\tensor k} \to {\Sym^k}(\g_n)$ is the linear map
\[
\begin{split}
&f^{(k)}(z_1 \tensor \cdots \tensor z_k):= \frac{1}{k!} f(z_1)f(z_2) \cdots f(z_k).
\end{split}
\]
and $\ov{\diag}_k \maps K_{\bl}(\cU(\Nag)) \to K_{\bl}(\cU(\Nag))^{\tensor k}$ is the $k$th reduced iterated diagonal defined recursively, as in \eqref{eq:red_diag}:  
\[
\ov{\diag}_1(z) := z, \quad  \ov{\diag}_2(z) := \Delta_{K_\bl(\cU(N_\ast \g))}(z) - z \tensor 1 - 1 \tensor z. \quad \text{etc.}  
\]
Using Eq.\ \ref{eq:FHT}, along with the explicit formula \eqref{eq:K-comult} for $\Delta_{K_\bl(\cU(\Nag))}$ and the identities
\[
y \mdot x = \pbw_\ast(yx) + \frac{1}{2} \pbw_\ast(\bbrac{y}{x}), \quad u \ast v = \pbw_\bl(uv) + \frac{1}{2} \pbw_\bl([u,v])
\]
for $u,v \in \g_{n}$, we obtain by direct computation the desired equality \eqref{eq:rho-xy}.
\end{proof}

\subsubsubsection{The universal $\cU(\Ng)$ bundle over $\CE(\Ng)$}
To keep the notation manageable, let $S(\bs \Ng):= \codgSym(\bs \Ng)$ from here on.   
Let $(E\cU(\Ng),\del_E) \in \dgcocom$ denote the dg cocommutative coalgebra whose underlying graded coalgebra is the product
\[
E\cU(\Ng):= \cU(\Ng) \tensor S(\bs \Ng), 
\]
with the following codifferential \cite[p.\ 302]{FHT}: Given $z \in \cU(\Ng)$ and $w=\bs x_1 \bs x_2 \cdots \bs x_k \in S(\bs  \Ng)$, define
\begin{equation} \label{eq:EU-diff}
\begin{split}
\del_{E}(z \tensor w) &:= (-1)^{\deg{z}}(z \tensor 1)\cdot  (1 \tensor \del_{\CE} w) + d_{\cU} z \tensor w + (-1)^{\deg{z}}\theta(z \otimes w),\\
\theta(z \otimes w) &:= \sum_{i=1}^k (-1)^{\deg{w_i}\deg{x_i} + n_i} z \mdot x_i \tensor w_i,
\end{split}
\end{equation}
where 
\begin{equation} \label{eq:codiff-sign}
w_i:=\bs x_1 \bs x_2 \cdots \bs x_{i-1} \cdot 1_\kk \cdot \bs x_{i+1} \cdots \bs x_k, \qquad n_i := \deg{w} + \sum_{j=i+1}^k \deg{\bs x_j}\deg{\bs x_i}.
\end{equation}
We recall that $E \cU(\Ng)$ is an acyclic $\cU(\Ng)$ module \cite[Prop.\ 22.3; 22.6]{FHT} in $\dgcocom$, which sits in a fiber sequence of dg coalgebras
\[ 
\cU(\Ng) \to E\cU(\Ng) \fib \CE(\Ng),
\]
and, as shown in \cite[Prop.\ 6.2]{Quillen:RHT}, exhibits the universal principal $\cU(\Ng)$ dg coalgebra bundle over $\CE(\Ng)$. Crucially for what follows, the sequence splits in $\gcocom$:
\begin{equation} \label{eq:CE-split}
S(\bs \Ng) \xto{i_\ast} E\cU(\Ng), \qquad i_\ast(\bs x_1 \bs x_2 \cdots \bs x_k):=1 \tensor \bs x_1 \bs x_2 \cdots \bs x_k.
\end{equation}    

\subsubsubsection{Comparing $K^{\coalg}_\bl(E\cU(\Nag))$ with $\cW_\bl \cU(\g)$} \label{sec:TheMap2}
Denote by $\pi_\cU \maps E\cU(\Ng) \to \cU(\Ng)$ and $\pi_{S} \maps E\cU(\Ng) \to S(\bs \Ng)$ the canonical projections in $\gcocom$
\[
\pi_{\cU}(z\tensor w): = \eps_S(w)z, \quad \pi_{S}(z\tensor w): = \eps_\cU(z)w, \quad 
\]
where $\eps_\cU \maps \cU(\Ng) \to \kk$ and $\eps_S \maps S(\bs \Ng) \to \kk$ are the respective counits. Applying the functor $K^{\coalg, +}_\bl:=U^{+}_\bl K^{\coalg}_\bl$ to these maps induces a morphism in $\spconil$ into the product of almost simplicial coalgebras
\[
\begin{split}
\alpha \maps K^{\coalg, +}_\bl(E\cU(\Nag)) \to K^{\coalg, +}_\bl\cU(\Nag) \tensor K^{\coalg, +}_\bl (S(\bs \Nag)  )\\
\al(z \tensor w):= \bigl( K^{\coalg, +}_\bl(\pi_{\cU}) \tensor K^{\coalg, +}_\bl(\pi_{S}) \bigr) \Delta_{K_\bl E\cU(\Nag)}.
\end{split}
\]
Let $z \tensor w \in K^{\coalg, +}_{p+q}(E\cU(\Nag))$ denote a non-degenerate element with $z \in \cU(\Ng)_p$ and $w \in S(\bs \Ng)_q$. A direct computation using formula \eqref{eq:K-comult} for the comultiplication $\Delta_{K_\bl E\cU(\Nag)}$ implies that
\[
\al(z \tensor w) = \hspace{-2ex} \sum_{(I,J) \in \Sh(p,q)} \hspace{-3ex} \sgn(I,J) s_{J} z \tensor s_{I} w.
\]
Denote by
\[
\iota \maps \Kcoalgp_{\bl}(S(\bs \Nag)) \to \Kcoalgp_{\bl}(E\cU(\Nag))
\]
the inclusion of almost simplicial coalgebras induced by the inclusion $i_\ast$ in \eqref{eq:CE-split}. Also, in what follows, denote by
\[
\dE_{0} \maps \Kcoalg_{n}(E\cU(\Nag)) \to \Kcoalg_{n-1}(E\cU(\Nag))
\]
the zeroth face map of the simplicial coalgebra $\Kcoalg_{n}(E\cU(\Nag))$. 

Now we will iteratively construct a collection of coalgebra morphisms 
\begin{equation} \label{eq:vph-def}
\vph_n \maps \Kcoalg_{n}(E\cU(\Nag)) \to \cW_n \cU(\g) = \cU(\g_n) \tensor \cW_{n-1} \cU(\g) \quad n \geq 0,
\end{equation} 
where $\cW_\bl \cU(\g) \in \sconil$ denotes the total space for the universal $\cU(\gb)$-bundle described in Sec.\ \ref{sec:Wbar}. 
Our construction is similar to the one used by Cartan \cite[\Sec 5, Prop.\ 1]{Cartan} to compare the bar construction with normalized chains on the $\Wb$-construction. See also the exposition \cite[Thm.\ A.20]{Gug-May} by Gugenheim and May. To begin, since $\Kcoalg_{0}(E\cU(\Nag)) = E\cU(\Ng)_0 = \cU(\g_0) \tensor \kk$, define $\vph_0(z \tensor 1) = z$, and for $n \geq 1$:
\begin{equation} \label{eq:vph}
\begin{split}
\vph_n &:= \bigl( \ro_n \tensor (\vph_{n-1} \cc \dE_{0} \cc \iota_n) \bigr) \cc \al_n\\
\vph_n(z \tensor w) &= \hspace{-2ex} \sum_{(I,J) \in \Sh(p,q)} \hspace{-3ex} \sgn(I,J) ( \rho_n(s_J z) \tensor 1) \cdot \si_{n-1}
\vph_{n-1}\dE_{0}\iota_n(s_I w).
\end{split}
\end{equation}
Above, $\rho$ is the morphism of simplicial coalgebras \eqref{eq:rho}, $\si_{n-1} \maps \cW_{n-1}\cU(\g) \to \cW_{n}\cU(\g)$ 
is the extra degeneracy \eqref{eq:extra1}, and ``$\cdot$'' denotes the multiplication in the tensor product of associative algebras $\cU(\g_n) \tensor \cW_{n-1} \cU(\g)$.

For technical reasons, it will be useful to also consider the graded subcoalgebra $(\kk \semiop \Ng) \tensor S(\bs \Ng)$
which lies in between $S(\bs \Ng) = \kk \tensor S(\bs \Ng)$ and $\cU(\Ng) \tensor S(\bs \Ng)$. As in \eqref{eq:prim-adj}, $\kk \semiop \Ng$ is the graded vector space $\kk \dsum \Ng$ equipped with the obvious comultiplication. Let 
\begin{equation} \label{eq:j-inc}
\jm \maps \Kcoalgp_\bl\bigl( (\kk \semiop \Nag) \tensor S(\bs \Nag) \bigr) \to \Kcoalgp_\bl(E\cU(\Nag))
\end{equation}
denote the inclusion of almost simplicial coalgebras. For each $n \geq 0$, composition with \eqref{eq:vph-def} gives the coalgebra morphism
\begin{equation} \label{eq:tvph}
\tvp_{n} \maps \Kcoalg_n\bigl( (\kk \semiop \Nag) \tensor S(\bs \Nag) \bigr) \to \cW_n\cU(\g), \qquad \tvp_n:=\vphi_{n} \cc \jm.
\end{equation}

\begin{proposition} \label{prop:themap}
The collection of coalgebra morphisms $\{\tvp_n\}$ assembles into a morphism of almost simplicial coalgebras
\[
\tvp \maps \Kcoalgp_\bl\bigl( (\kk \semiop \Nag) \tensor S(\bs \Nag) \bigr) \to U^{+}_\bl\cW\cU(\g)
\]
which has the following property:
\begin{equation} \label{eq:map-prop}
d^{\cW}_0 \tvp_m = \vph_{m-1}\dE_0 \jm \qquad \forall m \geq 0.
\end{equation}
\end{proposition}

We defer the proof of the proposition to Appendix \ref{sec:map-app}. 

\subsubsubsection{The comparison of $K^{\coalg}_\bl(\CE(\Nag))$ with $\Wbar \cU(\g)$}
Let 
\[
\sk^{+}_nK_\bl \bigl( (\kk \semiop \Nag) \tensor S(\bs \Nag)\bigr) \in \spVect
\]
denote the $n$-skeleton of the almost simplicial vector space
$K^{+}_\bl \bigl( (\kk \semiop \Nag) \tensor S(\bs \Nag) \bigr)$. We arrive at a key theorem.
\begin{theorem}\label{thm:themap}
The composition of morphisms of almost simplicial coalgebras
\[
\Kcoalgp_\bl(S(\bs \Nag)) \xto{\iota} \Kcoalgp_\bl(E\cU(\Nag)) \xto{\vph} U^{+}_\bl\cW \cU(\g) \xto{U^{+}(\pi^{\cW})} U^{+}_\bl\Wb{\cU(\g)},
\]
where $\pi^\cW$ denotes the quotient map \eqref{eq:piW}, lifts to a morphism of simplicial coalgebras 
\[
\psi \maps \Kcoalg_\bl(\CE(\Nag)) \to \Wbar{\cU(\g)}
\]
in $\sconil$.
\end{theorem}
\begin{proof}
Since $\psi:= \pi^{\cW} \cc \vph \cc \iota$ is a level-wise coalgebra morphism which is compatible with all degeneracies and positive face maps, it suffices to show that $\psi$ lifts from $\spVect$ to $\sVect$. Since the simplicial vector space $K_\bl(S(\bs \Ng))$ decomposes dimension-wise into a direct sum of normalized and degenerate simplices, it suffices, by Lemma \ref{lem:sk}, to verify for all $n \geq 0$
\[
d^{\Wb}_0 \psi_{n+1}(w) = \psi_{n} (d^{K\CE}_0 w),   
\]       
for all $w \in N_{n+1}K_\bl(S(\bs \Nag) = S(\bs \Ng)_{n+1}$. Let us factor the inclusion $\iota$ using \eqref{eq:j-inc}:
\[
\Kcoalgp_\bl(S(\bs \Nag)) \to\Kcoalgp_\bl\bigl( (\kk \semiop \Nag) \tensor S(\bs \Nag) \bigr) \xto{\jm} \Kcoalgp_\bl(E\cU(\Nag)),
\] 
so that we have $\psi_{n+1}(w) = \pi^{\cW}_{n+1}\tvp_{n+1}(1 \tensor w)$, where $\tvp$ is the almost simplicial map \eqref{eq:tvph}.
Let $w=\bs x_1 \bs x_2 \cdots \bs x_m \in S(\bs  \Ng)_{n+1}$, and for each $i=1, \ldots,m$, define $w_i:=
\bs x_1 \bs x_2 \cdots \wh{\bs x_i} \cdots \bs x_m$. Then, we have the equalities
\[
d^{\Wb}_0 \psi_{n+1}(w) = \pi^{\cW}_{n} d^{\cW}_0 \tvp_{n+1}(1 \tensor w) = \pi^{\cW}_{n} \vph_{n}\dE_0(1 \tensor w),
\] 
where the last equality follows from \eqref{eq:map-prop} in Prop.\ \ref{prop:themap}. As in \eqref{eq:simp-norm}, the zeroth face map $\dE_0$ equals the differential $\delta_E$ on $E\cU(\Ng)$. Therefore, we have     
\[
d^{\Wb}_0 \psi_{n+1}(w) = \pi^{\cW}_{n} \vph_{n} (1 \tensor \del_\CE(w)) + \pi^{\cW}_{n} \vph_{n}\theta(1 \tensor w).
\]
where, as in \eqref{eq:EU-diff}, $\theta(1 \otimes w)= \sum_{i=1}^m (-1)^{\deg{w_i}\deg{x_i} + n_i} x_i \tensor w_i$.
From the recursive definition \eqref{eq:vph} of $\vph_n$, we obtain
\[
\vph_{n}\theta(1 \tensor w) = \sum_{i=1}^m \sum_{(I,J) \in \Sh( \deg{x_i},\deg{w_i})}\hspace{-3ex} (-1)^{\deg{w_i}\deg{x_i} + n_i}\,  \sgn(I,J) ~ s^\g_{J} x_i \tensor \vph_{n-1}\dE_0(1 \tensor s_{I} w_i).
\]
Since $s^\g_{J} x_i$ lies in the kernel of the counit $\cU(\g_n) \to \kk$, it follows from the definition \eqref{eq:piW} of $\pi^{\cW}$ that
\[
\pi^{\cW}_n\vph_{n}\theta(1 \tensor w) = 0.
\]
Therefore, we conclude
\[
d^{\Wb}_0 \psi_{n+1}(w) = \pi^{\cW}_{n} \vph_{n} (1 \tensor \del_\CE(w)) = \pi^{\cW}_{n} \vph_{n} \iota(\del_\CE(w))
= \psi_n(d^{K\CE}_0 w).
\]
\end{proof}

\subsubsection{Complexes of primitives}
Recall that the adjunction \eqref{eq:prim-adj} for the primitives functor lifts to an adjunction of simplicial objects
\[
\kk \semiop^{\coalg}_{\bl}(-) \maps s\Vect \adjunct \scocom \maps \prim_\bl. 
\]
There is an analogous adjunction for dg coalgebras
\[
\kk \semiop^{\coalg}_{\ast}(-) \maps \Chain \adjunct \dgcocom \maps \prim_\ast.
\]
Given a chain complex $(V,d)$, $\kk \semiop^{\coalg}_{\ast}(V)$ is the dg cocommutative coalgebra 
whose underlying chain complex is $\kk \oplus V$, with comultiplication $\Del(v):= v \tensor 1 + 1 \tensor v$ for all $v \in V$. By using the construction of the comultiplication \eqref{eq:K-comult} for $K^\coalg_\bl(C)$, there is an isomorphism
\begin{equation} \label{eq:prim-iso}
K^\coalg_\bl\bigl(\kk \semiop^{\coalg}_{\ast}(V) \bigr) \cong  \kk \semiop^{\coalg}_{\bl}(K_\bl(V))
\end{equation} 
which is natural with respect to $(V,d)$.

\begin{proposition} \label{prop:K-prims}
Let $C \in \dgcocom$. There is an isomorphism of chain complexes
\[
\prim_\ast(C) \xto{\cong}    N_\ast \prim_\bl K^{\coalg}(C) 
\] 
natural in $C$.
\end{proposition}
\begin{proof}
The categories $\dgcocom$ and $\scocom$ are locally presentable\footnote{
The reference \cite{Pos} concerns categories of all cocommutative coalgebras, but a full bicomplete subcategory of a locally presentable category is locally presentable \cite{AR:2015}.},  e.g.\ \cite[Thm. 3.2]{Pos}, and $K^\coalg_\bl(-)$ preserves colimits by Lemma \ref{lem:coalg-adjoint}. Hence, $K^\coalg_\bl(-)$ has a right adjoint
$R_\ast \maps \scocom \to \dgcocom$, e.g., by the proof of \cite[Thm.\ 2.2.]{Nlab:AFT}. Since $K^\coalg_{\bl}$ is full and faithful, the unit of this adjunction is an isomorphism $\ti{\eta} \maps \id_{\dgcocom} \xto{\cong} R_\ast K^\coalg_\bl$. Using this, along with \eqref{eq:prim-iso}, we obtain a sequence of natural isomorphisms for each $V \in \Chain$:
\[
\begin{split}
\hom_{\Chain}(V,\prim_\ast C) & \cong  \hom_{\Chain}(V,\prim_\ast R_\ast K^{\coalg}_\bl(C)) \\ 
&\cong \hom_{\dgcocom}( \kk \semiop^{\coalg}_{\ast}(V), R_\ast K^{\coalg}_\bl(C))\\
& \cong \hom_{\scocom}( K^\coalg_\bl(\kk \semiop^{\coalg}_{\ast}(V)), K^{\coalg}_\bl(C))\\
& \cong \hom_{\scocom}(\kk \semiop^{\coalg}_{\bl}(K_\bl(V)), K^{\coalg}_\bl(C))\\
& \cong \hom_{\sVect}(K_\bl(V), \prim_\bl K^{\coalg}(C))\\
& \cong \hom_{\Chain}(V, N_\ast \prim_\bl K^{\coalg}(C)).
\end{split}
\]   
Hence, $\prim_\ast C \cong N_\ast \prim_\bl K^{\coalg}(C)$.
\end{proof}

\begin{corollary}\label{cor:prim-map}
Let $\gb$ be a simplicial Lie algebra, and $\psi \maps \Kcoalg_\bl(\CE(\Nag)) \to \Wbar{\cU(\g)}$ the morphism of simplicial coalgebras constructed in Thm.\ \ref{thm:themap}. Then 
\[
\bs \Ng \cong N_\ast \prim_\bl\Kcoalg(\CE(\Nag)),
\] 
and the induced morphism on complexes of primitives
\[
N_\ast \prim_\bl(\psi) \maps \bs \Ng \to N_\ast \prim_\bl \Wb{\cU(\g)}
\]
is an isomorphism.
\end{corollary}
\begin{proof}
We will need an explicit description of the evaluation of $\psi$ on non-degenerate primitives.
Let $\vph_n \maps \Kcoalg_{n}(E\cU(\Nag)) \to \cW_n \cU(\g)$ be the coalgebra morphism \eqref{eq:vph-def}.
First, if $y\tensor 1 \in \cU(\Ng) \tensor S(\bs \Ng)$, with $y \in N_n\gb$, then by Eq.\ \ref{eq:vph}:
\begin{equation} \label{eq:prim-map-1}
\begin{split}
\vph_n(y \tensor 1) &= \bigl( \ro_n \tensor (\vph_{n-1} \cc \dE_{0} \cc \iota_n) \bigr) \cc \al_n(y \tensor 1)\\
&= \bigl (\ro_n \tensor (\vph_{n-1} \cc \dE_{0} \cc \iota_n) \bigr)(y \tensor s^{n}_0(1))\\
&= \rho_n(y) \tensor \vph_{n-1}\dE_0s^{n}_0(1 \tensor 1)\\
&= y \tensor (s^{\cW}_0)^{n-1}\vph_0(1 \tensor 1) \\
& = y \tensor 1^{\tensor n}.
\end{split}
\end{equation}
Next, if $\bs x \in (\bs \Ng)_n$, then, again via Eq.\ \ref{eq:vph}: 
\begin{equation} \label{eq:prim-map-2}
\vph_n(1 \tensor \bs x)= 1 \tensor \vph_{n-1}\dE_0(1 \tensor \bs x).
\end{equation}
Furthermore, by Eqs.\ \ref{eq:EU-diff} and \ref{eq:CE-diff}, we have
\begin{equation} \label{eq:prim-map-3}
\begin{split}
\dE_0(1 \tensor \bs x) & = \del_E(1 \tensor \bs x)\\
& = 1 \tensor \del_{\CE}(\bs x) + \tha(1 \tensor \bs x) \\
& = 1\tensor \bs d^{\g}_0 x + (-1)^{\deg{1_\kk} \deg{x} + \deg{\bs x}} x \tensor 1\\
& = 1 \tensor \bs d^{\g}_0 x + (-1)^n x \tensor 1. 
\end{split}
\end{equation}
From this, we claim that for all $n \geq 1 $ and  $\bs x \in (\bs \Ng)_n$:
\begin{equation} \label{eq:prim-claim}
\vph_n(1 \tensor \bs x) = (-1)^{n-1} 1^{\tensor 2} \tensor d^\g_0 x \tensor 1^{\tensor n-2}
+ (-1)^{n} 1 \tensor x \tensor 1^{\tensor n-1}.
\end{equation}
Proceeding by induction, and assuming \eqref{eq:prim-claim} holds for $n \geq 1$, let $\bs y \in (\bs \Ng)_{n+1}$.
From Eq.\ \ref{eq:prim-map-3}, we have
\[
\vph_{n+1} ( 1 \tensor \bs y) = 1 \tensor \vph_{n}\bigl( 1 \tensor \bs d^{\g}_0 y + (-1)^{n+1} y \tensor 1 \bigr).
\]
Eq.\ \ref{eq:prim-map-1} implies that $(-1)^{n+1} \vph_{n} ( y \tensor 1)$ = $(-1)^{n+1} y \tensor 1^{\tensor n}$, while the induction hypothesis \eqref{eq:prim-claim} gives $\vph_{n}( 1 \tensor \bs d^{\g}_0 y ) = (-1)^n 1 \tensor d^{\g}_0 y \tensor 1^{\tensor n-1}$. This proves the claim.

Now, let $\bs x \in (\bs \Ng)_n$, and evaluate:
\[
\begin{split}
N_\ast\prim_\bl(\psi)(\bs x) & = \psi_n(\bs x) = \pi^{\cW}_n \vph_n(1 \tensor \bs x)   \\
& = (\eps_{\cU(\g_n)} \tensor \id_{\cW_{n-1} \cU(\g)})\bigl(
(-1)^{n-1} 1^{\tensor 2} \tensor d^\g_0 x \tensor 1^{\tensor n-2}
+ (-1)^{n} 1 \tensor x \tensor 1^{\tensor n-1} \bigr)\\
&=(-1)^{n-1} 1 \tensor d^\g_0 x \tensor 1^{\tensor n-2} + (-1)^n x \tensor 1^{\tensor n-1}.
\end{split}
\] 
On the other hand, by Prop.\ \ref{prop:K-Wbar-vect} and Cor.\ \ref{cor:prim-iso},
we have isomorphisms of chain complexes
\[
\begin{split}
N_\ast\phi_{\und{\g}} \maps \bs N_\ast (\und{\g}) &\xto{\cong} N_\ast \Wbar{(\und{\g})}\\
N_\ast \xi_{\g} \maps N_\ast\Wbar{(\und{\g})} &\xto{\cong} N_\ast \prim_\bl \Wb{\cU(\g)}. 
\end{split}
\]
As described in Eqs.\ \ref{eq:K-Wbar-vect1} and \ref{eq:cor-prim-map}, respectively, these isomorphisms are constructed from morphisms
$\phi'_{\und{\g}} \maps \bs N_\ast (\und{\g}) \to \Wbar{(\und{\g})}$, and 
$\ga \maps \Wbarin{(\und{\g})} \to \prim_\bl(\Wbin\cU(\g))$, along with the natural isomorphism $\Wbarin(-) \xto{\cong} \Wbar(-)$ of Lemma \ref{lem:Wbarinh}.
Therefore, the evaluation of the composition $N_\ast(\xi_{\g} \cc \phi_{\und{\g}})$ with
$\bs x \in (\bs \Ng)_n$ yields
\[
\begin{split}
\xi_{\g}\cc \phi_{\und{\g}}( \bs x) &= \phi_{\Wb \cU(\g)} \ga \bigl( \phi'_{\und{\g}}(\bs x) \bigr)\\
&=(-1)^{n-1}\phi_{\Wb \cU(\g)} \bigl(\ga(x,0,\ldots,0)\bigr)\\
&=(-1)^{n-1}\phi_{\Wb \cU(\g)}(x \tensor 1^{\tensor n-1}) \\
&=(-1)^{n}x \tensor 1^{\tensor n-1} + (-1)^{n-1}1 \tensor d^\g_0x \tensor 1^{\tensor n-2}.
\end{split}
\] 
We conclude that $N_\ast \prim_\bl(\psi) =  N_\ast\bigl( \phi_{\und{\g}} \cc \xi_{\g} \bigr)$, and hence $N_\ast \prim_\bl(\psi)$ is an isomorphism.
\end{proof}

\subsection{Cosimplicial and almost cosimplicial Dold-Kan correspondence} \label{sec:cosimDK}
Applying the classical Dold-Kan correspondence from Sec.\ \ref{sec:simp-DK} to the abelian category $\Vect^{\op}$ yields the cosimplicial Dold-Kan correspondence. This is the adjoint equivalence
\begin{equation} \label{eq:cDK}
\Nas \maps \cVect \adjunct \cCh \maps \Kb, \qquad \eps_{\DK}^\vee \maps \Nas \Kb \xto{\cong} \id_{\cCh}, \quad  \eta_{\DK}^\vee \maps \id_{\cVect} \xto{\cong}  \Kb \Nas.
\end{equation}
We follow the exposition for the cosimplicial Dold-Kan correspondence given in \cite{F2} and \cite{BGill}. In particular, $\Nas(-)$ denotes the normalized cochains functor constructed by intersecting the kernels of co-degeneracy maps. Explicitly,  
$N^0(A^\bul):= A^0$ and for $n \geq 1$:
\[
N^n(A^\bl):= \bigcap^{n-1}_{i=0} \ker \bigl( s^{i} \maps A^{n} \to A^{n-1} \bigr),
\]
The differential $\del \maps N^{n}(A^\bl) \to N^{n+1}(A^\bl)$  is defined using the coface maps
$\delta := \sum_{i=0}^{n+1} (-1)^{n+1 +i} d^i$. 

The cosimplicial Eilenberg-Mac Lane shuffle map \cite[\Sec 7.1]{BGill} gives a symmetric comonoidal natural quasi-isomorphism on the normalized cochains functor
\begin{equation} \label{eq:EM-dual}
\EM^\bl_{AB} \maps \Nas(A^\bl \tensor B^{\bl}) \weq \Nas(A^\bl) \tensor \Nas(B^\bl). 
\end{equation}
Since the shuffle map is symmetric comonoidal, if $C^{\bl}$ is a cosimplicial cocommutative coalgebra over $\kk$, then $\Nas(C^{\bl})$ is naturally a dg cocommutative coalgebra. For example, the diagonal map $\Del_n \to \Del_n \times \Del_n$ gives $\kk\Del_n$ the structure of a cocommutative cosimplicial coalgebra. Thus $\Nas(\Del_n)$ is a dg cocommutative coalgebra.

The inverse functor $\Kb(-)$ to normalized cochains can be defined by dualizing the simplicial formulas \eqref{eq:simp-K}. On the other hand, the equivalence \eqref{eq:cDK} yields the identification
\begin{equation} \label{eq:cosimp-K}
K^{n}(C^\ast)  \cong \Hom_{\cCh}\bigl(\Nas(\kk\Del_n),C^\ast \bigr) \quad \forall n \geq 0.
\end{equation}

\subsubsection{Almost cosimplicial Dold-Kan correspondence} \label{sec:alm-cosim}
The almost simplicial constructions of Sec.\ \ref{sec:almost} transfer over to the opposite category. Let $c^+\cC$ denote the almost cosimplicial objects in a category $\cC$. We have the forgetful functor 
\[
U^\bl_+ \maps c\cC \to c^+\cC,
\]
which drops the zeroth coface map $d^0$. (This functor already appeared in Sec.\ \ref{sec:formal-pbw} for $\cC = \clnAlg$.) Analogously, for $\cC$ abelian, denote by
\[
U^\bl_{+} \maps \cCh(\cC) \to \grC^\ast,
\]
which sends a non-negatively graded cochain complex to its underlying graded object. Applying Sec.\ \ref{sec:al-simp-DK} to the category $\cC=\Vect^{\op}$, we obtain the almost cosimplicial Dold-Kan correspondence $N^\ast_{+} \maps \cVect \adjunct \cCh \maps K^\bl_+$ 
and the respective compatibilities:
\[
\begin{split}
N^\ast_+ U^\bl_+  = U^\ast_+ N^\ast, &\quad K^\bl_+ U^\ast_+  = U^\bl_+ K^\bl \\
U^\bl_+ \eta^{\vee}_{\DK} = \eta^{\vee +}_{\DK} U^\bl_+, &\quad U^\ast_+ \eps^{\vee}_{\DK} = \eps^{\vee +}_{\DK} U^\ast_+.
\end{split}
\]

\subsection{Cosimplicial Dold-Kan adjunction for algebras}
\label{sec:cosimpDKalgebras}

\subsubsection{Monadic constructions} \label{sec:monad}
Let $\cCom$ and $\cpCom$ denote the categories of cosimplicial and almost cosimplicial commutative $\kk$-algebras, respectively. Denote by $\cdga$ and $\cga$ the categories of commutative dg algebras, and graded commutative $\kk$-algebras, respectively. All four categories are bicomplete, and 
we have adjunctions
\[
\begin{split}
\Sym^\bl \maps \cVect \adjunct \cCom \maps U^\bl_{\Com}, & \quad 
\Sym^\bl_{+} \maps \cpVect \adjunct \cpCom \maps U^\bl_{\Com^+}\\
\Sym^\ast \maps \cCh \adjunct \cdga \maps U^\ast_{\Com}, & \quad 
\Sym^\ast_{+} \maps \gVect \adjunct \cga \maps U^\ast_{\Com^+}.
\end{split}
\]  
Note that all of the above right adjoints are monadic. As in the simplicial/coalgebra case (Sec.\ \ref{sec:sim-dg-coalg}), the functor $K^\bl$ is compatible with multiplication, and extends to functors
\[
K^\bl_{\Com} \maps \cdga \to \cCom, \quad K^\bl_{\Com^+} \maps \cga \to \cpCom.
\]  
This can readily be seen from Eq.\ \ref{eq:cosimp-K}. Indeed, if $C^*$ is a cdga/cga, then 
\[
\Hom_{\cCh}\bigl(\Nas(\kk\Del_n),C^\ast \bigr)
\]
is equipped with a commutative convolution product.

On the other hand, the functor $K^\bl_{\Com}$ has an alternative description as explained in E.\ Getzler's work \cite{EG:DMC}. For $n \geq 0$, let  
\begin{equation} \label{eq:lam-def}
\Lam^n:=\kk[e_0,e_1,\ldots,e_n] 
\end{equation}
denote the free graded commutative algebra with generators $\{e_0,e_1,\ldots,e_n\}$ all in degree $-1$. The differential is the unique derivation on $\Lam^n$ which sends each generator to the unit $1$. To each map  $\al \maps [m] \to [n]$ in $\Del$, there exists a unique morphism of cdgas $\al_\ast \maps \Lam^m \to \Lam^n$ which sends the generator $e_i$ to $e_{\al(i)} \in \Lam^n$. Hence, $\Lam^\bl$ is a cosimplicial cdga. It follows from \cite[Prop.\ 3.3]{EG:DMC} that the dual of the dg cocommutative coalgebra $\Nas(\kk\Del_n)$ is isomorphic, as a cdga, to $\Lam^n$, and so we obtain an isomorphism of cosimplicial algebras natural in $C^\ast \in \cdga$: 
\begin{equation} \label{eq:K-lam}
K^\bl_{\Com}(C^\ast)  \cong  Z^0 ( C^\ast \tensor \Lam^\bl). 
\end{equation}

\begin{theorem} \label{thm:D-exists}
The functors $\Kb_{\Com} \maps \cdga \to \cCom$ and $\Kb_{\Compp} \maps \cga \to \cpCom$ admit left adjoints 
\[
\Da \maps \cCom \to \cdga, \quad  \Dap \maps \cpCom \to \cga
\]
with the following properties
\begin{enumerate}

\item If $\Vb \in \cVect$ and $\VVb \in \cpVect$, then we have natural isomorphisms
\begin{equation} \label{eq:D-exists-Sym}
\Da(\cSym(V)) \cong \dgSym(N(\Vb)), \quad \Dap(\cpSym(W)) \cong \gSym(N_+(\VVb)).
\end{equation}

\item There is a natural isomorphism of functors
\[
U^*_{+, \Com} \Da  \cong \Dap U^\bl_{+,\Com},  
\]
where $U^\bl_{+,\Com} \maps \cCom \to \cpCom$ and $U^\ast_{+,\Com} \maps \cdga \to \cga$ are the forgetful functors that drop the zeroth coface map and differential, respectively.

\item The counits 
\[
\eps_{\Com} \maps \Da \Kb_{\Com} \to \id_{\cdga}, \qquad  \eps_{\Compp} \maps \Dap \Kb_{\Compp} \to \id_{\cga} 
\]
are isomorphisms. 

\end{enumerate}
\end{theorem}

\begin{remark}\label{rmk:JP}
We learned about the functor $D^\ast(-)$ and its applicability towards ``higher differentiation'' from the work of J.\ Pridham. In particular, \cite[Def.\ 4.20]{JP:Def}, and \cite[Rmk 4.51]{JP:Mixed}, \cite[Ex.\ 3.6]{JP:Poisson}. 
\end{remark}

The functors involved in the proof of Thm.\ \ref{thm:D-exists} are described in Figure \ref{fig:cube} and satisfy the following identities:
\begin{subequations} \label{eq:compat}
\begin{align}
N^\ast_{+}U^{\bl}_{+} = U^{\ast}_{+}N^\ast,& \quad \Kb_{+} U^\ast_{+} = U^\bl_{+}\Kb \label{KN}\\
U^\ast_{+,\Com} \, \Sym^\ast = \Sym^\ast_{+} U^\ast_{+}, &\quad U^\bl_{+,\Com}\, \Sym^\bl = \Sym^{\bl}_{+} U^{\bl}_+ \label{Sym}\\
\Kb_{\PP} U^\ast_{\Com^{\PP}} = U^\bl_{\Com^{\PP}}\Kb_{\Com^{\PP}}, &\quad  \Kb_{\Com^+}U^\ast_{+,\Com} = U^\bl_{+,\Com}K^\bl_{\Com} \label{Kcom}\\
U^\bl_{\Com^+}\, U^\bl_{+,\Com} = U^{\bl}_{+}\, U^{\bl}_{\Com}, & \quad U^\ast_{\Compp}\, U^{\ast}_{+,\Com} = U^\ast_{+} \, U^\ast_{\Com}\label{UU}\\
U^\bl_{+} \cc \eta^{\vee}_{\DK} = \eta^{\vee +}_{\DK}\cc  U^\bl_{+}, &\quad U^\ast_{+} \cc \eps^{\vee}_{\DK} = \eps^{\vee +}_{\DK} \cc U^\ast_{+} \label{Uadj}
\end{align}
\end{subequations}
Note that only certain faces in Figure \ref{fig:cube} commute.

The key ingredient in the proof of Thm.\ \ref{thm:D-exists} is a careful application of the adjoint lifting theorem for monadic functors. 
Our main reference is \cite[Lem.\ 1; Thm.\ 2]{Johnstone}, and we will adopt a compatible notation with this reference when it is reasonable to do so. (See also \cite[\Sec 4.5]{Borceux2}, and the exposition \cite{nlab:ALT}.)

We will use the following notation for our monads and for one of the counits: 
\[
\cH_{\PP}:= U^\ast_{\Com^{\PP}} \Sym^{\ast}_{\PP} , \qquad \cK_{\PP}:= U^\bl_{\Com^{\PP}} \Sym^{\bl}_{\PP}, \qquad  \eps^{\cH} \maps \Sym^\ast_{\PP} U^{\ast}_{\Com^{\PP}} \to \id_{\cat{c(d)ga}}.
\]

Let $V^\ast_{\PP}$ be either a cochain complex or a graded vector space. Apply $\Kb_{\PP}$ to the canonical inclusion $i \maps V^\ast_{\PP} \to U^{\ast}_{\Com^{\PP}} \Sym^\ast_{\PP}(V^\ast_{\PP})$. The universal property of $\Sym^{\bl}_{\PP}$, along with the identities \eqref{Kcom}, yield a morphism of (almost) cosimplicial vector spaces
\[
{\lam_{\PP}}_{V} \maps U^\bl_{\Com^{\PP}} \Sym^\bl_{\PP}\Kb_{\PP}(V) \to 
\Kb_{\PP}U^\ast_{\Com^{\PP}}\Sym^\ast_{\PP}(V).
\]
This defines a natural transformation 
\begin{equation} \label{eq:lam}
\lam_{\PP} \maps  \cK_{\PP} \Kb_{\PP} \to \Kb_{\PP} \cH_{\PP},
\end{equation}
which, as in \cite[Lem.\ 1]{Johnstone}, encodes the fact that $\Kb_{\PP}(-)$ lifts through the forgetful functors to a functor $\Kb_{\Com^{\PP}}(-)$ between algebras.  

It will be convenient to slightly abuse notation, and consider algebras $A^\bl_{\PP} \in c^{\PP}\cat{Com}$ as pairs $A^\bl_{\PP}:=(U^\bl_{\PP}A^\bl_{\PP},\xi)$, where $A^\bl_{\PP}=U^\bl_{\PP}A^\bl_{\PP} \in c^{\PP}\Vect$, and $\xi \maps \Sym^\bl_{\PP}(A^\bl) \xto{\xi} A^\bl$ is the linear morphism in $c^{\PP}\Vect$ that encodes the multiplication\footnote{In other words, we are tacitly identifying $c^{\PP}\cat{Com}$ with the category of algebras over the monad $\cK_{\PP}$.}.     

\begin{proof}[{\bf Proof of Theorem \ref{thm:D-exists}}]
Eq.\ \ref{Kcom} implies that the functor $\Kb_{\Com^{\PP}}$ lifts $\Kb_{\PP}$ through monadic forgetful functors. Therefore, since $\Kb_{\PP}$ has a left adjoint and $\cqga$ is cocomplete, \cite[Thm.\ 2]{Johnstone} implies that $\Kb_{\Com^{\PP}}$ has a left adjoint 
\[D^\ast_{\PP} \adj \Kb_{\Com^{\PP}}
\]
via an explicit construction. Given $A^\bl_{\PP}=(A^\bl_{\PP}, \xi)$, define $\DPP\bigl(A^\bl_{\PP}\bigr)$ as the coequalizer
\begin{equation} \label{eq:Dcoeq}
\DPP\bigl(A^\bl_{\PP}\bigr):= \coeq \Bigl( 
 \xymatrix{
 \Sym^\ast_{\PP} \bigl (N^\ast_{\PP}\cK_{\PP}(A) \bigr)  \ar^*++{f_{A_{\PP}}}@<-.7ex>[rr]
 \ar@<+.7ex>[rr]_*++{g_{A_{\PP}}} && \Sym^\ast_{\PP}\bigl(N^\ast_{\PP}(A) \bigr)
} \Bigr).
\end{equation}
in the category $\cqga$. Above, $f_{A_{\PP}}:= \Sym^{\ast}_{\PP}(N^\ast_{\PP}(\xi))$, and $g_{A_{\PP}}$ is constructed by the following procedure:\\
\begin{enumerate}[label=(\roman*),leftmargin=65pt]
\item Define $g^{\prime}_{A_{\PP}} \maps \cK_{\PP}(A) \to \Kb_{\PP}\cH_{\PP}N^\ast_{\PP}(A)$ to be the composition 
\[
{\lam_{\PP}}_{\Nas_{\PP}A} \cc \cK_{\PP}\bigl(\eta^{\vee \PP}_{\DK}\bigr),
\]
where $\lam_{\PP}$ is the natural transformation \eqref{eq:lam}.\\

\item Define $g^{\prime \prime}_{A_{\PP}} \maps \Nas_{\PP}\cK_{\PP}(A) \to \cH_{\PP} \Nas_{\PP}(A)$ to be the composition
\[
\bigl(\eps^{\vee \PP}_{\DK}\bigr)_{\cH_{\PP}\Nas_{\PP}(A)} \cc \Nas_{\PP}(g^{\prime}_{A_{\PP}}).
\]
\\

\item Finally, define $g_{A_{\PP}} \maps  \Sym^\ast_{\PP} \bigl (N^\ast_{\PP}\cK_{\PP}(A) \bigr)  \to \Sym^\ast_{\PP}\bigl(N^\ast_{\PP}(A) \bigr)$ 
as the composition
\[
{\eps^{\cH}}_{\Sym^\ast_{\PP}\Nas_{\PP}(A)} \, \circ \Sym^\ast_{\PP}(g^{\prime \prime}_{A_{\PP}}).
\]
\end{enumerate}
We next verify the three assertions in the statement of the theorem.
\begin{enumerate}[leftmargin=15pt]

\item We have an equality between compositions of right adjoints 
$\Kb_{\PP} U^\ast_{\Com^{\PP}} = U^\bl_{\Com^{\PP}}\Kb_{\Com^{\PP}}$. 
Hence, the compositions of left adjoints 
$\DPP \cc \Sym^\bl_{\PP}$ and $\Sym^\ast_{\PP} \cc \Nas_{\PP}$ are isomorphic, by uniqueness of adjoints.

\item Let $A^\bl \in \cCom$, and let $A^\bl_{+}:= U^\bl_{+ \Com}(A^\bl)$. 
Using the identities \eqref{eq:compat}, along with the explicit description of the fork in Eq.\ \ref{eq:Dcoeq}, we obtain an equality of diagrams
\[
\begin{split}
U^\ast_{+,\Com}\Bigl( 
 \xymatrix{
 \Sym^\ast \bigl (N^\ast\cK(A) \bigr)  \ar^*++{f_{A}}@<-.7ex>[r]
 \ar@<+.7ex>[r]_*++{g_{A}} & \Sym^\ast\bigl(N^\ast(A) \bigr)
} \Bigr)
\\
=
\xymatrix{
 \Sym^\ast_{+} \bigl (N^\ast_{+}\cK_{+}(A_{+}) \bigr)  \ar^*++{f_{A_{+}}}@<-.7ex>[r]
 \ar@<+.7ex>[r]_*++{g_{A_{+}}} & \Sym^\ast_{+}\bigl(N^\ast_{+}(A_{+}) \bigr)
} 
\end{split}
\] 
Since $U^\ast_{+,\Com} \maps \cdga \to \cga$ preserves colimits, we conclude that 
\[
U^*_{+, \Com} \Da(A^\bl)  \cong \Dap U^\bl_{+,\Com}(A^\bl).
\]

\item The proof of \cite[Prop.\ 2.13]{SS} for simplicial algebras easily adapts to the context of (almost) cosimplicial algebras. In particular, the counit for the (almost) cosimplicial Dold-Kan correspondence $\eps^{\op \PP}_{\DK} \maps N^\ast_{\PP} \Kb_{\PP} \xto{\cong} \id$ lifts to a monoidal natural transformation. Therefore, in analogy with \cite[Prop.\ 2.13]{SS} and Lemma \ref{lem:coalg-adjoint}, we deduce that $\Kb_{\Com^{\PP}}$ is full and faithful. Hence, the counits 
$\eps_{\cat{Com}^{\PP}} \maps \DPP \Kb_{\Com^{\PP}} \to \id$ are isomorphisms.    
\end{enumerate}
\end{proof}
\newpage
\mbox{}

\begin{figure}[H]
\centering
\[
\begin{tikzpicture}[back line/.style={densely dotted},
  cross line/.style={preaction={draw=white, -,line width=6pt}}
]
  \matrix (m)[
    matrix of math nodes,
    row sep=6em, column sep=7em,
    text height=1.5ex, text depth=0.25ex
  ]{
     \& \cCom    \&   \& \cdga   \\   
    \cpCom    \&   \& \cga    \&       \\   
     \& \cVect   \&   \& \cCh    \\   
    \cpVect    \&   \& \gVect  \&       \\   
  };

  \AdjRLh{m-1-2}{m-1-4}{\tc{D^{*}}}{\tc{K^{\bl}_{\Com}}}
  \AdjRLh[above left][below right]{m-1-2}{m-2-1}{U^\bl_{+,\Com}}{}
  \AdjRLh[above left][below right]{m-1-4}{m-2-3}{U^\ast_{+,\Com}}{} 
  \AdjRLv{m-1-4}{m-3-4}{\tc{U^{\ast}_{\Com}}}{\tc{\Sym^{\ast}}}
  \AdjRLv{m-2-1}{m-4-1}{U^{\bl}_{\Com^{+}}}{\Sym^{\bl}_{+}}
  \AdjRLh{m-4-1}{m-4-3}{N^{\ast}_{+}}{K^\bl_{+}}
  \AdjRLh{m-3-4}{m-4-3}{U^\ast_{+}}{}

  \begin{scope}[back line]
    \AdjRLv[above left, near end][above right,near end]{m-1-2}{m-3-2}{\tc{U^{\bl}_{\Com}}}{\tc{\Sym^{\bl}}}  
    \AdjRLh[above,near start][below,near start]{m-3-2}{m-3-4}{\tc{N^{\ast}}}{\tc{K^{\bl}}}  
    \AdjRLh{m-3-2}{m-4-1}{U^\bl_{+}}{} 
  \end{scope}

  \begin{scope}[cross line]
    \AdjRLh[above,near end][below,near end]{m-2-1}{m-2-3}{D^*_{+}}{K^{\bl}_{\Com^{+}}}
    \AdjRLv[above left,near start][above right,near start]{m-2-3}{m-4-3}{U^\ast_{\Com^{+}}}{\Sym^{*}_{+}}
  \end{scope}
\end{tikzpicture}
\]
\caption{A cube of adjoints with the following identities:}
\[
\begin{split}
N^\ast_{+}U^{\bl}_{+} = U^{\ast}_{+}N^\ast,& \quad \Kb_{+} U^\ast_{+} = U^\bl_{+}\Kb\\
U^\ast_{+,\Com} \, \Sym^\ast = \Sym^\ast_{+} U^\ast_{+}, &\quad U^\bl_{+,\Com}\, \Sym^\bl = \Sym^{\bl}_{+} U^{\bl}_+ \\
\Kb_{\PP} U^\ast_{\Com^{\PP}} = U^\bl_{\Com^{\PP}}\Kb_{\Com^{\PP}}, &\quad  \Kb_{\Com^+}U^\ast_{+,\Com} = U^\bl_{+,\Com}K^\bl_{\Com}\\
U^\bl_{\Com^+}\, U^\bl_{+,\Com} = U^{\bl}_{+}\, U^{\bl}_{\Com}, & \quad U^\ast_{\Compp}\, U^{\ast}_{+,\Com} = U^\ast_{+} \, U^\ast_{\Com}\\
U^\bl_{+} \cc \eta^{\vee}_{\DK} = \eta^{\vee +}_{\DK}\cc  U^\bl_{+}, &\quad U^\ast_{+} \cc \eps^{\vee}_{\DK} = \eps^{\vee +}_{\DK} \cc U^\ast_{+}
\end{split}
\]
\label{fig:cube}
\end{figure}
\mbox{}

\newpage
\subsubsection{Reduced finitely-generated algebras, completions, and cotangent complexes} \label{sec:reduced} 
Denote by $\crComfg \sse \cCom$ the subcategory of (finitely-generated) {\bf reduced} unital augmented cosimplicial commutative algebras. Such an algebra $A^\bl$ satisfies $A^0=\kk$, and hence has a canonical augmentation $s^0 \maps A^\bl \to \kk$.
Define $\crpComfg \sse \cpCom$ to be the analogous category of (finitely-generated) reduced almost cosimplicial algebras. 

In what follows, if $A^\bl$ is in $\crComfg$ or $\crpComfg$, we denote by $\mm_A \ideal A^\bl$ the augmentation ideal and $\hA^\bl$ the corresponding $\mm$-adic completion. Since the completion with respect to a maximal ideal is always local, we obtain a functor
\begin{equation} \label{eq:comp-func}
(-)^{\wh{}} \maps \crCom \to c\rclnAlg,
\end{equation}
where $c\rclnAlg$ denotes the full subcategory of reduced objects in $c\clnAlg$. There are the analogous definitions for the subcategories $\rcdgafg$ and $\rcgafg$ of (finitely-generated) {\bf reduced} cdgas and cgas, respectively. In particular, a reduced c(d)ga $C^*$ satisfies $C^0 =\kk$ and is finitely-generated if $C^{i}$ is finite-dimensional over $\kk$ for all $i >1$.   

We have a functor
$\cot^\bl \maps \cCom_0 \to \cVect$ which sends $A^\bl$ to the Zariski cotangent space  
$\cot^\bl(A):= \mm^\bl_A/{\mm_A^{\bl}}^2$. The \df{cotangent complex} is the normalized cochain complex 
$\Nas(\cot^\bl(A))$ of the cotangent space. 

Similarly, if $C^\ast \in \cdga_0$ with augmentation ideal $\ov{C} \ideal C$, then its \df{cotangent complex} is the space of indecomposables $\ov{C}/\ov{C}^2$, which is naturally a cochain complex, and this gives a functor
\[
\ccot{\cdga} \maps \cdga_0 \to \cCh, \qquad  \ccot{\cdga}(C^\ast):=\ov{C}/\ov{C}^2.  
\]
These functors are part of the usual adjunctions e.g.\ \cite[\Sec 2]{Quillen:Com}
\[
\cot^\bl \maps \cCom_0 \adjunct \cVect \maps (\kk \semiop -)^\bl, \qquad  
\ccot{\cdga} \maps \cdga_0 \adjunct \cCh \maps (\kk \semiop -)^\ast.
\]
Above, the right adjoints send a cosimplicial vector space $V^\bl$ (resp.\ chain complex $C^\ast$) to the 
reduced cosimplicial (resp.\ dg) commutative algebra corresponding to the 
trivial extension of $\kk$ by $V^\bl$ (resp.\ $C^\ast$).

The next proposition is the dual of Prop.\ \ref{prop:K-prims} concerning primitives of simplicial coalgebras. The proof is exactly the same.
\begin{proposition}\label{prop:cot}
Let $A^\bl \in \cCom_0$. There is a natural isomorphism of cochain complexes
\[
\ccot{\cdga}(D(A^{\bl})) \cong \Nas(\cot^\bl(A)).
\]
\end{proposition}

The main result of this subsection is:
\begin{proposition} \label{prop:reduced}
\mbox{}
\begin{enumerate}
\item The adjunction $\Da \maps \cCom \adjunct \cdga \maps \Kb_{\Com}$
restricts to an adjunction between finitely generated reduced algebras
\begin{equation} \label{eq:prop-reduced}
\Da \maps \crCom \adjunct \rcdga \maps \Kb_{\Com}
\end{equation}

\item If $A^\bl \in \cCom_0 $, then there is a natural isomorphism of reduced cdga
\begin{equation} \label{eq:prop-reduced2}
\Da(\wh{A}^\bl) \cong \Da(A^\bl)
\end{equation}
and the adjunction \eqref{eq:prop-reduced} extends along \eqref{eq:comp-func}:
\begin{equation} \label{eq:prop-reduced3}
\Da \maps c\rclnAlg \adjunct \rcdga \maps \Kb_{\Com}
\end{equation}
\end{enumerate}
\end{proposition}
\begin{proof}
Let $A^\bl \in \crCom$. Since $A^{\bl}$ is unital and augmented, we have $A^\bl =\kk \dsum \mm^{\bl}_A$ as simplicial vector spaces, with $\mm^{0}_A = 0$. The inclusion $\mm^\bl_A \sse A^{\bl}$ induces a epimorphism $\Sym^{\bl}(\mm_A) \to A^{\bl}$ in $\cCom$.
Since $\Da \maps \cCom \to \cdga$ is left adjoint, we obtain, via statement (1) of Thm.\ \ref{thm:D-exists}, a degree-wise surjection 
between cdgas $\Sym^{\ast}(N(\mm^\bl_A)) \to \Da(A^\bl).$
Since $N^0(\mm^\bl_A) = \mm^0_A=0$, we conclude that $\Da(A^\bl)$ is reduced.
Next, we verify that the restriction $K^{\bl}_{\Com} \maps \cdga_0 \to \cCom_0$ is well-defined, and that the isomorphism 
\eqref{eq:prop-reduced2} in the second statement exists. If $C \in \cdga_0$, then it is easy to see from the identification
introduced in \eqref{eq:K-lam}:
\begin{equation} \label{eq:reduced-pf1}
K^\bl_{\Com}(C) \cong Z^0 ( C^\ast \tensor \Lam^\bl),
\end{equation}
that $K^\bl_{\Com}(C)$ is reduced. Therefore, the adjunction restricts to reduced algebras $\Da \maps \cCom_0 \adjunct \cdga_0 \maps \Kb_{\Com}$. Moreover, from \eqref{eq:reduced-pf1} and the definition (Def.\ \ref{eq:lam-def}) of $\Lam^\bl$, it is clear that for each $C\in \cdga_0$, the augmentation ideal of $K^\bl_{\Com}(C)$ is dimension-wise nilpotent. Hence, ${\wh{K}^\bl}_{\Com}(C) = K^\bl_{\Com}(C)$, and therefore, by the universal property of completion
\[
\hom_{\cCom_0}({A}^{\bl}, K^\bl_{\Com}(C)) = \hom_{\cCom_0}(\hA^{\bl}, K^\bl_{\Com}(C)).
\]
The isomorphism $\Da(\wh{A}^\bl) \cong \Da(A^\bl)$ then follows.

Finally, we show that the restriction \eqref{prop:reduced} of the adjunction to finitely-generated reduced algebras is well-defined. If $C \in \rcdga$, then $K^n_{\Com}(C) \sse \bigoplus_{i=0}^{n+1} C^{i} \tensor \Lam^{n}_{-i}$ implies that $K^n_{\Com}(C)$ is finite-dimensional. Hence, $K^\bl_{\Com}(C) \in \crCom$. On the other hand, let $A^\bl \in \crCom$. Then $\hA^\bl$ is local with $\mm^\bl_{\hA}$ finitely-generated, and therefore, by Prop.\ \ref{prop:cot}
\[
\ov{\Da}(\hA^\bl)/\ov{\Da}(\hA^\bl)^2 \cong \ccot{\cdga}(D^\ast(\hA^{\bl})) \cong \Nas(\mm^\bl_{\hA}/(\mm^\bl_{\hA})^2)
\]
is degree-wise finite-dimensional over $\kk$. Combining this along with the fact that the graded vector space $\ov{\Da}(\hA^\bl)$ is concentrated in positive degrees, a degree-wise induction argument shows $\ov{\Da}(\hA^\bl)$ is degree-wise finite-dimensional over $\kk$. We conclude that $\Da(\hA^{\bl})$ is finitely-generated, and hence, by \eqref{eq:prop-reduced2}, $\Da(A^{\bl})$ is as well.      
\end{proof}

\begin{remark}\label{rmk:reduced}
Both statements in Prop.\ \ref{prop:reduced} also hold {\it mutatis mutandis}
in the almost cosimplicial/graded context. We obtain an adjunction $\Dap \maps \cprCom \adjunct \rcga \maps \Kb_{\Comp}$, and the extension $\Dap \maps c^{+}\rclnAlg \to \rcga$ with $\Dap(\wh{A}^\bl) \cong \Dap(A^\bl)$. 
\end{remark}

\subsection{The formal differentiation functor} \label{sec:FDiff}
In this section, we construct a differentiation functor for formal $\infty$-groups.
In analogy with Sec.\ \ref{sec:reduced}, let $s\rconil$ denote the full subcategory of $s\conil$ consisting of those coalgebras $C_\bl$ which are reduced as simplicial objects, i.e.\ $C_0=\kk$. Similarly, let $\rdgcocom$ denote the full subcategory of $\dgcocom$ consisting of those coalgebras $C_\ast$ satisfying $C_0=\kk$, and such that $\prim_\ast(C)$ is degree-wise finite-dimensional. In either case, we say such coalgebras are \df{reduced}. We have the analogous categories of reduced almost simplicial coalgebras $s^{+}\rconil$ and reduced graded coalgebras $\rgcocom$.

In analogy with the simplicial case \eqref{eq:simp-dual}, the equivalence of categories $\conil \simeq \bigl(\clnAlg\bigr)^\op$ recorded in Prop.\ \ref{eq:dual} also extends to the reduced differential graded context
\[
\kk[-]^\ast \maps \rdgcocom \overset{\simeq}{\longleftrightarrow} \bigl(\rcdga)^\op \maps \dgspf(-).
\]
Here, as in \cite[Rmk.\ A.12]{LH:2009} and \cite[Thm.\ 2.4]{Lazarev}, $\kk[C]^\ast := \Hom_{\gVect}(C,\kk)$ is the usual graded dual, and $\dgspf(A):=\Hom^{\cont}_{\gVect}(A,\kk)$ denotes continuous graded vector space homomorphisms. As usual, $A \in \rcdga$ is equipped with its adic topology induced by the canonical augmentation ideal $\ov{A}$, and $\kk$ is discrete concentrated in degree zero. Note that since $A$ is reduced, $\ov{A}$ is concentrated in positive degrees. Hence, $A$ is complete with respect to the adic topology. We also have the analogous equivalence between reduced graded coalgebras and reduced graded algebras:
\[
\kk[-]_{+}^\ast \maps \rgcocom \overset{\simeq}{\longleftrightarrow} \bigl(\cga_0)^\op \maps \gspf(-).
\]
In particular, if $V \in \gFDVect$ and $V = V_{\ast \geq 1}$, then there is a natural isomorphism of coalgebras
\[
\cogrSym(V) \xto{\cong} \gspf \bigl(\Sym_{+}(V^\vee) \bigr).
\]
\subsubsection{Comparisons between $s\rconil$ and $\rdgcocom$}
Leading up to our analysis of the formal differentiation functor, we first record a few lemmas. 
\begin{lemma}\label{lem:G1}
Let $C \in s\rconil$. There is a natural isomorphism of chain complexes
\[
N_\ast \prim_\bl(C) \xto{\cong} \prim_\ast\bigl( \spf( D (\kk[C]^{\bl})) \bigr), 
\]
where $\kk[-]^{\bl} \maps s\rconil \xto{\simeq} \bigl(c\rclnAlg \bigr)^\op$ is the equivalence \eqref{eq:simp-dual}, and
$\Da \maps c\rclnAlg \to \rcdga$ is the left adjoint functor \eqref{eq:prop-reduced3}. 
\end{lemma}
\begin{proof}
Let $A^\bl = \kk[C]^\bl$, and $B^\ast = D^\ast A^\bl$. By Cor.\ \ref{cor:prim-cot}, we have a natural isomorphism of finite-dimensional cosimplicial vector spaces
\begin{equation} \label{eq:G1-1}
(\tha^\bl)^{\vee} \maps \mm^\bl_A/(\mm^\bl_A)^2  \xto{\cong}  \prim_\bl(C)^\vee,
\end{equation}
where $\mm^\bl_A \ideal A^\bl$ is the augmentation ideal. As shown in \cite[\Sec 5.0.13]{F2}, taking the $\kk$-linear simplicial/graded dual commutes with taking normalized chains/cochains. That is, for any $V_\bl \in \sVect$, there is a natural isomorphism of cochain complexes
\[
\Nas(V^{\vee}_\bl) \cong N_\ast(V_\bl)^{\vee}.
\]
Combining this with \eqref{eq:G1-1}, we obtain natural isomorphisms in $\cCh$
\begin{equation} \label{eq:G1-2}
\Nas \bigl(\mm^\bl_A/(\mm^\bl_A)^2 \bigr) \xto{\cong} \Nas \bigl(\prim_\bl(C)^\vee \bigr) \xto{\cong} N_\ast \bigl( \prim_\bl(C) \bigr)^{\vee}.
\end{equation}
On the other hand, Prop.\ \ref{prop:cot} provides another natural isomorphism in $\cCh$ between degree-wise finite-dimensional complexes
\[
\ov{B^\ast}/\ov{B^\ast}^2 \xto{\cong} \Nas\bigl ( \mm^\bl_A/(\mm^\bl_A)^2 \bigr),
\]
where $\ov{B^\ast}$ is the augmentation ideal of $B^\ast$. Furthermore, the dg analogue of Cor.\ \ref{cor:prim-cot} is a natural isomorphism of chain complexes
\[
\tha^\ast \maps \bigl( \ov{B^\ast}/\ov{B^\ast}^2 \bigr)^{\vee}  \xto{\cong}  \prim_\ast(\spf(B)).
\]
Combining this with \eqref{eq:G1-2} and dualizing gives the desired isomorphism.
\end{proof}

\newcommand{\KC}{\mathrm{K}C}  
\begin{lemma}\label{lem:G2}
Let $C_\ast \in \rdgcocom$, and $A^{\bl}_{\KC}:=\kk[\Kcoalg(C_\ast)]^\bl \in c\rclnAlg$. There is an isomorphism of dg coalgebras
\[
C_\ast \xto{\cong} \dgspf(D(A^\bl_{\KC})) 
\]
which is natural in $C_\ast$.
\end{lemma}

\begin{proof}
By \cite[\Sec 5.0.13]{F2}, there is a natural isomorphism of $\kk$-linear duals $K_\bl(C)^{\vee} \cong K^\bl(C^{\vee})$.
Comparison of the simplicial shuffle map \eqref{eq:EM} with its dual \eqref{eq:EM-dual} shows that this isomorphism is compatible with the multiplicative structures on either side. 
Therefore, we have
\[
D^\ast(A^{\bl}_{\KC}) = D^{\ast}\bigl(\Kcoalg_\bl(C)^{\vee} \bigr) \cong D^{\ast}\bigl(K^\bl_{\Com}(C^{\vee}) \bigr).
\]
By statement (3) of Thm.\ \ref{thm:D-exists}, there is a natural isomorphism $D^{\ast}\bigl(K^\bl_{\Com}(C^{\vee}) \bigr) \cong  C^{\vee}$ of reduced cdga. Hence, $C_\ast \cong \dgspf(C^{\vee}) \cong \dgspf \bigl( D\bigl(K^\bl_{\Com}(C^{\vee}) \bigr) \bigr)$.
\end{proof}

\subsubsection{Lie $\infty$-algebras} \label{sec:Linf}
Our conventions for Lie $\infty$-algebras follow \cite{R}, with some minor changes. 
In particular, denote by
\[
\linf \sse \rdgcocom 
\]
the full subcategory of reduced finite-type dg coalgebras $(C_\ast,\del)$ such that there exists an isomorphism of graded coalgebras
\[
\cogrSym(L[1]) \cong C_\ast
\]
for some $L \in \gFDVect$. We call objects of $\linf$ {\bf finite-type Lie $\infty$-algebras}. As discussed in Sec.\ \ref{sec:intro-sec3} of the introduction, there is a canonical equivalence $\linf \simeq \LnA{\infty}^{\ft}$
where $\LnA{\infty}^{\ft}$ is the category of finite-type Lie $\infty$-algebras considered in \cite{R}. Hence, every object $(C_\ast,\del)$ in $\linf$ is non-canonically isomorphic to the Chevalley-Eilenberg coalgebra of a Lie $\infty$-algebra of the form $(L[-1],\el_1,\el_2,\el_3, \cdots)$ where
\[
(L_\ast,\bs \el_1 \bs^{-1}) \cong \prim_\ast(C) \quad \text{in $\Chain$}.
\] 
Formalizing the above, define the \df{tangent complex} or ``underlying complex'' of a Lie $\infty$-algebra via the functor
\[
\tan_\ast \maps \linf \to \Chain, \qquad \tan_\ast(C,\del):= \prim_{\ast}(C)[-1].
\]
In what follows, we tacitly identify Lie $\infty$-algebras as objects of $\linf$ with their Chevalley-Eilenberg coalgebras.

\begin{theorem}[\cite{R}]\label{thm:Linf-CFO}
The category $\linf$ inherits a category of fibrant objects structure
in which a morphism $F \maps (C,\del) \to (C',\del')$ is:
\begin{itemize}
\item a weak equivalence if and only if $\tan_\ast(F)$ is a quasi-isomorphism.
\item a fibration if and only if $\tan_\ast(F)$ is a surjection in positive degrees.
\end{itemize}
\end{theorem}

\subsubsection{The functor $\FDiff(-)$ and its properties}

\begin{theorem}\label{thm:FDiff}
The composition
\[
s\rconil \xto{\kk[-]^{\bl}} \bigl(c\rclnAlg \bigr)^{\op} \xto{{D^\ast}^{\op}} \bigl(\rcdga \bigr)^{\op} \xto{\dgspf} \rdgcocom 
\]
restricts to a functor
\[
\FDiff(-) \maps \FG \to \linf
\]
with the following properties:
\begin{enumerate}

\item The tangent complex of the finite-type Lie $\infty$-algebra $\FDiff(\fG_\bl)$ is isomorphic to the shifted normalized complex of primitives of $\fG_\bl$, i.e.,
\[
\tan_\ast \bigl( \FDiff(\fG_\bl) \bigr) \cong N_\ast\bigl( \prim_\bl(\fG)\bigr)[-1]. 
\]

\item If $\gb$ is a simplicial Lie algebra that is level-wise finite-dimensional, then there is a \und{canonical} natural isomorphism of $L_\infty$-algebras
\[
\Phi_{\gb} \maps \Ng \xto{\cong} \FDiff( \Wbar \cU(\g)).
\]

\item $\FDiff(-)$ preserves limits, preserves and reflects weak equivalences, and preserves and reflects fibrations.
\end{enumerate}
\end{theorem}
\begin{proof}
\mbox{}

\begin{enumerate}[leftmargin=15pt]

\item By Thm.\ref{thm:pbw}, $\fG_\bl$ admits a PBW basis, i.e.\ there exists an isomorphism of almost simplicial coalgebras $U^{+}_\bl \fG \xto{\cong} \cospSym(U^{+}\prim(\fG))$. This yields an isomorphism of complete reduced almost cosimplicial algebras
\[
U^{\bl}_{+}\kk[\fG] \xto{\cong} U^{\bl}_{+}\kk[\Sym^{\coalg}(\prim(\fG))] = \cphSym(\prim(\fG)^{\vee}). 
\]  
We apply the functor $\Dap \maps c^{+}\rclnAlg \to \rcga$.  Statement (2) of Thm.\ \ref{thm:D-exists}, along with statement (2) of Prop.\ \ref{prop:reduced} and Rmk.\ \ref{rmk:reduced}, yield the isomorphisms
\[
U^\ast_{+} D(\kk[\fG]^\bl) \cong \Dap(U^\bl_{+}\kk[\fG]) \xto{\cong} \Dap  
\bigl(\cphSym(\prim(\fG)^{\vee}) \bigr) \cong \Dap\bigl(\cpSym(\prim(\fG)^{\vee})\bigr).
\] 
Combining this with statement (1) of Thm.\ \ref{thm:D-exists}, we obtain a not necessarily canonical isomorphism of reduced finitely generated graded algebras 
\begin{equation} \label{eq:FDiffpf1}
U^\ast_{+} D(\kk[\fG]^\bl) \xto{\cong} U_{+}^\ast \dgSym(\Nas(\prim_\bl(\fG)^{\vee})).
\end{equation}
As in the proof of Lem.\ \ref{lem:G1}, it follows from \cite[\Sec 5.0.13]{F2} that we have an isomorphism of cochain complexes $\Nas(\prim_\bl(\fG)^{\vee}) \cong N_\ast(\prim_{\bl}(\fG))^{\vee}$. 
Therefore, applying the functor $\gspf(-)$ to \eqref{eq:FDiffpf1} gives an isomorphism of graded coalgebras
\[
\Psi \maps U^{+}_\ast \FDiff(\fG_\bl) \xto{\cong} \cogrSym\bigl(U^{+} N(\prim_\bl(\fG))\bigr)
\]
Let $\del$ denote the codifferential on $\FDiff(\fG_\bl) \in \rdgcocom$. Then the composition $\del^{\Psi}:= \Psi \cc \del \cc \Psi^{-1}$ defines a codifferential on the graded coalgebra $\cogrSym\bigl(U^{+}_\ast N_\ast(\prim(\fG))\bigr)$ which allows us to lift $\Psi$ to an isomorphism in $\rdgcocom$:
\[
\Psi \maps \FDiff(\fG_\bl) \xto{\cong} \Big(\codgSym\bigl(U^{+} N(\prim_\bl(\fG))\bigr), \del^{\Psi} \Bigr)
\]
We conclude $\FDiff(\fG_\bl) \in \linf$, and that the underlying graded vector spaces of the complexes $\tan_\ast \FDiff(\fG_\bl)$ and $N_\ast(\prim_\bl(\fG)[-1]$ are isomorphic. By Lemma \ref{lem:G1}, it follows that $\tan_\ast \FDiff(\fG_\bl) \cong N_\ast(\prim_\bl(\fG)[-1]$ as chain complexes.

\item  Let $\gb \in \sLie^{\fd}$ and denote by $\psi \maps \Kcoalg_\bl(\CE(N \g)) \to \Wbar{\cU(\g)}$ the morphism of simplicial coalgebras introduced in Thm.\ \ref{thm:themap}. Let $\Psi:= \dgspf D(\kk[\psi])$. By Lemma \ref{lem:G1}, we have a commutative diagram in $\Chain$.
\[
\begin{tikzdiag}{2}{2}
{
N_\ast\prim_\bl \Kcoalg(\CE(N \g)) \&  \& N_\ast \prim_{\bl}\Wb{\cU(\g)}\\
\prim_\ast\spf(D(A)) \&  \& \prim_\ast \FDiff(\Wbar{\cU(\g)}) \\
};

\path[->,font=\scriptsize]
(m-1-1) edge node[auto] {$N_\ast(\prim(\psi)) $} (m-1-3)
(m-1-1) edge node[auto,sloped] {$\cong$} (m-2-1)
(m-1-3) edge node[auto,sloped] {$\cong$} (m-2-3)
(m-2-1) edge node[auto] {$\prim_\ast(\Psi)$} (m-2-3)
;
\end{tikzdiag}
\] 
where $A^\bl=\kk[\Kcoalg(\CE(N \g))]^{\bl}$. Corollary \ref{cor:prim-map} implies that $N_\ast(\prim(\psi))$ is an isomorphism. Hence, $\prim_\ast(\Psi)$ is as well.
By Lemma \ref{lem:G2}, we have a natural isomorphism $F \maps \CE_{\ast}(N \g) \xto{\cong} \dgspf(\Da(A))$ in $\rdgcocom$. Hence, composition with $\Psi$ defines a morphism of Lie $\infty$-algebras
\[
\Phi_{\gb} \maps \Ng \to \FDiff(\Wbar{\cU(\g)})
\]
whose induced map on tangent complexes 
\[
\tan_\ast(\Phi_{\gb}) = \bigl(\prim_\ast(\Psi) \cc \prim_\ast(F)\bigr)[-1]
\]
is an isomorphism. Since the functor $\tan_\ast \maps \linf \to \Chain$
reflects isomorphisms (e.g., \cite[\Sec 3.1]{R}), we conclude that $\Phi_{\gb}$ is an isomorphism.   

\item The functors $\kk[-]^{\bl}$ and $\dgspf(-)$ are equivalences while $\Da(-)^{\op}$ is a right adjoint by statement (1) of Prop.\ \ref{prop:reduced}. Hence, $\FDiff(-)$ preserves limits. Given $\phi \maps \fG \to \fG'$ in $\FG$,  Lemma \ref{lem:G1} yields the commutative diagram
\[
\begin{tikzdiag}{2}{2}
{
N_\ast(\prim_\bl \fG) \& \prim_\ast(\FDiff(\fG))  \& \\
N_\ast(\prim_\bl \fG') \& \prim_\ast(\FDiff(\fG')) \& \\
};

\path[->,font=\scriptsize]
(m-1-1) edge node[auto] {$\cong$} (m-1-2)
(m-2-1) edge node[auto] {$\cong$} (m-2-2)
(m-1-1) edge node[auto,swap] {$N_\ast(\prim_\bl(\phi))$} (m-2-1)
(m-1-2) edge node[auto] {$\prim_\ast\FDiff(\phi)$} (m-2-2)
;
\end{tikzdiag}
\]
Hence, by Rmk.\ \ref{rmk:formal-tan} and statement (2) of Thm.\ \ref{thm:formal-hmtpy}, $\phi$ is a weak equivalence if and only if $N_\ast(\prim_\bl(\phi))$ is a quasi-isomorphism, if and only $\FDiff(\phi)$ is a weak equivalence of Lie $\infty$-algebras by Thm.\ \ref{thm:Linf-CFO}. The same argument implies the analogous statement for fibrations.  
\end{enumerate}
\end{proof}

\begin{remark}\label{rmk:conj}
We conjecture that the functor $\FDiff \maps \FG \to \linf$ induces an equivalence
between the underlying $\infty$-categories
\[
\mathbf{FmlGrp_{\infty}}[W^{-1}] \xto{\simeq} \mathbf{Lie_{\infty}Alg^{fin}}[W^{-1}],
\]
and, in particular, an equivalence between homotopy categories 
$\Ho (\FG) \simeq \Ho (\linf)$ of the iCFO and CFO structures on  $\FGpd$ and $\linf$, respectively. 
\end{remark}

\section{Differentiation of Lie $\infty$-groups} \label{sec:LieDiff}
Given a Lie $\infty$-group $\cG_\bl$, let $T_{n}\cG$ denote the tangent space of the manifold $\cG_n$ at the canonical basepoint. This induces the \df{tangent functor} 
\[
T_{\bl}(-) \maps \LinfGrp \to s\FDVect
\]
It follows from the definition of the point distribution functor \eqref{eq:dist} that we have a natural isomorphism
\begin{equation} \label{eq:Tan-prim}
T_{\bl}(\cG) \cong \prim_\bl \bigl(\Dist(\cG)\bigr).
\end{equation}
The \df{differentiation functor} $\Diff(-) \maps \LinfGrp \to \linf$
is the composition
\[
\LinfGrp \xto{\Dist_\bl} \FG \xto{\FDiff} \linf.
\]
We conclude with the main result of the paper.

\begin{theorem}\label{thm:Diff}
The differentiation functor $\Diff(-) \maps \LinfGrp \to \linf$ has the following properties:
\begin{enumerate}
\item The underlying complex of the Lie $\infty$-algebra $\Diff(\cG_\bl)$ is isomorphic to the shifted tangent complex of $\cG_{\bl}$ i.e.,
\[
\tan_\ast \Diff(\cG_\bl) \cong N_{\ast}(T_{\bl}\cG)[-1].
\]
In particular, if $\cG_\bl$ is a Lie $n$-group, then $\Diff(\cG_\bl)$ is a Lie $n$-algebra.

\item Let $G_\bl$ be a simplicial Lie group with Lie algebra $\gb$. There is a \und{canonical} natural isomorphism of $L_\infty$-algebras
\[
\Phi_{G_\bl} \maps N_\ast \gb \xto{\cong} \Diff(\Wbar{G}).
\]

\item $\Diff(-)$ preserves weak equivalences, fibrations, and pullbacks along fibrations. 

\end{enumerate}
\end{theorem}
\begin{proof}
\mbox{}
\begin{enumerate}[leftmargin=15pt]
\item The first statement follows from statement (1) of Thm.\ \ref{thm:FDiff} and \eqref{eq:Tan-prim}. Suppose $\cG_\bl$ is a Lie $n$-group. By Cor.\ \ref{cor:geo-func-ex}, the functors $\Dist \maps \Mfd_\ast \to \conil$, 
and $\prim \maps \conilsm \to \FDVect$ are geometric. Hence, by Cor.\ \ref{cor:n-grpd} and \eqref{eq:Tan-prim}, $T_{\bl}(\cG)$ is an $n$-groupoid in $\sFDVect$. The uniqueness of the horn filling condition in dimension $\geq n$ implies that $N_\ast T_{\bl}(\cG)[-1]$ is concentrated in degrees $\leq n-1$. 
Hence, $\Dist(\cG)$ is a Lie $n$-algebra.

\item Let $G_\bl$ be a simplicial Lie group with Lie algebra $\gb$. By Cor.\ \ref{cor:WG-Ug}, there is an isomorphism of formal $\infty$-groups $\Dist_\bl(\Wb{G}) \cong \Wbar{\cU(\g)}$. The result then follows from statement (2) of Thm.\ \ref{thm:FDiff}.

\item Statement (4) of Thm.\ \ref{thm:formal-hmtpy} implies that 
the functor $\Dist_\bl \maps \LinfGrp \to \FG$ preserves weak equivalences, fibrations, and pullbacks along fibrations. Combining this with statement (3) of Thm.\ \ref{thm:FDiff} gives the desired result. 
\end{enumerate}
\end{proof}

\newpage
\appendix

\section{Recollections on formal commutative algebra} \label{sec:comalg}
For convenience, we recall some basic facts about complete local $\kk$-algebras $({R},{\mm})$
with residue field $\kk$.  Unless otherwise specified, by ``complete'' we mean that ${R}$ is isomorphic as a local ring to its $\mm$-adic completion.    
We denote by $\clnAlg$ the category of such algebras and local morphisms between them. In particular, $\clnAlg$ is a not necessarily full subcategory of topological $\kk$-algebras (e.g.\ \cite[\href{https://stacks.math.columbia.edu/tag/07E9}{Tag 07E9}]{stacks-project}).


Given continuous morphisms of complete local Noetherian $\kk$-algebras $C \to {A}$, and $C \to B$, denote by ${A} \wh{\tensor}_{{C}} {B}$ the complete tensor product as 
in \cite[0\textsubscript{I}, 7.7]{EGA}. As explained there, we have an isomorphism of complete topological $\kk$-algebras
\begin{equation} \label{eq:ideal-of-def}
\begin{split}
{A} \wh{\tensor}_{{C}} {B} &= \plim_n (A/\mm^n_A) \tensor_{C/\mm^n_C} (B/\mm^n_B) \\
& \cong \plim_n A \tensor_C B / (\mm_A, \mm_B)^n,
\end{split}
\end{equation}
where $(\mm_A, \mm_B) = \mm_A \tensor_C B + A \tensor_C \mm_B$. In particular, $\wh{(\mm_A, \mm_B)}$ is an ideal of definition for ${A} \wh{\tensor}_{{C}} {B}$. 

The tensor product ${A} \wh{\tensor}_{{C}} {B}$ satisfies the usual universal property for pushouts in complete topological $\kk$-algebras.
\begin{example}\label{ex:powerseries}
There is an isomorphism of complete topological $\kk$-algebras 
\[
\kk[[x_1,\ldots,x_n]] \wh{\tensor} \kk[[y_1,\ldots,y_m]] \cong \kk[[x_1,\ldots,x_n,y_1,\ldots,y_m]].
\]
Let $\mm_X = \gen{x_1,\ldots,x_n}$, and $\mm_Y = \gen{y_1,\ldots,y_m}$.
Under this identification, the completion of the ideal $\bigl( \mm_X,\mm_Y  \bigr)$ is the ideal $\gen{x_1,\ldots,x_n,y_1,\ldots,y_m} \ideal \kk[[x_1,\ldots,x_n,y_1,\ldots,y_m]]$.
\end{example}

\begin{proposition} \label{prop:colimits}
The category $\clnAlg$ admits: (1) coequalizers, (2) finite coproducts,
and hence, pushouts. (3) In particular,  the pushout of a diagram ${A} \leftarrow {C} \to {B}$ in $\clnAlg$ is the tensor product ${A} \wh{\tensor}_{{C}} {B}$.
\end{proposition}
\begin{proof}
\mbox{}
\begin{enumerate}[leftmargin=15pt]
\item Let ${A} \in \clnAlg$ with maximal ideal $\mm$, and $I \ideal A$. Then $A$ complete Noetherian implies that $A/I$ is complete with respect to the $\mm/I$-adic topology, e.g.\  \cite[\href{https://stacks.math.columbia.edu/tag/0325}{Tag 0325}]{stacks-project}.

\item If ${A}, {B} \in \clnAlg$, then via the Cohen structure theorem \cite[\href{https://stacks.math.columbia.edu/tag/032A}{Tag 032A}]{stacks-project}, there are natural numbers $n$ and $m$ along with surjections of local rings $\kk[[x_1,\ldots,x_n]] \to {A}$, 
and $\kk[[y_1,\ldots,y_m]] \to {B}$. Hence, we have a surjection
\begin{equation} \label{eq:coprod-durj}
\kk[[x_1,\ldots,x_n]] \tensor \kk[[y_1,\ldots,y_m]] \to A \tensor B
\end{equation}
which maps the maximal ideal $\bigl( \mm_X,\mm_Y  \bigr)$ onto $(\mm_A, \mm_B)$. Therefore, the $(\mm_A, \mm_B)$-adic filtration equals the $\bigl( \mm_X,\mm_Y  \bigr)$-adic filtration on
$A \tensor B$ as a module over $\kk[[x_1,\ldots,x_n]] \tensor \kk[[y_1,\ldots,y_m]]$. Hence, statement 2 of 
\cite[\href{https://stacks.math.columbia.edu/tag/0315}{Tag 0315}]{stacks-project} implies that the morphism
\eqref{eq:coprod-durj} induces a surjection on the completions
\[
\kk[[x_1,\ldots,x_n,y_1,\ldots,y_m]] =
\kk[[x_1,\ldots,x_n]] \wh{\tensor}_\kk \kk[[y_1,\ldots,y_m]] \to {A}\wh{\tensor}_\kk {B} 
\]  
Since $\kk[[x_1,\ldots,x_n,y_1,\ldots,y_m]] \in \clnAlg$, we conclude $A \wh{\tensor}_\kk {B} \in \clnAlg$ by (1). 

\item As ${A} \wh{\tensor}_{{C}} {B}$ is the pushout in topological $\kk$-algebras, it suffices to verify that the universal morphisms are homomorphisms of local algebras, which is straightforward. 
\end{enumerate}

\end{proof}

\begin{lemma} \label{lem:gen}
Let $\kk$ be a field, $(A,\mm)$ a regular complete local Noetherian $\kk$-algebra of dimension $n$ such that the unit induces an isomorphism $\kk \cong A/\mm$. Choose any $z_1,\ldots,z_n \in \mm$ such that $\{z_1 + \mm^2 , \ldots,z_n+ \mm^2\}$ is a basis for $\mm/\mm^2$. Then the continuous $\kk$-algebra morphism
\[
\kk[[x_,\ldots,x_n]] \to A, \qquad x_i \mapsto z_i
\]
is an isomorphism. 
\end{lemma}
\begin{proof}
See the proof of \cite[\href{https://stacks.math.columbia.edu/tag/0C0S}{Tag 0C0S}]{stacks-project}.
\end{proof}

\section{Shuffles and codegeneracies} \label{sec:shuf}
Our conventions for shuffle permutations follow \cite{EM}. In particular,
we denote by $\Sh(p,q)$ the set of \df{$(p,q)$-shuffles} whose elements are ordered partitions $(I,J)$ of $[p+q-1]$ with $I=\{i_1 < \ldots < i_p\}$ and $J=\{j_1 < \ldots < j_q\}$. We denote by $\sgn(I,J)$ the \df{sign} of the shuffle $(I,J)$.
Recall the standard formula e.g., \cite[p.\ 64]{EM} for the sign of a shuffle in terms of its \df{signature} $\sig(I,J)$: 
\begin{equation} \label{eq:sgn-sh}
\sgn(I,J)=(-1)^{\sig(I,J)}, \qquad \sig(I,J):=\sum_{k=1}^{p} ( i_k - (k-1)). 
\end{equation}
Note that for any $k=1,\ldots, p$, there are exactly $i_k - (k-1)$ elements of $J$ that are less than $i_k$, i.e.
\begin{equation} \label{eq:sig}
i_k - (k-1) = \abs{J \cap [i_{k} -1]}.
\end{equation}
Next, we fix some notation which will be used throughout this section and Sec.\ \ref{sec:map-app}.
\begin{notation}\label{note:ord}
Let $I=\{i_1 < i_2 < \cdots < i_{p}\} \sse \N$ be an ordered subset. For $r \in \N$ define $I^{> r} \sse I$ to be the ordered subset of $I$
\begin{equation} \label{eq:super-ord}
I^{> r}=\{ i \in I \st i >r \}.
\end{equation}
We define $I^{\geq r}$, $I^{< r}$, and $I^{\leq r}$ analogously. Let $t \in \Z$ be an integer such that $i_1 \geq -t$. Denote by $I_t \sse \N$ the ordered set
\begin{equation} \label{eq:sub-ord}
I_{t}:= \{i_{1} + t <    i_{2} + t <  \cdots < i_{p} + t \}.       
\end{equation}
For denoting ordered sets constructed by combining operations 
\eqref{eq:super-ord} and \eqref{eq:sub-ord}, we declare that {\it superscripts precede subscripts}. That is, given $I$, $r$, and $t$ as above, if $k$ is the smallest integer such that $r < i_{k}$, then
\[
I^{> r}_{t}= \{i_{k}+ t <  i_{k+1}+ t < \cdots < i_{p} + t \}.
\]
Finally, we denote by 
\[
\Sh(p,q)_{0 \in J} \sse \Sh(p,q)
\]
the subset consisting of those $(p,q)$-shuffles $(I,J)$ satisfying  $0 \in J$. 
\end{notation}

Next, we recall some basic facts concerning codegeneracy maps in the category of ordinals $\Delta$. We shall use the following notation throughout this section.
\begin{notation}
Given $n \geq 1$ and an ordered subset $B=\{\be_{1} < \be_{2} < \cdots < \be_{q} \} \sse [n-1]$ denote by 
\[
s^{B} \maps [n] \to [n-q]
\]
the epimorphism $s^{B}:= s^{\be_1}s^{\be_2} \cdots s^{\be_q}$ in $\Del$.
\end{notation}

Next, if $f \maps [n] \to [m]$ in $\Delta$ is an epimorphism, then recall that 
\begin{equation} \label{eq:epi-set}
f= s^{\sur(f)}, 
\end{equation} 
where $\sur(f):=\{ t \in [n-1] \st f(t) = f(t+1)\}$. The following lemma is elementary.
\begin{lemma}\label{lem:surj}
Let $B=\{\be_{1} < \be_{2} < \cdots < \be_{q} \} \sse [n-1]$ as above, and $A:=[n-1]\setminus B$. Then 
\[
s^{B}(t)= \abs{\{ \al \in A \st \al < t\}}. 
\]
Furthermore  
\begin{equation} \label{eq:lem-surj}
s^{B}(t+1) = s^{B}(t) +1
\end{equation}
if and only if $t \in A$.
\end{lemma}

The next proposition and its corollary are technical results that will be used in the next section for the proof of Prop.\ \ref{prop:themap}.


\begin{proposition} \label{prop:shuf}
Let $n=p_1 + p_2 + p_3 >0$. There is a bijection of sets
\[
\begin{split}
\vtha \maps \Sh(p_1+p_2,p_3) & \times \Sh(p_1,p_2) \xto{\cong} \Sh(p_1,p_2+p_3+1)_{0 \in J} \times \Sh(p_2,p_3)\\
&\vtha \bigl((A,B),(C,D) \bigr):= \bigl( (I,J), (M,N) \bigr) 
\end{split}
\]
such that the following equalities hold: 
\begin{equation} \label{eq:propshuf1}
\sgn(A,B)\, \sgn(C,D) = (-1)^{p_1}\sgn(I,J)\, \sgn(M,N),
\end{equation}
\begin{equation} \label{eq:propshuf2}
s^{N}s^{I_{-1}} = s^{C}s^{B} \maps [n] \to [p_2],
\end{equation}
\begin{equation} \label{eq:propshuf3}
s^{J^{> 0}_{-1}} = s^{D}s^{B} \maps [n] \to [p_1],
\end{equation}
\begin{equation} \label{eq:propshuf4}
s^{A}= s^{M}s^{I_{-1}} \maps [n] \to [p_3].
\end{equation}
\end{proposition}
\begin{proof}
Let $(A,B) \in \Sh(p_1 + p_2,p_3)$ with $A=\{\al_1 < \cdots < \al_{p_1 +p_2}\}$ and $B=\{\be_1 < \cdots < \be_{p_3}\}.$
Let $(C,D) \in \Sh(p_1,p_2)$ with $C=\{\ga_1 < \cdots < \ga_{p_1} \}$, and $D=\{\del_1  < \cdots < \del_{p_2}\}$. Denote by $(K,L) \in \Sh(p_1,p_2+p_3)$ the unique shuffle such that
\[
K=\{ \al_{\ga_{1}+1} < \cdots < \al_{\ga_{p_1}+1}  \}, \quad L = [n-1]\setminus K.
\]
Write $L= \{ \lam_1 < \cdots < \lam_{p_2 + p_3}\}$. Then there exists a unique shuffle $(M,N) \in \Sh(p_2,p_3)$, with $M=\{ \mu_1 < \cdots < \mu_{p_2}\}$, and $N=\{\nu_1 < \cdots < \nu_{p_3}\}$ such that 
\begin{equation} \label{eq:MN-shuf}
\begin{split}
\al_{\del_{\el} +1 } &= \lam_{\mu_{\el} +1}, \quad \el=1,\ldots,p_2 \\
\be_{\el} &= \lam_{\nu_{\el} +1}, \quad \el=1,\ldots,p_3 \\
\end{split}
\end{equation}
Now define $(I,J) \in \Sh(p_1, p_2+p_3 +1)_{0\in J}$ by
\begin{equation} \label{eq:IJ-shuf}
I:=K_{+1}, \quad J:= L_{+1} \cup \{0\},
\end{equation}
and set $\vtha \bigl((A,B),(C,D) \bigr):= \bigl( (I,J), (M,N) \bigr)$.
By construction, $\vtha$ is a bijection. 

\begin{itemize}[leftmargin=15pt]
\item\und{Verification of Eq.\ \ref{eq:propshuf1}}:  We use the formula \eqref{eq:sgn-sh} to compute the signs of the shuffles:
\begin{equation} \label{eq:propshuf-pf1}
\begin{split}
\sig(A,B) + \sig(C,D) &= \Bigl(\sum_{i=1}^{p_1} (\al_{\ga_{i} + 1} - \ga_{i}) +  \sum_{j=1}^{p_2} (\al_{\del_{j} +1} -\del_{j}) \Bigr) + \sum_{i=1}^{p_1}( \ga_{i} - (i -1)) \\
&= \sum_{i=1}^{p_1} (\al_{\ga_{i} + 1} - (i-1)) + \sum_{j=1}^{p_2} (\al_{\del_{j} +1} -\del_{j}).
\end{split}
\end{equation}
Now consider $\mu_\el \in M$. Then
\[
\mu_\el = \el -1 + \abs{N \cap [\mu_{\el} -1]}.
\]
From Eq.\ \ref{eq:MN-shuf}, we observe that an element $\nu_k \in N$ satisfies
the inequality $\nu_k < \mu_{\el}$ if and only if $\be_{k} < \al_{\del_{\el} +1}$. Hence,   
\[
\begin{split}
\mu_\el &= \el -1 + \abs{ B \cap [ \al_{\del_{\el}+1} -1]}\\
&= \el -1 + \al_{\del_{\el}+1} -  \del_{\el},
\end{split}
\]
where the last equality above follows from Eq.\ \ref{eq:sig}. Therefore, we have
\[
\begin{split}
\sig(K,L) + \sig(M,N) &= \sum_{i=1}^{p_1} (\al_{\ga_{i} + 1} - (i-1)) + \sum_{\el=1}^{p_2} (\mu_{\el} - (\el-1))\\
&=  \sum_{i=1}^{p_1} (\al_{\ga_{i} + 1} - (i-1)) + \sum_{\el=1}^{p_2} (\al_{\del_{\el}+1} -  \del_{\el})\\
&=\sig(A,B) + \sig(C,D).
\end{split}
\]
It then follows that $\sgn(A,B)\, \sgn(C,D) = \sgn(K,L)\, \sgn(M,N)$, and since \eqref{eq:IJ-shuf} implies
\[
\sig(I,J) = \sum_{i=1}^{p_1} ( \al_{\ga_{i} +1} + 1 - (i-1))= p_1 + \sig(K,L),
\]
we conclude that $\sgn(A,B)\, \sgn(C,D) =(-1)^{p_1}\sgn(I,J)\, \sgn(M,N)$.

\item \und{Verification of Eq.\ \ref{eq:propshuf2}}: We show that $s^Ns^K = s^Cs^B$. By \eqref{eq:epi-set}, it suffices to verify that 
$\sur(s^Ns^K) =  \sur(s^Cs^B)$. Eq.\ \ref{eq:lem-surj} implies that
\newcommand{\DO}{\sur(s^Ns^K)}
\newcommand{\DT}{\sur(s^Cs^B)}
\[
\begin{split}
\DO  &= K \cup \{ t \in L \st s^{N}s^{K}(t) = s^{N}(s^{K}(t) + 1) \},\\
\DT &= B \cup \{ t \in A \st s^{C}s^{B}(t) = s^{C}(s^{B}(t) + 1) \}.
\end{split}
\]
Let $\al_r \in K=\DO \cap K$, where $r = \ga_i +1$ for some $\ga_i \in C$.
Then Lemma \ref{lem:surj} implies that $s^{B}(\al_r) = r-1 = \ga_i$. Since $\ga_i$ in $C$, we have $s^{C}(\ga_{i})=s^{C}(\ga_{i}+1)$. Hence, $\al_r \in \DT \cap A$. Conversely, if $\al_\el \in  \DT \cap A$, then 
\[
s^{C}(\el-1) = s^{C}(s^{B}(\al_\el))= s^{C}(s^{B}(\al_{\el}) + 1)= s^{C}(\el).
\]
Hence, there exists $\ga_i \in C$ such that $\el -1 = \ga_i$, and so $\al_{\el}= \al_{\ga_i +1} \in K$. The same argument {\it mutatis mutandis} shows that $\DT \cap B = \DO \cap L$. Therefore, we conclude that $\DO=\DT$. 
\release{\DO}
\release{\DT}

\item \und{Verification of Eq.\ \ref{eq:propshuf3}}: Observe that $J^{> 0}_{-1} = L = B \cup E$, where $E:=
\{ \al_{\del_{i} +1}  \st i=1,\ldots, p_2\}$. By \eqref{eq:epi-set}, it suffices to show that $B \cup E = \sur(s^{D}s^{B})$. By Eq.\ \ref{eq:lem-surj},
\[
B \sse \sur(s^{D}s^{B}) = B \cup \{ t \in A \st s^{D}s^{B}(t)= s^{D}\bigl(s^{B}(t) + 1 \bigr)\}.
\]  
Hence, it remains to verify that $\sur(s^{D}s^{B}) \cap A \sse E$. If $\al_\el \in  \sur(s^{D}s^{B}) \cap A$, then $s^{B}(\al_\el) = \el-1$ by Lemma \ref{lem:surj}, and so $\el-1 \in D$. Therefore, there exists $\del_k$ such that $\del_k = \el -1$, which implies that $\al_\el = \al_{\del_k + 1} \in E$. 

\item\und{Verification of Eq.\ \ref{eq:propshuf4}}: We prove $s^A = s^Ms^K$ by showing $A= \sur(s^{M}s^{K})$. Since $K \sse A$, it remains to verify 
\begin{equation} \label{eq:propshufpf2}
L \cap A = \{ \lam \in L \st s^{M}s^{K}(\lam) = s^{M}(s^{K}(\lam) +1)\}.
\end{equation}
Elements in $L \cap A$ are of the form $\lam_{\mu_{\el} +1}$ for some $\mu_{\el} \in M$. As $s^{K}(\lam_{\mu_{\el} +1}) = \mu_\el$, by Lem.\ \ref{lem:surj}, we have $s^{M}s^{K}(\lam_{\mu_{\el} +1}) = s^{M}(\mu_{\el}) = s^{M}(\mu_{\el} +1 )$. Hence, 
$L \cap A$ is contained in the right-hand side of \eqref{eq:propshufpf2}. Conversely, let $\lam \in L = L \cap A \cup L \cap B$. If $\lam \in B$, then  $\lam = \lam_{\nu_{i} +1}$ for some $\nu_i \in N$, and so $s^{K}(\lam_{\nu_{i} +1}) = \nu_i$. By Eq. \ref{eq:lem-surj}, $s^{M}(\nu_{i} +1) =  s^{M}(\nu_{i}) + 1$. Therefore, $s^{M}s^{K}(\lam) \neq s^{M}(s^{K}(\lam) +1)$, and we conclude that
every element of the right-hand side of \eqref{eq:propshufpf2} is contained in $A$. 
\end{itemize}
\end{proof}

\begin{corollary}\label{cor:shuf}
Let $n=p_1 + p_2 + p_3 >0$. Let $(A,B) \in \Sh(p_1 + p_2,p_3)$, $(C,D) \in \Sh(p_1, p_2+p_3)$, and 
$\vtha \bigl((A,B),(C,D) \bigr) = \bigl( (I,J), (M,N) \bigr) \in \Sh(p_1,p_2+p_3+1)_{0 \in J} \times \Sh(p_2,p_3)$ as in  
in Prop.\ \ref{prop:shuf}.
\begin{enumerate}

\item If $0 \in A$ and $1 \in I$, then $s^{A^{> 0}_{-1}}=s^{M} s^{I^{> 1}_{-2}}$.

\item If $0 \in A$, $1 \in J$, then $0 \in M$ and $s^{A^{>0}_{-1}}=s^{M^{>0}_{-1}}s^{I_{-2}}$.

\item If $0 \in B$, then $1 \in J$, $0 \in N$, and $s^{A_{-1}}=s^{M_{-1}}s^{I_{-2}}$.

\end{enumerate}
\end{corollary}
\begin{proof}
Let $d^0 \maps [n-1] \to [n]$ denote the zeroth coface map in $\Delta$.
\begin{enumerate}[leftmargin=15pt]
\item By Eq.\ \ref{eq:propshuf4}, we have $s^{A}d^0= s^{M}s^{I_{-1}}d^0$. Since $0 \in A$ and $0 \in I_{-1}$, the desired equality holds. 

\item Let $L \sse [n]$ as in the proof of Prop.\ \ref{prop:shuf}.  Since $1 \in J$, we deduce from \eqref{eq:IJ-shuf} that $0 \in L \cap A$. Hence, $0 \in M$ by \eqref{eq:MN-shuf}. Again, by Eq.\ \ref{eq:propshuf4}, we have $s^{A}d^0= s^{M}s^{I_{-1}}d^0$. This establishes the desired equality, since $0 \notin I_{-1}$. 

\item Since $0 \in B$, \eqref{eq:MN-shuf} implies $0 \in L \cap B$, and therefore $0 \in N$. It then follows from \eqref{eq:IJ-shuf} that $1 \in J$. Let $s^0 \maps [p_3] \to [p_3-1]$ be the zeroth codegeneracy operator. Eq.\ \ref{eq:propshuf4} implies that $s^0s^{A}d^0= s^0s^{M}s^{I_{-1}}d^0$. The desired equality now follows, since $0 \notin A$, $0 \notin M$, and $1 \notin I$.    
\end{enumerate}
\end{proof}

\section{Technical Proposition \ref{prop:themap}} \label{sec:map-app}
We recall the statement:
\begin{proposition*} 
The collection \eqref{eq:tvph} of coalgebra morphisms $\{\tvp_n\}$ assemble into a morphism of almost simplicial coalgebras
\[
\tvp \maps \Kcoalgp_\bl\bigl( (\kk \semiop \Ng) \tensor S(\bs \Ng) \bigr) \to U^{+}_\bl\cW_\bl\cU(\gb)
\]
which has the following property:
\begin{equation} \label{eq:App-map-prop}
d^{\cW}_0 \tvp_m = \vph_{m-1}\dE_0 \jm \qquad \forall m \geq 0.
\end{equation}
\end{proposition*}
Note that it suffices to verify that the collection of maps $\{\tvp_n\}$ defines a morphism of almost simplicial vector spaces
\[
K^{+}_\bl \bigl( (\kk \semiop \Ng) \tensor S(\bs \Ng) \bigr) \xto{\tvp} U^{+}_\bl\cW_\bl\cU(\gb) \quad \text{in $\spVect$}
\]
satisfying Eq.\ \ref{eq:App-map-prop}.

\myspace
\subsection{Proof of Prop.\ \ref{prop:themap}}
We proceed by induction on ``almost $n$-skeleta'' in order to verify that Eq.\ \ref{eq:App-map-prop} holds for all $m \leq n$.

\subsubsection{The $n=1$ base case} \label{sec:prop-base-case}
In dimension $0$, by \eqref{eq:vph}, the morphism $\tvp_0$ is
\[
\tvp_0 \maps \g_0 \tensor \kk \dsum \kk \tensor \kk \to \cU(\g_0), \qquad \tvp_0(u \tensor 1 + 1\tensor 1) = u + 1.
\] 
In dimension $1$, the domain of $\tvp_1$ is
\[
K_{1} \bigl( (\kk \semiop \Ng) \tensor S(\bs \Ng) \bigr)=
N_1\gb \tensor \kk \, \oplus \, \kk \tensor \bs \g_0 \,  \oplus 
\g_0 \tensor \bs \g_0 \, \oplus \, s_0(\g_0 \tensor \kk). 
\]
We first consider elements of the subspace $N_1\gb \tensor \kk  \oplus  \kk \tensor \bs \g_0$. Using \eqref{eq:vph} and the definition \eqref{eq:EU-diff} of the differential $\dE_{0}=\del_E$, a direct calculation shows\footnote{For further details, see \eqref{eq:prim-map-1} and \eqref{eq:prim-map-2} in the proof of Cor.\ \ref{cor:prim-map}.}
\[
\tvp_1(u \tensor 1 + 1 \tensor \bs x) = u \tensor 1 - 1 \tensor x \in \cU(\g_1) \tensor \cU(\g_0) \quad \forall u \in N_1\gb, ~ \forall \bs x \in \bs \g_0. 
\]
In this particular case, the face maps \eqref{eq:W} of $\cW_\bl \cU(\gb)$ simplify to:
\begin{equation} \label{eq:dw-basecase}
\begin{aligned}[t]
d^{\cW}_0,d^{\cW}_1 \maps \cU(\g_1) \tensor \cU(\g_0) \to \cU(\g_0),
\end{aligned}
\qquad
\begin{aligned}[t]
d^{\cW}_0(z_1 \tensor z_0) &= (d^{\, \cU(\gb)}_0z_1) \ast z_0 \\ 
d^{\cW}_1(z_1 \tensor z_0) &= \eps(z_0)d^{\, \cU(\gb)}_1z_1.
\end{aligned}
\end{equation}
We first verify $d^{\cW}_0 \tvp_1 = \vph_0 \dE_0 \jm$. Evaluation of the left hand side yields:
\[
d^{\cW}_0 \tvp_1( u \tensor 1 + 1 \tensor \bs x) = 
d^{\, \cU(\gb)}_0u - d^{\, \cU(\gb)}_0(1) \ast x = d^{\g}_0u - x.
\]
Since $\dE_0 = \del_E$ on normalized simplices, \eqref{eq:EU-diff} implies:
\begin{equation} \label{eq:diff-basecase}
\begin{aligned}[t]
\dE_0(u \tensor 1) = d^\g_0u \tensor 1,
\end{aligned}
\qquad
\begin{aligned}[t]
\dE_0(1 \tensor \bs x ) &= 1 \tensor \bs d^\g_0x + (-1)^{\deg{1_\kk}\deg{x} + \deg{\bs x}}x \tensor 1 \\
&= - x \tensor 1.
\end{aligned}
\end{equation}
Evaluation of the right hand side recovers the desired expression:
\[
\vph \dE_0(u \tensor 1 + 1 \tensor \bs x) = d^\g_0u - x.
\] 
Next we verify that $d^{\cW}_1 \tvp_1 = \vph_0 \dE_1 \jm$, or equivalently
$d^{\cW}_1 \tvp_1(u \tensor 1 + 1 \tensor \bs x)=0$, since we are restricting to normalized simplices. Eq.\ \ref{eq:dw-basecase} implies that
\[
d^{\cW}_1 \tvp_1(u \tensor 1) = \eps(1)d^{\g}_1u =0,
\]
since $u \in N_1\gb$. Furthermore, 
\[
d^{\cW}_1 \tvp_1(1 \tensor \bs x) = \eps(x) \cdot 1 =0,
\]
since $\g_0 \sse \ker \eps$. 

Lastly, we check that the relevant equalities
hold for elements $u \tensor \bs x$ of the subspace $\g_0 \tensor \bs \g_0$. From \eqref{eq:vph} and the computation \eqref{eq:diff-basecase} above: 
\[
\begin{split}
\tvp_1(u \tensor \bs x) &= \rho_1(s^{\g}_0u)\tensor \vph_0\dE_0(1 \tensor \bs x)\\
&= s^{\g}_0u \tensor \vph_0(- x \tensor 1) = -s^{\g}_0u \tensor x.
\end{split}
\]   
It then follows from the above formulae \eqref{eq:dw-basecase} for the face maps $d^{\cW}_0,d^{\cW}_1$ that 
\[
d^{\cW}_0\tvp_1(u \tensor \bs x) = - u \ast x,
\] 
and $d^{\cW}_1(\tvp_1(u \tensor \bs x))=0$, as required. It remains to evaluate:
\[
\begin{split}
\vph_0 \dE_0(u \tensor \bs x) &= \vph_0(\del_E(u \tensor \bs x)) \\ 
&= \vph_0( \tha(u \tensor \bs x)) = -\vph_0(u \mdot x \tensor 1)\\
&= - u \mdot x.
\end{split}
\]
Note $\cU(\Ng)_0= \cU(N_0\g_\bl)= \cU(\g_0)$ as associative algebras, and so
$u \mdot x = u \ast x$. Therefore, we conclude that
\[
d^{\cW}_0\tvp_1(u \tensor \bs x)  =  \vph_0 \dE_0(u \tensor \bs x).
\]
This completes the proof of the base case.

\subsubsection{The inductive step} \label{sec:prop-induct-step}
Now let $n \geq 1$ and assume the maps $\tvp_1,\ldots, \tvp_{n}$ lift to a morphism of almost simplicial vector spaces
\begin{equation} \label{eq:almost-simp-induct}
\sk^+_n K_\bl \bigl( (\kk \semiop \Ng) \tensor S(\bs \Ng) \bigr) \to  
U^{+}_\bl\cW_\bl\cU(\gb)
\end{equation}
such that 
\begin{equation} \label{eq:map-prop-pf0}
d^{\cW}_0 \tvp_m = \vph_{m-1}\dE_0 \jm \qquad \forall m \leq n.
\end{equation}
First, we extend this morphism to $\sk^{+}_{n+1}$. Decompose $K^{+}_{n+1} \bigl( (\kk \semiop \Ng) \tensor S(\bs \Ng) \bigr)$ into the sum of linear subspaces consisting of normalized and degenerate simplices, respectively. By the almost simplicial analog of Lemma \ref{lem:sk}, it suffices to show
\begin{equation} \label{eq:map-prop-pf1}
d^{\cW}_{i \geq 1} \tvp_{n+1}(1 \tensor w)=0 \quad \forall w \in S(\bs \Ng)_{n+1},
\end{equation}  
and, for all $p+q =n+1$:
\[
d^{\cW}_{i \geq 1} \tvp_{n+1}(u \tensor w)=0 \quad \forall u \in N_{p}\gb, ~ \forall  w \in S(\bs \Ng)_{q}.
\]

We first verify Eq.\ \ref{eq:map-prop-pf1}. Let $w \in S(\bs \Ng)_{n+1}$ and $i \geq 1$. Then, by \eqref{eq:vph}:
\[
d^{\cW}_i \tvp_{n+1}(1 \tensor w) = d^{\cW}_i \si_{n}\vph_{n} \dE_{0}(1 \tensor w).
\]  
The identities \eqref{eq:extra1} encoding the compatibility between the extra degeneracy $\si_n$ and the face maps $d^{\cW}_{i}$ yield the equality
\[
d^{\cW}_i \si_{n}\vph_{n} \dE_{0}(1\tensor w) = \si_{n-1}d^{\cW}_{i-1} \vph_{n} \dE_{0}(1 \tensor w).
\]
From Eq.\ \ref{eq:EU-diff} for the differential $\dE_{0}=\del_E$, we see that
\[
\dE_{0}(1 \tensor w) \in \kk \tensor S(\bs \Ng) \oplus \Ng \tensor S(\bs \Ng).
\]
Therefore, for $i > 1$, the induction hypothesis along with the fact that $w$ is normalized, imply that  $d^{\cW}_{i-1}\vph_{n} \dE_{0}(1 \tensor w) = \vph_{n-1} \dE_{i-1}\dE_{0}(1 \tensor w) =0$. For $i=1$, from induction hypothesis -- specifically Eq.\ \ref{eq:map-prop-pf0} -- we obtain $d^{\cW}_{0}\vph_{n} \dE_{0}(1\tensor w) = \vph_{n-1}\dE_{0}\dE_{0}(1 \tensor w)=0$. Therefore, the desired equality
\begin{equation} \label{eq:map-prop-pf1-check}
d^{\cW}_{i \geq 1} \tvp_{n+1}(1 \tensor w)=0 \quad \forall w \in S(\bs \Ng)_{n+1},
\end{equation}  
is satisfied. Next, we claim:
\vspace{1.5em}
\begin{adjustwidth}{0em}{0em}
\begin{lemclaim}\label{lemclaim:BigMap1}
Let $p,q \geq 0$ such that $p+q = n+1$, with $u \in N_{p}\gb$ and $w \in S(\bs \Ng)_{q}$. Then
\begin{equation} \label{eq:lemclaim1}
d^{\cW}_{r \geq 1} \tvp_{n+1}(u \tensor w)=0.
\end{equation}
\end{lemclaim}
\end{adjustwidth}
\vspace{1.75em}
We postpone the proof of Claim \ref{lemclaim:BigMap1} to Sec.\ \ref{sec:lemclaim2} below. We then conclude that the maps $\tvp_1,\ldots, \tvp_{n}, \tvp_{n+1}$ induce a morphism of almost simplicial vector spaces
\[
\sk^+_{n+1} K_\bl \bigl( (\kk \semiop \Ng) \tensor S(\bs \Ng) \bigr) \to  
U^{+}_\bl\cW_\bl\cU(\gb). 
\]

It remains to show that 
\[
d^{\cW}_0 \tvp_{n+1} = \vph_{n}\dE_0 \jm.
\]
If $ z \in K_{n+1} \bigl( (\kk \semiop \Ng) \tensor S(\bs \Ng) \bigr)$ is of the form 
$z = s_{\el_{k}} \cdots s_{\el_{1}}(z')$, with $z'$ a normalized simplex of dimension $m < n+1$ and $0 \leq \el_1 < \el_2 < \cdots < \el_{k}$, and $\el_1=0$, then clearly $d^{\cW}_0 \tvp_{n+1}(z) = \vph_{n}\dE_0(z)$, as $\tvp_{\leq n+1}$ is a map of almost simplicial vector spaces.  Otherwise, if $\el_1> 0$, then by the inductive hypothesis
\[
d^{\cW}_0 \tvp_{n+1}(z) = s^{\cW}_{\el_{k}-1} \cdots s^{\cW}_{\el_{1}-1} d^{\cW}_0\ph_{m}(z') =  
s^{\cW}_{\el_{k}-1} \cdots s^{\cW}_{\el_{1}-1} \ph_{m-1}\dE_0\jm(z') = \ph_{n}\dE_0 \jm(z),
\] 
Observe the last equality follows since $\jm$ is an almost simplicial morphism. On the other hand, if $z$ is normalized, and of the form $z=u \tensor w$, then we claim:
\vspace{1.5em}
\begin{adjustwidth}{0em}{0em}
\begin{lemclaim}\label{lemclaim:BigMap2}
Let $p,q \geq 0$ such that $p+q = n+1$, with $u \in N_{p}\gb$ and $w \in S(\bs \Ng)_{q}$. Then
\begin{equation} \label{eq:lemclaim2}
d^{\cW}_0 \tvp_{n+1}(u \tensor w) =\vph_{n} \dE_0(u \tensor w).
\end{equation}
\end{lemclaim}
\end{adjustwidth}
\vspace{1.75em}
We postpone the proof of Claim \ref{lemclaim:BigMap2} to Sec.\ \ref{sec:lemclaim2} below. Finally, if $z=1 \tensor w$, then it follows from the definition \eqref{eq:vph} of $\tvp_{n+1}$ and the identities \eqref{eq:extra1} that
\[
d^{\cW}_{0}\tvp_{n+1}(1 \tensor w) =  d^{\cW}_{0} \si_{n}\vph_{n} \dE_{0}(1\tensor w)=\vph_{n} \dE_{0}(1 \tensor w).
\]
This completes the verification of the inductive step, and the proof of the proposition. \hfill \qed

\subsection{Proof of Claim \ref{lemclaim:BigMap1}} \label{sec:lemclaim1}
We use throughout the proof the notation established in \eqref{note:ord}.
Recall that in Sec.\ \ref{sec:prop-induct-step}, prior to asserting Claim \ref{lemclaim:BigMap1}, we assumed an inductive hypothesis, namely: for $n \geq 1$ \eqref{eq:almost-simp-induct} is a morphism of almost simplicial vector spaces, such that Eq.\ \ref{eq:map-prop-pf0} is satisfied for all $m \leq n$. 
We also verified, again prior to stating the claim, that Eq.\ \ref{eq:map-prop-pf1-check} was satisfied. 

Now let $p,q \geq 0$ such that $p+q = n+1$, with $u \in N_{p}\gb$ and $w \in S(\bs \Ng)_{q}$. We check by direct computation across multiple cases that  Eq.\ \ref{eq:lemclaim1} holds. If $q=0$, then $w=1$, and via Eq.\ \ref{eq:prim-map-1}: 
\[
\vph_{n+1}(u \tensor 1) = u \tensor 1^{\tensor n+1}.
\] 
We apply the $r \geq 1$ face maps \eqref{eq:W} for $\cW_\bl \cU(\gb)$ and obtain
\[
d^{\cW}_{r \geq 1}\vph_{n+1}(u \tensor 1) = d^\g_r u \tensor 1^{\tensor n}=0,
\]
since $u$ is normalized. 

From here on, we assume $q > 0$. The recursive identity \eqref{eq:W-recurse} for $d^{\cW}_r$ along with the definition of $\tvp_{n+1}$ gives the equality
\begin{equation} \label{eq:B1-pf1}
d^{\cW}_{r}(u \tensor w) = \hspace{-2ex} \sum_{(I,J) \in \Sh(p,q)} \hspace{-3ex}
\sgn(I,J) \, d_{r}s_{J} u \tensor  d^{\cW}_{r-1}\vph_{n}\dE_{0}s_I(1 \tensor w)
\end{equation}   
From the definition of the differential $\dE_0 = \del_{E}$, we deduce that 
\[
\dE_{0}s_I(1 \tensor w) \in K_{n} \bigl( (\kk \semiop \Ng) \tensor S(\bs \Ng) \bigr),
\]
and therefore, the inductive hypotheses \eqref{eq:almost-simp-induct} and 
\eqref{eq:map-prop-pf0} imply that 
\[
d^{\cW}_{r-1}\vph_{n}\dE_{0}s_I(1 \tensor w) = \vph_{n-1} \dE_{r-1}\dE_{0}s_I(1 \tensor w).
\]
Therefore the right hand side of Eq.\ \ref{eq:B1-pf1} becomes
\begin{equation} \label{eq:B1-pf2}
\begin{split}
\hspace{-2ex} \sum_{(I,J) \in \Sh(p,q)} \hspace{-3ex}
\sgn(I,J) \, &d_{r}s_{J} u \tensor  \vph_{n-1} \dE_{0}\dE_{r}s_I(1 \tensor w)  = \\
&= \sum_{(I,J) \in \Sh(p,q)} \hspace{-3ex}
\sgn(I,J) \, s_{J^{> r}_{-1}} \, d_{r}s_{J^{ \leq r}} u \, \tensor
\vph_{n-1}\dE_{0} s_{I^{>r}_{-1}}\,  \dE_{r}s_{I^{ \leq r}}(1 \tensor w).
\end{split}
\end{equation}
Consider the summand of \eqref{eq:B1-pf2} corresponding to a fixed shuffle $(I,J) \in \Sh(p,q)$. Suppose both  $r-1\in I$ and $r \in I$. Then 
$J^{ \leq r} = J^{ \leq r-2}$, and therefore 
\[
d_{r}s_{J^{ \leq r}} u = s_{J^{ \leq r-2}} d_{r- k}  u,
\]
where $k:= \abs{J^{ \leq r-2}} \leq r-1$. Hence, since $u$ is normalized, and $r-k \geq 1$, we conclude
\[
d_{r}s_{J^{ \leq r}} u =0.
\]
If instead both  $r-1\in J$ and $r \in J$, then since $1 \tensor w$ is also normalized, the same argument implies that  
\[
\dE_{r}s_{I^{ \leq r}}(1 \tensor w) =0.
\]
On the other hand, suppose $r-1 \in I$ and $r \in J$. Let $(I',J') \in \Sh(p,q)$ be the unique shuffle such that
\[
I' = I^{< r-1} \cup \{r\} \cup I^{ > r-1}.
\] 
Then $\sgn(I,J)= -\sgn(I',J')$; hence
\[
\begin{split}
\sgn(I,J) \, s_{J^{> r}_{-1}} \, d_{r}s_{J^{ \leq r}} u \, &\tensor
\vph_{n-1}\dE_{0} s_{I^{>r}_{-1}}\,  \dE_{r}s_{I^{ \leq r}}(1 \tensor w)\\
&+\sgn(I',J') \, s_{J'^{> r}_{-1}} \, d_{r}s_{J'^{ \leq r}} u \, \tensor
\vph_{n-1}\dE_{0} s_{I'^{>r}_{-1}}\,  \dE_{r}s_{I'^{ \leq r}}(1 \tensor w) \\
&=0.
\end{split}
\]
Therefore we conclude that the sum \eqref{eq:B1-pf2} is zero, and hence
$d^{\cW}_{r \geq 1}(u \tensor w) =0$. This completes the proof of the claim \hfill \qed

\subsection{Proof of Claim \ref{lemclaim:BigMap2}} \label{sec:lemclaim2}
Let $p,q \geq 0$ such that $p+q = n+1$, with $u \in N_{p}\gb$ and $w \in S(\bs \Ng)_{q}$. As before, we verify casewise by direct computation that Eq.\ \ref{eq:lemclaim2} holds. If $q=0$, then $w=1$, and via Eq.\ \ref{eq:prim-map-1}, we obtain
\[
\vph_{n+1}(u \tensor 1) = u \tensor 1^{\tensor n+1}
\] 
We apply the face map $d^{\cW}_0$ for $\cW_\bl \cU(\gb)$:
\[
\begin{split}
d^{\cW}_0\vph_{n+1}(u \tensor 1) &= d^\g_0 u \tensor 1^{\tensor n}\\
\end{split}
\]
Since $\dE_0 = \del_E$ on normalized simplices, \eqref{eq:EU-diff} implies that $\dE_0(u \tensor 1) = d^\g_0u \tensor 1$. Hence, by Eq.\ \ref{eq:prim-map-1} again, we conclude
\[
d^{\cW}_0\vph_{n+1}(u \tensor 1) = \vph_n\dE_0(u \tensor 1).
\]
From here on, we assume $q > 0$. We begin by expanding out the left-hand side of the desired equality \eqref{eq:lemclaim2}
\newcommand{\LI}{\mathscr{L}_{0\in I}}
\newcommand{\LJ}{\mathscr{L}_{0\in J}}
\newcommand{\LJA}{\mathscr{L}^{(1)}_{0\in J}}
\newcommand{\LJB}{\mathscr{L}^{(2)}_{0\in J}}
\newcommand{\RA}{\mathscr{R}_{1}}
\newcommand{\RB}{\mathscr{R}_{2}}
\newcommand{\RC}{\mathscr{R}_{3}}

\[
\begin{split}
d^{\cW}_0 \tvp_{n+1}(u \tensor w) &= 
\hspace{-2ex} \sum_{(I,J) \in \Sh(p,q)} \hspace{-3ex}
\sgn(I,J) \, d^{\cW}_0 \bigl(s_{J} u \tensor  \vph_{n}\dE_{0}s_I(1 \tensor w) \bigr)\\
&= \LI + \LJ.
\end{split}
\]
where
\[
\begin{split}
\LI &:= \hspace{-2ex} \sum_{\substack{(I,J) \in \Sh(p,q)\\ 0 \in I } } \hspace{-3ex}
\sgn(I,J) \, d^{\cW}_0 \bigl(s_{J} u \tensor  \vph_{n}s_{I^{>0}_{-1}}(1 \tensor w) \bigr)\\
\LJ&:=  \hspace{-2ex} \sum_{\substack{(I,J) \in \Sh(p,q)\\ 0 \in J } } \hspace{-3ex}
\sgn(I,J) \, d^{\cW}_0 \bigl(s_{J} u \tensor  \vph_{n}\dE_{0}s_I(1 \tensor w) \bigr)
\end{split}
\]
It will be useful to expand the term $\LJ$ one step further. Fix a shuffle $(I,J)$ with $0 \in J$. Recall that on normalized simplices, $\dE_0=\del_E$, where $\del_E$ is the differential \eqref{eq:EU-diff}. Using the notation introduced in \eqref{eq:EU-diff}, we have the equality
\[
\dE_0 s_I(1 \tensor w) = s_{I_{-1}} (1 \tensor \del_{\CE}w + \tha(1 \tensor w) ).
\]
Therefore, we write
\[
\LJ = \LJA + \LJB, 
\]
where
\begin{equation} \label{eq:B2-LHS2}
\begin{split}
\LJA&:= \hspace{-2ex} \sum_{\substack{(I,J) \in \Sh(p,q)\\ 0 \in J } } \hspace{-3ex}
\sgn(I,J) \, d^{\cW}_0 \bigl(s_{J} u \tensor  s^{\cW}_{I_{-1}}\vph_{q-1}(1 \tensor \del_{\CE}w) \bigr), \\
\LJB&:= \hspace{-2ex} \sum_{\substack{(I,J) \in \Sh(p,q)\\ 0 \in J } } \hspace{-3ex}
\sgn(I,J) \, d^{\cW}_0 \bigl(s_{J} u \tensor  s^{\cW}_{I_{-1}}\vph_{q-1}\tha(1 \tensor w) \bigr).
\end{split}
\end{equation}
Next, we expand the right-hand side of \eqref{eq:lemclaim2} as
\[
\vph_{n} \dE_0(u \tensor w) = \RA + \RB + \RC
\]
where
\begin{equation} \label{eq:B2-RHS1}
\begin{split}
\RA&:= (-1)^{p}\vph_n(u \tensor \del_{\CE}w)\\
&=\hspace{-2ex} \sum_{(I',J') \in \Sh(p,q-1)} \hspace{-3ex}
(-1)^{p}\, \sgn(I',J') \, s_{J'} u \tensor \vph_{n-1} \dE_0 s_{I'}\del_{\CE}(1 \tensor w)\\
\RB&:= \vph_n(d_0u \tensor w)\\
&=\hspace{-2ex} \sum_{(I',J') \in \Sh(p-1,q)} \hspace{-3ex}
\sgn(I',J') \, s_{J'} d_0u \tensor \vph_{n-1} \dE_0 s_{I'}(1 \tensor w)\\
\RC&:=(-1)^{p} \vph_n\tha(u \tensor w) \\
&= (-1)^{p}\sum_{i=1}^k(-1)^{\eps_i} \vph_{n}(u \mdot x_{i} \tensor w_{i}).
\end{split}
\end{equation}
In the last line above, if $w=\bs x_1 \bs x_2 \cdots \bs x_k$, then $w_i=\bs x_1 \bs x_2 \cdots \bs x_{i-1} \cdot 1_\kk \cdot \bs x_{i+1} \cdots \bs x_k$, and $\eps_i= \deg{w_i}\deg{x_i} + n_i$ with $n_i$ defined as in \eqref{eq:codiff-sign}. 
\subsubsection{The equality $\LI = \RB$} \label{sec:LI=RB}
Applying the recursive definition \eqref{eq:vph} to $\vph_{n-1}$ yields
\[
\begin{split}
\vph_{n}s_{I^{>0}_{-1}}(1 \tensor w) &= s^{\cW}_{I^{>0}_{-1}} \vph_{q}(1 \tensor w)\\
&=s^{\cW}_{I^{>0}_{-1}} \si_{q-1}\vph_{q-1}\dE_0(1 \tensor w)\\
&=\begin{cases}
\si_{n-1} s^{\cW}_{I^{>1}_{-2}} \si_{q-1} \vph_{q-1} \dE_0(1 \tensor w), & \text{if $1 \in I$}\\
\si_{n-1} s^{\cW}_{I^{>1}_{-2}} \vph_{q-1} \dE_0(1 \tensor w), & \text{if $1 \notin I$}\\
\end{cases}
\end{split}
\]
where the last equality above follows from the identities \eqref{eq:extra2} which exhibit the compatibility of the degeneracies $s^{\cW}_\bl$ with the extra degeneracy operator $\si_\bl$. Using this, along with the fact that $\vph_{q}(1 \tensor w)=\si_{q-1}\vph_{q-1}\dE_0(1 \tensor w)$, we rewrite  $\LI$ as
\begin{equation} \label{eq:LIR2-1}
\begin{split}
\LI &= \hspace{-2ex} \sum_{\substack{(I,J) \in \Sh(p,q)\\ 0,1 \in I } } \hspace{-3ex}
\sgn(I,J) \, d^{\cW}_0 \bigl(s_{J} u \tensor  \si_{n-1} s^{\cW}_{I^{>1}_{-2}} \vph_{q}(1 \tensor w) \bigr)\\
& \quad + \hspace{-2ex} \sum_{\substack{(I,J) \in \Sh(p,q)\\ 0 \in I,\, 1 \notin I} } \hspace{-3ex}
\sgn(I,J) \, d^{\cW}_0 \bigl(s_{J} u \tensor  \si_{n-1} s^{\cW}_{I^{>1}_{-2}} \vph_{q-1} \dE_0(1 \tensor w) \bigr).\\
\end{split}
\end{equation}
On the other hand, given $(I',J') \in \Sh(p-1,q)$, the corresponding summand
in the defining expression \eqref{eq:B2-RHS1} for $\RB$ can be rewritten as
\[
\begin{split}
s_{J'} d_0u \tensor \vph_{n-1} \dE_0 s_{I'}(1 \tensor w)
&= d_0 s_{J'_{1}} u \ast 1 \tensor \vph_{n-1} \dE_0 s_{I'}(1 \tensor w)\\
&=d^{\cW}_0 \bigl(s_{J'_{1}} u \tensor \si_{n-1}\vph_{n-1} \dE_0 s_{I'}(1 \tensor w) \bigr).
\end{split}
\] 
Therefore:
\begin{equation} \label{eq:LIR2-2}
\begin{split}
\RB &=\hspace{-2ex} \sum_{\substack{(I',J') \in \Sh(p-1,q)\\ 0 \in I'}} \hspace{-3ex}
\sgn(I',J') \, d^{\cW}_0 \bigl(s_{J'_{1}} u \tensor \si_{n-1}s^{\cW}_{I^{\prime >0}_{-1}}\vph_{q}(1 \tensor w) \bigr)\\
& \quad +  
\hspace{-2ex} \sum_{\substack{(I',J') \in \Sh(p-1,q)\\ 0 \notin I'}} \hspace{-3ex}
\sgn(I',J') \, d^{\cW}_0 \bigl(s_{J'_{1}} u \tensor \si_{n-1}s^{\cW}_{I^{\prime }_{-1}}\vph_{q-1}\dE_0(1 \tensor w) \bigr).\\
\end{split}
\end{equation}
The comparison between Eqs.\ \ref{eq:LIR2-1} and \ref{eq:LIR2-2} is now clear. Indeed, the subset of shuffles $(I,J) \in \Sh(p,q)$ for which $0 \in I$ is in bijection with shuffles
$(I',J') \in \Sh(p-1,q)$ via the assignment $I':=I^{>0}_{-1}$, $J':=[n-1]\setminus I'$. It then follows that $J=J'_{1}$. Furthermore, if $1 \in I$, then $I^{\prime >0}_{-1} = I^{>1}_{-2}$; otherwise, $I'_{-1} = I^{>1}_{-2}$. Finally, we verify that this assignment preserves signs. Given $(I,J)$, let $i_1< \cdots < i_{p}$ and  
$i'_1< \cdots < i'_{p-1}$ denote the elements of $I$ and $I'$, respectively. 
We compute via Eq.\ \ref{eq:sgn-sh}:
\[
\begin{split}
\sig(I',J') = \sum_{k=1}^{p-1} ( i'_k - (k-1)) 
= \sum_{k=1}^{p-1} ( i_{k+1}-1  - (k-1)) 
=\sum_{k=2}^{p} ( i_{k}  - (k-1)).
\end{split}
\]
Therefore, since $i_1=0$, we obtain the equality $\sgn(I,J)= \sgn(I',J')$, and we conclude that 
\[
\LI =\RB.
\]
\subsubsection{The equality $\LJA = \RA$} \label{sec:LJA=RA}
First, note that if $q=1$, then $w \in S(\bs \Ng)_1 = \bs(\g_0)$. Hence, $\del_{\CE}w =0$ and therefore it follows from Eqs.\ \ref{eq:B2-LHS2}
and \ref{eq:B2-RHS1} that $\LJA = 0 = \RA$.
For the $q > 1$ case, we apply the exact same approach that we used in the previous section to show $\LI =\RB$. We rewrite the defining equation \eqref{eq:B2-LHS2} for $\LJA$ as two sums:
\begin{equation} \label{eq:LJARA-1}
\begin{split}
\LJA &= \hspace{-2ex} \sum_{\substack{(I,J) \in \Sh(p,q)\\ 0 \in J, \, 1 \in I } } \hspace{-3ex}
\sgn(I,J) \, d^{\cW}_0 \bigl(s_{J} u \tensor  \si_{n-1} s^{\cW}_{I^{>1}_{-2}} \vph_{q-1}(1 \tensor \del_{\CE}w) \bigr) \\ 
& \quad + \hspace{-2ex} \sum_{\substack{(I,J) \in \Sh(p,q)\\ 0 \in J, \, 1 \notin I } } \hspace{-3ex}
\sgn(I,J) \, d^{\cW}_0 \bigl(s_{J} u \tensor  
\si_{n-1} s^{\cW}_{I_{-2}} \vph_{q-2}\dE_0(1 \tensor \del_{\CE}w) \bigr).
\end{split}
\end{equation}
Terms of the form $s_{J'} u \tensor \vph_{n-1} \dE_0 s_{I'}\del_{\CE}(1 \tensor w)$
appearing in the defining equation \eqref{eq:B2-RHS1} for $\RA$ can be rewritten as
\[
\begin{split}
s_{J'} u \tensor \vph_{n-1} \dE_0 s_{I'}\del_{\CE}(1 \tensor w) &= 
d_0 s_{J'_{+1}} s_0 u \tensor \vph_{n-1} \dE_0 s_{I'}\del_{\CE}(1 \tensor w) \\
&=d^{\cW}_0 \bigl(s_{J'_{+1}} s_0 u \tensor \si_{n-1}\vph_{n-1} \dE_0 s_{I'}\del_{\CE}(1 \tensor w) \bigr).
\end{split}
\]
Applying this, along with the definition \eqref{eq:vph} for $\vph_{n-1}$ and
the identities \eqref{eq:extra1} involving the extra degeneracy, yields
\begin{equation} \label{eq:LJARA-2}
\begin{split}
\RA &=  
\hspace{-2ex} \sum_{\substack{(I',J') \in \Sh(p,q-1)\\ 0 \in I'}} \hspace{-3ex}
(-1)^{p}\, \sgn(I',J')\, d^{\cW}_0 \bigl( s_{J'_{+1}} s_0 u \tensor \si_{n-1}
s^{\cW}_{I^{\prime >0}_{-1}} \vph_{q-1} \del_{\CE}(1 \tensor w) \bigr)\\
& \quad +
\sum_{\substack{(I',J') \in \Sh(p,q-1)\\ 0 \notin I'}} \hspace{-3ex}
(-1)^{p}\, \sgn(I',J')\, d^{\cW}_0 \bigl(  s_{J'_{+1}} s_0 u \tensor \si_{n-1}
s^{\cW}_{I^{\prime}_{-1}} \vph_{q-2} \del_{\CE}(1 \tensor w) \bigr).
\end{split}
\end{equation}
We compare the terms in Eqs.\ \ref{eq:LJARA-1} and \ref{eq:LJARA-2}. Shuffles $(I',J') \in \Sh(p,q-1)$ are in bijection with the subset of shuffles $(I,J) \in \Sh(p,q)$ satisfying $0 \in J$ via the assignment $I:=I'_{+1}$, $J:=[n]\setminus I$.
Clearly, this implies that $J=\{0\} \cup J'_{+1}$. If $1 \in I$, then $0 \in I'$, and hence $I^{>1}_{-2} = I^{\prime > 0}_{-1}$. If not, then $I^{\prime}_{-1}$ is well-defined and equal to $I_{-2}$. It remains to check signs. Given $(I',J')$, let $i'_1< \cdots < i'_{p}$ and  $i_1< \cdots < i_{p}$ denote the elements of $I'$ and $I$, respectively. We apply Eq.\ \ref{eq:sgn-sh} and compute
\[
\sig(I,J) = \sum_{k=1}^p (i_k - (k-1) ) = \sum_{k=1}^p (i'_k + 1 - (k-1) ) = \sig(I',J') +p.
\]
Hence, $\sgn(I,J)=(-1)^{p} \sgn(I',J')$, and we conclude that
\[
\LJA = \RA.
\]
\subsubsection{The equality $\LJB = \RC$} \label{sec:LJB=RC}
To begin, consider terms appearing in the defining equation \eqref{eq:B2-LHS2} of $\LJB$ of the form $s_{J} u \tensor  s^{\cW}_{I_{-1}}\vph_{q-1}\tha(1 \tensor w)$. From the definition \eqref{eq:EU-diff} for $\tha$, we obtain
\[
\tha(1 \tensor w) = \sum_{i=1}^k(-1)^{\eps_i} x_i \tensor w_i
\] 
with $\eps_i$, $x_i$ and $w_i$ defined as in Eq.\ \ref{eq:B2-RHS1}. Applying the definition of $\vph_{q-1}$ yields
\[
\begin{split}
s_{J} u \tensor s^{\cW}_{I_{-1}}\vph_{q-1}&\tha(1 \tensor w)  \\
 &=\sum_{i=1}^k(-1)^{\eps_i} \hspace{-5ex} \sum_{(M,N) \in \Sh(r_i,q_i)} \hspace{-4ex} \sgn(M,N) s_{J} u \tensor s^{\cW}_{I_{-1}} \bigl( s_N x_i \tensor \vph_{q-2}\dE_0 s_{M}(1 \tensor w_i) \bigr).\\
\end{split}
\]  
where $r_i = \abs{x_i}$, and $q_i = \abs{w_i}$. To keep the notation under control, define
\begin{equation} \label{eq:f-def}
f_{i}(I,J,M,N):=s_{J} u \tensor s^{\cW}_{I_{-1}} \bigl( s_N x_i \tensor \vph_{q-2}\dE_0 s_{M}(1 \tensor w_i) \bigr).
\end{equation}
so that
\begin{equation} \label{eq:LJB-new}
\LJB = \sum_{i=1}^k(-1)^{\eps_i} \hspace{-3ex}  \sum_{\substack{(I,J) \in \Sh(p,q)\\ 0 \in J}}\hspace{-3ex} \sgn(I,J)\hspace{-4ex}  \sum_{(M,N) \in \Sh(r_i,q_i)} \hspace{-4ex} \sgn(M,N)\,  d^{\cW}_{0} f_{i}(I,J,M,N)
\end{equation}
Let us turn to the defining expression \eqref{eq:B2-RHS1} for $\RC$. We consider terms of the form $\vph_n(u \mdot x_i \tensor w_{i})$. Applying the definition of $\vph_n$ yields
\[
\vph_n(u \mdot x_i \tensor w_{i}) = 
\hspace{-3ex}\sum_{(A,B) \in \Sh(p+r_i,q_i)}\hspace{-3ex} \sgn(A,B)\,
s^{\cU(\gb)}_B \ro(u \mdot x_i) \tensor \vph_{n-1}\dE_0 s_A(1\tensor w_i).
\]
Lemma \ref{lem:rho} implies that
\[
\begin{split}
s^{\cU(\gb)}_B\ro(u \mdot x_i) &= \hspace{-3ex} \sum_{(C,D) \in \Sh(p,r_i)}\hspace{-3ex} \sgn(C,D)\,
s^{\cU(\gb)}_B \bigl( s_D u \ast s_C x_i \bigr) \\
&= \hspace{-3ex} \sum_{(C,D) \in \Sh(p,r_i)}\hspace{-3ex} \sgn(C,D)\,
 s_Bs_D u \ast s_Bs_C x_i.
\end{split}
\] 
Using this, along with the definition of $d^{\cW}_0$, we obtain 
\[
\begin{split}
s^{\cU(\gb)}_B \ro(u \mdot x_i) \tensor &\vph_{n-1}\dE_0 s_A(1\tensor w_i)\\
&=\hspace{-3ex} \sum_{(C,D) \in \Sh(p,r_i)}\hspace{-3ex} \sgn(C,D)\,
d^{\cW}_0 \bigl( s_{B_{+1}}s_{D_{+1}}s_0 u \tensor s_Bs_C x_i
\tensor \vph_{n-1}\dE_0 s_A(1\tensor w_i) \bigr).
\end{split}
\]
Therefore, we can rewrite $\RC$ as
\begin{equation} \label{eq:RC-new}
\begin{split}
\RC= (-1)^p\sum_{i=1}^k(-1)^{\eps_i} \hspace{-3ex}  \sum_{(A,B) \in \Sh(p+r_i,q_i)}\hspace{-3ex} \sgn(A,B)\hspace{-4ex}  \sum_{(C,D) \in \Sh(p,r_i)} \hspace{-4ex} \sgn(C,D)\, d^{\cW}_0 g_{i}(A,B,C,D),
\end{split}
\end{equation}  
where
\begin{equation} \label{eq:g-def}
g_{i}(A,B,C,D):= s_{B_{+1}}s_{D_{+1}}s_0 u \tensor s_Bs_C x_i
\tensor \vph_{n-1}\dE_0 s_A(1\tensor w_i).
\end{equation}

We now identify, term-by-term, \eqref{eq:RC-new} with \eqref{eq:LJB-new}. We first apply Prop.\ \ref{prop:shuf} with $p_1 = p$, $p_2=r_i$, and $p_3 =q_i$. This gives a bijection 
\begin{equation} \label{eq:bij}
\begin{split}
\Sh(p+r_i,q_i)  \times \Sh(p,r_i) &\xto{\cong} \Sh(p,q)_{0 \in J} \times \Sh(r_i,q_i)\\
\bigl((A,B),(C,D) \bigr) &\mapsto \bigl( (I,J), (M,N) \bigr)
\end{split}
\end{equation}
such that $\sgn(A,B)\sgn(C,D) = (-1)^{p}\sgn(I,J) \sgn(M,N)$. Therefore, the signs and sets of shuffles indexing the summations in \eqref{eq:RC-new} are equal to those of \eqref{eq:LJB-new}. Fix $i=1,\ldots,k$ and shuffles $(A,B) \in \Sh(p+r_i,q_i)$, and $(C,D) \in \Sh(p,r_i +q_i)$. Let $(I,J)$ and $(M,N)$ be the corresponding shuffles via \eqref{eq:bij}. To complete the proof, it suffices to verify that $g_{i}(A,B,C,D) = f_{i}(I,J,M,N)$. 

First, we observe that, for any ordered subset $T\sse [n]$, the defining formulae for the degeneracies \eqref{eq:W} and their compatibility with the extra degeneracy \eqref{eq:extra2} yield the expression
\[
s^{\cW}_{T} = 
\begin{cases}
s^\g_T \tensor s^{\cW}_{T^{>0}_{-1}} \si_{n-1}, & 0 \in T,\\
s^\g_T \tensor s^{\cW}_{T_{-1}}, & 0 \notin T.
\end{cases}
\]
Applying this to \eqref{eq:f-def} gives
\begin{equation} \label{eq:f-def-new}
f_{i}(I,J,M,N)=
\begin{cases}
s_{J} u \tensor  s_{I_{-1}}s_N x_i \tensor s^{\cW}_{I^{>1}_{-2}}\si_{q-2}\vph_{q-2}\dE_0 s_{M}(1 \tensor w_i),  & 1 \in I \\
s_{J} u \tensor  s_{I_{-1}}s_N x_i \tensor s^{\cW}_{I_{-2}}\vph_{q-2} \dE_0 s_{M}(1 \tensor w_i),  & 1 \notin I.
\end{cases}
\end{equation}
Comparing the above to the defining expression \eqref{eq:g-def} for $g_i(A,B,C,D)$, we see that there are three factors in the tensor product to consider: one involving either $u$, $x_i$, or $w_i$.

The factors involving $x_i$ are equal. Indeed, the identity \eqref{eq:propshuf2} from Prop.\ \ref{prop:shuf} immediately implies that
\[
s_Bs_Cx_i = s_{I_{-1}}s_Nx_i.
\]
Turning to ``$u$ factors'', the identity \eqref{eq:propshuf3} from Prop.\ \ref{prop:shuf} implies that $s_{J^{>0}_{-1}} = s_B s_D$. Therefore, since $0 \in J$, we obtain the desired equality:
\[
s_Ju = s_0 s_{J^{>0}_{-1}}u = s_0s_B s_Du = s_{B_{+1}} s_{D_{+1}}s_0u.  
\]   

Let $f^{w_i}$ and $g^{w_i}$ denote the factors involving $w_i$ in $f_{i}(I,J,M,N)$ and $g_i(A,B,C,D)$, respectively. We consider cases. First, we have
\begin{equation} \label{eq:g-cases}
g^{w_i}:=\vph_{n-1}\dE_0 s_A(1\tensor w_i)=
\begin{cases}
s^{\cW}_{A^{>0}_{-1}} \vph_{q_i}(1 \tensor w_i), & 0 \in A\\
s^{\cW}_{A_{-1}} \vph_{q_{i}-1}\dE_0(1 \tensor w_i), & 0 \notin A
\end{cases}
\end{equation}
\begin{itemize}[leftmargin=15pt]
\item {\bf Case 1} ($0 \in A$, $1 \in I$, $0 \in M$):

\mind By Eq.\ \ref{eq:f-def-new}, we have
\begin{equation} \label{eq:fwork1}
\begin{split}
f^{w_i}=s^{\cW}_{I^{>1}_{-2}}\si_{q-2}\vph_{q-2}\dE_0 s_{M}(1 \tensor w_i) &=
s^{\cW}_{I^{>1}_{-2}}\si_{q-2}s^{\cW}_{M^{>0}_{-1}}\vph_{q_i}(1 \tensor w_i)\\
&= s^{\cW}_{I^{>1}_{-2}}s^{\cW}_{M^{>0}}\si_{q_i}\vph_{q_i}(1 \tensor w_i).
\end{split}
\end{equation}
Note that the last equality above was obtained via the identity \eqref{eq:extra4}.
The definition of $\vph_{q_i}$ gives  $\vph_{q_i}(1 \tensor w_i) = \si_{q_i - 1} \vph_{q_{i} -1} \dE_0(1 \tensor w_i)$. Combining this with \eqref{eq:fwork1} and the identity \eqref{eq:extra5} yields
\[
\begin{split}
f^{w_i} &= s^{\cW}_{I^{>1}_{-2}}s^{\cW}_{M^{>0}}s_0\si_{q_i - 1} \vph_{q_{i} -1} \dE_0(1 \tensor w_i) \\
&= s^{\cW}_{I^{>1}_{-2}}s^{\cW}_{M}\vph_{q_{i}}(1 \tensor w_i) \\
\end{split}
\]  
Statement 1 of Cor.\ \ref{cor:shuf} implies that $s^{\cW}_{A^{>0}_{-1}} = s^{\cW}_{I^{>1}_{-2}} s^{\cW}_M$. Hence, by Eq.\ \ref{eq:g-cases}, $g^{w_{i}} = f^{w_{i}}$.
\item {\bf Case 2} ($0 \in A$, $1 \in I$, $0 \in N$):

\mind Eq.\ \ref{eq:f-def-new} and the identity \eqref{eq:extra4} imply that
\[
\begin{split}
f^{w_{i}}&= s^{\cW}_{I^{>1}_{-2}}\si_{q-2}s^{\cW}_{M_{-1}}\vph_{q_{i}-1}\dE_0(1 \tensor w_i)\\
&= s^{\cW}_{I^{>1}_{-2}}s^{\cW}_{M}\si_{q_{i}-1}\vph_{q_{i}-1}\dE_0(1 \tensor w_i)\\
&= s^{\cW}_{I^{>1}_{-2}}s^{\cW}_{M}\vph_{q_i}(1 \tensor w_i). 
\end{split}
\]
It then follows, as in the previous case, by statement 1 of Cor.\ \ref{cor:shuf} and Eq.\ \ref{eq:g-cases} that $g^{w_{i}} = f^{w_{i}}$. 

\item {\bf Case 3} ($0 \in A$, $1 \in J$):

\mind Statement 2 of Cor.\ \ref{cor:shuf} implies that $0 \in M$ and 
$s^{\cW}_{A^{>0}_{-1}}=s^{\cW}_{I_{-2}} s^{\cW}_{M^{>0}_{-1}}$. Therefore, since $1 \notin I$, Eq.\ \ref{eq:f-def-new} yields
\[
f^{w_i} = s^{\cW}_{I_{-2}}\vph_{q-2} \dE_0 s_{M}(1 \tensor w_i)= s^{\cW}_{I_{-2}}s_{M^{>0}_{-1}}\vph_{q_i}(1 \tensor w_i).
\]
Hence, we deduce from Eq.\ \ref{eq:g-cases} that $g^{w_{i}} = f^{w_{i}}$.

\item {\bf Case 4} ($0 \in B$):

\mind Eq.\ \ref{eq:g-cases} implies that 
\[
g^{w_{i}} = s^{\cW}_{A_{-1}} \vph_{q_{i}-1}\dE_0(1 \tensor w_i).
\] 

From statement 3 of Cor.\ \ref{cor:shuf}, we deduce that $1 \in J$, $0 \notin M$, and $s^{\cW}_{A_{-1}}=s^{\cW}_{I_{-2}}s^{\cW}_{M_{-1}}$. Combining this with  Eq.\ \ref{eq:f-def-new}, we obtain
\[
\begin{split}
f^{w_{i}} &= s^{\cW}_{I_{-2}}\vph_{q-2} \dE_0 s_{M}(1 \tensor w_i)\\
& = s^{\cW}_{I_{-2}}s^{\cW}_{M_{-1}} \vph_{q_{i}-1} \dE_0 (1 \tensor w_i)\\
&=s^{\cW}_{A_{-1}}\vph_{q_{i}-1} \dE_0 (1 \tensor w_i)\\
&=g^{w_i}.
\end{split}
\]  
\end{itemize}

We have exhausted all possible cases, and so we conclude that $\LJB = \RC$.

This completes the proof of Claim \ref{lemclaim:BigMap2}. \hfill \qed
\bibliographystyle{amsplain}
\bibliography{lie2}


\end{document}